\documentclass[a4paper,reqno]{amsart}
\usepackage{amsfonts}
\usepackage{stmaryrd}

\usepackage[left=3cm,right=3cm,top=3cm,bottom=3cm,headsep=0.7cm]{geometry}
\usepackage[T1]{fontenc} 
\usepackage[utf8]{inputenc}
\usepackage[french, italian, english]{babel}
\usepackage{amsmath,amssymb,amsthm,mathrsfs,esint}
\usepackage{mathtools}

\newcommand{\R}{{\mathbb R}}

\newcommand{\C}{{\mathbf A}}
\newcommand{\Q}{{\mathcal Q}}
\newcommand{\HH}{{\mathcal H}}
\newcommand{\E}{{\mathrm{E}}}
\newcommand{\I}{{\rm I}}
\newcommand{\OO}{{\mathbf{0}}}

\newcommand{\Rn}{{\R}^n}

\newcommand{\dx}{\, \mathrm{d} x}
\newcommand{\ds}{\, \mathrm{d} s}
\newcommand{\de}{\, \mathrm{d}}

\renewcommand{\dh}{\, \mathrm{d} \mathcal{H}^{n-1}}
\newcommand{\hn}{\mathcal{H}^{n-1}}

\newcommand{\wt}{\widetilde}
\newcommand{\tr}{\mathrm{tr}\,}
\newcommand{\diver}{\mathrm{div}\,}

\newcommand{\Ker}{\mathrm{Ker}\,}
\newcommand{\Ima}{\mathrm{Im}\,}
\newcommand{\dime}{\mathrm{dim}\,}
\newcommand{\rk}{\mathrm{rk}\,}
\newcommand{\sm}{\setminus}
\newcommand{\Mnn}{{\mathbb{M}^{n}_{\mathrm{sym}}}}
\newcommand{\Mnnp}{{\mathbb{M}^{n, +}_{\mathrm{sym}}}}

\newcommand{\Mb}{\mathcal{M}_b}
\newcommand{\Mbn}{\mathcal{M}_b(\Omega;\Mnn)}

\newcommand{\Lnn}{L^2(\Omega;\Mnn)}
\newcommand{\Lbn}{L^2(\dom; \Rn)}
\newcommand{\Hdiv}{H(\mathrm{div},\Omega)}

\newcommand{\dom}{{\partial \Omega}}
\newcommand{\weak}{\rightharpoonup}
\newcommand{\wstar}{\stackrel{*}\rightharpoonup}

\newcommand{\mres}{\mathbin{\vrule height 1.6ex depth 0pt width
0.13ex\vrule height 0.13ex depth 0pt width 1.3ex}}

\theoremstyle{plain}
\begingroup
\theoremstyle{plain}
\newtheorem{theorem}{Theorem}[section]
\newtheorem{corollary}[theorem]{Corollary}
\newtheorem{proposition}[theorem]{Proposition}
\newtheorem{lemma}[theorem]{Lemma}
\theoremstyle{definition}
\newtheorem{definition}[theorem]{Definition}
\theoremstyle{remark}
\newtheorem{remark}[theorem]{Remark}

\endgroup

\numberwithin{equation}{section}
\title[Dissipative boundary conditions and entropic solutions]{Dissipative boundary conditions and entropic solutions in dynamical perfect plasticity}
\author[J.-F. Babadjian]{Jean-Fran\c cois Babadjian}
\address[J.-F. Babadjian]{Universit\'e Paris-Saclay, CNRS,  Laboratoire de math\'ematiques d'Orsay, 91405, Orsay, France.}
\email[J.-F. Babadjian]{jean-francois.babadjian@math.u-psud.fr}

\author[V. Crismale]{Vito Crismale}
\address[V. Crismale]{CMAP, \'Ecole Polytechnique, CNRS, 91128 Palaiseau Cedex, France.}
\email[V. Crismale]{vito.crismale@polytechnique.edu}

\begin{document}
\begin{abstract}
We prove the well--posedness of a dynamical perfect plasticity model under general assumptions on the stress constraint set and on the reference configuration. The problem is studied by combining both calculus of variations and hyperbolic methods. The hyperbolic point of view enables one to derive a class of dissipative boundary conditions, somehow intermediate between homogeneous Dirichlet and Neumann ones. By using variational methods, we show the existence and uniqueness of solutions. Then we establish the equivalence between the original variational solutions and generalized entropic--dissipative ones, derived from a weak hyperbolic formulation for initial--boundary value Friedrichs' systems with convex constraints.
\end{abstract}

\maketitle

\setcounter{tocdepth}{1}  
\tableofcontents

\section{Introduction}

In this paper we study the problem of small--strain dynamical perfect plasticity, under general constitutive assumptions, exploiting both the theory of constrained initial--boundary value Friedrichs'  systems and techniques from the calculus of variations. Perfect plasticity is a classical theory in continuum mechanics (see e.g. \cite{Hill, Lubliner}), developed to predict the appearance of permanent deformations in solids as the result of critical internal stresses. 
Let  $\Omega \subset \Rn$ be a  bounded open set  representing the reference configuration of an elasto--plastic body, $u :\Omega \times (0,T) \to \Rn$ be a kinematically admissible displacement field, and $\sigma:\Omega \times (0,T) \to \Mnn$ be a statically admissible Cauchy stress tensor. In a dynamical framework, given an external body load $f\colon\Omega \times (0,T) \to \Rn$, $u$ and $\sigma$ satisfy the equation of motion
\begin{subequations}\label{eqs:perplasintro}
\begin{equation}\label{1.4}
\ddot{u} - \diver \sigma=f\quad\text{in }\Omega \times (0,T)\,,
\end{equation}
where $\ddot u$, the second partial derivative of $u$ with respect to time, is the acceleration. In small strain perfect plasticity, the linearized strain tensor $\E u:=(\mathrm{D}u+\mathrm{D}u^T)/2$ decomposes as
\begin{equation}\label{1.1}
\E u=e+p\,,
\end{equation}
where $e$ and $p:\Omega \times (0,T) \to \mathbb M^n_{\rm sym}$ are the elastic and the plastic strains, corresponding to reversible and irreversible deformations, respectively. The {\it Hooke's law}
\begin{equation}\label{1.1b}
\sigma:=\C e= \lambda(\tr e) \I_n + 2 \mu e
\end{equation} 
gives the expression of the Cauchy stress as a linear mapping of the elastic strain, for some  fourth order symmetric and isotropic elasticity tensor $\C$, whose expression involves the {\it Lam\'e coefficients} $\lambda$, $\mu$ satisfying the ellipticity conditions $ \mu>0$ and $2\mu + n \lambda >0$. In perfect plasticity, the Cauchy stress tensor is constrained to stay in a fixed closed and convex set $\mathbf K \subset \Mnn$ containing the origin as an interior point
\begin{equation}\label{1.2}
\sigma \in \mathbf K\,.
\end{equation}
When $\sigma$ lies inside the interior of $\mathbf K$, the material behaves elastically  and no additional inelastic strains are created  ($\dot p=0$). On the contrary, when $\sigma$ reaches the boundary of $\mathbf K$ a plastic flow may develop in such a way that a non--trivial permanent plastic strain $p$ may remain after unloading.  The evolution of $p$ is described by the {\it Prandtl-Reuss law}
$\dot p \in N_{\mathbf K}(\sigma)$,
where $N_{\mathbf K}(\sigma)$ is the normal cone to $\mathbf K$ at $\sigma$, or equivalently by the {\it Hill's principle of maximum plastic work 
}
\begin{equation}\label{1.3}
\sigma\colon\dot p=H(\dot p)\,,
\end{equation}
with $H(q)=\sup_{\tau \in \mathbf K}\tau\colon q$ the support function of $\mathbf K$.
 
The problem of dynamical perfect plasticity consists thus in finding a displacement
$u \colon \Omega \times (0,T) \to \Rn$, an elastic strain $e \colon \Omega \times (0,T) \to \Mnn$, a plastic strain $p \colon \Omega \times (0,T) \to \Mnn$, and a Cauchy stress $\sigma \colon \Omega \times (0,T) \to \Mnn$ 
such that \eqref{1.4}--\eqref{1.3} hold in $\Omega \times (0,T)$, starting from initial conditions
\begin{equation}\label{1.5}
(u(0),\dot u(0),e(0),p(0))=(u_0,v_0,e_0,p_0),
\end{equation}
and with suitable boundary conditions on $\partial \Omega \times (0,T)$.
\end{subequations}
In the quasistatic setting, i.e., when the equation of motion \eqref{1.4} is replaced by the equilibrium equation $-\diver \sigma=f$ in $\Omega \times (0,T)$, the problem has been studied in \cite{Suq81b} and later revisited in \cite{DMDSMor06} in the framework of rate independent processes. A similar static problem, also known as Hencky plasticity, has been investigated in \cite{TemStr80,AnzGia82}. In the dynamical case, existence and uniqueness of solutions to this problem has been proven in the case of Dirichlet boundary conditions on $u$ in \cite{AnzLuc87} when $\mathbf{K}$ is cylindrical and bounded in the direction $\mathbb M^n_D$ of deviatoric matrices (typically the Von Mises or Tresca models), and in \cite{BabMor15} for a general $\mathbf{K}$. 
The proof in \cite{AnzLuc87, BabMor15} relies on the approximation through visco--plastic regularized models. Another approach introduced in \cite{DavSte19} is based on the minimization of time--space convex functionals (called WIDE for weighted--inertia--dissipation--energy) involving a small parameter in an exponentially time decreasing weight, and by passing to the parameter limit. We remark that in this last reference, the authors also consider the case of mixed Dirichlet and homogeneous--Neumann boundary conditions but with a cylindrical elasticity set as in \cite{AnzLuc87}, bounded in the direction of $\mathbb M^n_D$.

The present contribution  focuses on an approach to dynamic perfect plasticity based on entropic formulations of hyperbolic conservation laws. A simplified anti--plane model has been studied in \cite{BM17}. We extend here the result of \cite{BM17} to the general vectorial case of \eqref{eqs:perplasintro}, by removing the antiplane shear displacement  assumption, by considering a general convex set $\mathbf{K}$, and by working in general Lipschitz reference configurations $\Omega$. In the sequel, we describe the abstract approach, and point out the mathematical issues arising in the general elasto--plastic setting.

\medskip

The starting point is to express \eqref{1.4}, \eqref{1.1}, \eqref{1.1b} and \eqref{1.5} as a non--stationary Friedrichs' system \begin{subequations}\label{eqs:FrisistIntro}
\begin{equation}\label{Frisist1Intro}
\begin{cases}
\displaystyle \partial_t U + \sum_{i=1}^n A_i \,\partial_i U +P= F\,,\\
U(0)=U_0\,,
\end{cases}
\end{equation}
where each component of $U\colon\Omega \times (0,T) \to \R^d$, with $d=n+\frac{n(n+1)}{2}$ is a suitable explicit linear combination of the $n$ components of $\dot u$ and of the $\frac{n(n+1)}{2}$ components of the $n \times n$ symmetric matrix $\sigma$ (and similarly for $U_0$). The vector $P\colon\Omega \times (0,T) \to \R^d$ is a residual term (usually not present in Friedrichs systems) coming from the plastic strain rate $\dot p$, the source term $F\colon\Omega \times (0,T) \to \R^d$ comes from the body loads $f$, and the $d \times d$ symmetric matrices $A_1,\ldots, A_n$ are only depending  explicitly on the Lam\'e coefficients. The plasticity conditions \eqref{1.2}, \eqref{1.3} then correspond to \begin{equation}\label{Frisist2Intro}
U \in \mathbf{C},\qquad P\cdot (U-\xi) \geq 0 \quad\text{ for every }\xi \in \mathbf{C}\,,
\end{equation}
\end{subequations}
for a suitable closed and convex set $\mathbf{C}\subset \R^d$ containing $0_{\R^d}$ as an interior point. Therefore, \eqref{eqs:perplasintro} can be formally recast into  \eqref{eqs:FrisistIntro}, that we interpret as a Friedrichs' system with a convex constraint.

In the classical case of elastodynamics, we have $P=0$ and this formulation of the generalized wave equation describing the equation of motion in the framework of Friedrichs' systems has been introduced and studied in \cite{MH}.  Besides elastodynamics, many problems can be formulated in terms of Friedrichs' systems, for instance advection--reaction equations, advection--diffusion--reaction equations, the wave equation, the linearized Euler equations, the Maxwell equations in the so-called elliptic regime, and several other equations in mathematical physics (see e.g.\ \cite{BurErc16,AntBurCrnErc17}).

When the problem is set in the whole space $\R^n$, a notion of entropic solutions of Friedrichs' systems with convex constraints has been introduced in \cite{DepLagSeg11} and characterized in \cite{BabMifSeg16}  as the unique limit of  regularized solutions by means of a constraint penalization and a vanishing diffusion. This theory turns out to be well defined in Lebesgue type spaces. However, problems in continuum mechanics are often formulated in bounded  domains. For that reason, it is important to extend the previous approach to initial--boundary value problems. The difficulty is that the boundary conditions must be compatible with the evolution, especially in the hyperbolic framework where the initial conditions are propagated at finite speed through the characteristics and, possibly, up to the boundary. The starting point of our analysis (which was also the point of view of \cite{DMS,BM17}) are so-called dissipative boundary conditions introduced in \cite{F}. This notion of boundary conditions was then used in \cite{DMS} to formulate an $L^2$-theory for initial--boundary value  Friedrichs' system with an intrinsic formulation of boundary conditions, in the spirit of the $L^\infty$-theory of  boundary value scalar conservation laws \cite{MNRR, O} (see also the abstract approach to Friedrichs' systems in Hilbert spaces of \cite{ErnGueCap07, ErnGue06}, and \cite{AntBur10} on the relations between different types of boundary conditions). The advantage of this so-called entropic--dissipative formulation is that it does not involve any notion of trace of the solution on the boundary, which is meaningful when one is looking for solutions in Lebesgue type spaces.

This paper completes the analysis of a notion of \emph{entropic--dissipative solution} for initial--boundary value Friedrichs' systems with a convex constraint, obtained by a combination of \cite{DMS} and \cite{DepLagSeg11}, in the case of 
dynamic perfect plasticity. Such  solutions are functions $U \in L^2(\Omega \times (0,T);\mathbf C)$ such that for all constant vectors $\xi \in \mathbf C$ and all test functions $\varphi \in C^\infty_c( \R^n  \times (-\infty,T))$ with $\varphi \geq 0$,
\begin{subequations}
\begin{multline}\label{FSIntro}
 \int_0^T \int_{\Omega}   |U-\xi|^2\partial_t \varphi\dx\de t + \sum_{i=1}^n \int_0^T \int_{\Omega} A_i(U-\xi)\cdot (U-\xi)\partial_i \varphi\dx \de t \\
+ \int_\Omega |U_0-\xi|^2 \varphi(0)\dx + 2\int_0^T \int_{\Omega} F\cdot (U-\xi)\varphi\dx\de t 
 +\int_0^T \int_{\partial \Omega} M \xi^{+}\cdot \xi^{+} \varphi\,  \de \mathcal H^{n-1}\de t \ge 0.
\end{multline}
In the expression above, for $x \in \partial \Omega$, we denote by $\nu(x)$ the unit outer normal to $\partial \Omega$ at $x$,  $A_\nu(x):= \sum_{i=1}^n \nu_i(x) A_i$ and
$M(x) \in \mathbb{M}^{d}_{\rm sym}$ satisfies 
\begin{equation}\label{Mintro}
\begin{dcases}
\hspace{-1em}& M(x) \text{ is nonnegative,}\\
\hspace{-1em}& \Ker A_\nu(x) \subset \Ker M(x)\,,\\
\hspace{-1em}& \R^{d} = \Ker\big(A_\nu(x) - M(x)\big) + \Ker\big(A_\nu(x) + M(x)\big)\,,
\end{dcases}
\end{equation}
\end{subequations}
and $\xi^+$ is the projection of $\xi$ onto $\Ker\big(A_\nu(x) + M(x)\big) \cap \Ima A_\nu$. The formulation \eqref{FSIntro} is a combination of the unconstrained initial--boundary value problem introduced in \cite[Definition~1]{DMS}, and the constrained Cauchy problem in the full space introduced in \cite[Definition~2]{DepLagSeg11}. It has two main features: first, the fact the test vectors $\xi$ belong to the convex set $\mathbf C$ entirely contains the information of the flow rule; and second, the fact that the boundary condition is expressed through a boundary term not involving the trace of the solution which allows for a Lebesgue spaces type theory. The formulation \eqref{FSIntro} is referred to as an {\it entropic formulation} since it involves functions $U \mapsto |U-\xi|^2$ which are analogous to the Kru\v{z}kov entropies in \cite{Kru70} for scalar conservations laws, and to as a {\it dissipative formulation} by analogy of the loss of kinetic energy for Euler equations in \cite{Lio96} since, in the case $F=0$, the boundary condition implies the time decrease of $\|U(t)\|_{L^2(\Omega)}$.

In \cite{F} the admissible boundary conditions in a ``strong sense'' are given by
\begin{equation}\label{FriBC}
(A_\nu -M)U=0\quad\text{on }\partial \Omega\times(0,T)\,,
\end{equation} for a matrix $M$ satisfying \eqref{Mintro}. In the case where $A_\nu$ is invertible, as in plasticity, the  difference between \cite{F} and  \cite{DMS} on admissible matrices $M$ is that in the latter case, $M$ has to be symmetric since only its symmetric part is involved  in the formulation \eqref{FSIntro}. 

In the present case, the characterization of all admissible $M$ in the sense of \eqref{Mintro} requires a delicate abstract algebraic treatment. This is detailed in Paragraph~\ref{par:3.2.3} in dimension $n=3$, but it may be extended to any dimension. It turns out that, for a given $A_\nu$, the matrix $M$ is determined by a  $n\times n$ positive definite matrix $S$, in terms of which we reformulate \eqref{FriBC} as 
\begin{equation}\label{admBCVar}
S\dot{u} + \sigma\nu=0\quad\text{on }\partial \Omega\times(0,T)\,.
\end{equation}
Thanks to this characterization of the admissible boundary conditions, we can formulate \eqref{FSIntro} directly for the pair $(\dot{u}, \sigma)$ (instead of $U$), to specialize the notion of \emph{entropic--dissipative solution} to the present dynamic elasto--plasticity model \eqref{eqs:FrisistIntro} (see Definition~\ref{def:entropic}). 
The boundary condition \eqref{admBCVar} now formally complements the problem \eqref{eqs:perplasintro}. Unfortunately, the convex constraint \eqref{1.2} might be incompatible with \eqref{admBCVar} because $S \dot{u}$ may not belong to $K\nu$ on some part of $\partial \Omega$. Motivated by this observation, we consider an elasto--visco--plastic approximation, where in correspondence to a viscosity parameter $\varepsilon>0$, we add a damping term by replacing the stress $\sigma$ by $\sigma_\varepsilon + \varepsilon \E \dot{u}_\varepsilon$ in \eqref{1.4} and \eqref{admBCVar}, following a Kelvin-Voigt rheology, and the convex constraint \eqref{1.2}, \eqref{1.3} is penalized through a Perzyna visco--plastic approximation 
$$\dot{p}_\varepsilon=\frac{\sigma_\varepsilon- \mathrm{P}_{\mathbf K}(\sigma_\varepsilon)}{\varepsilon}\,.$$
In the spirit of \cite{DMS} and 
\cite{BM17}, we study the asymptotics of these approximating solutions (whose existence and uniqueness is standard, see e.g. Proposition~\ref{teo:apprvisc}) as $\varepsilon\to 0$.
Our main results are twofold:
\begin{enumerate}
\item In Theorem~\ref{teo:existence_variational}, we prove the existence and uniqueness of a \emph{variational solution to} \eqref{eqs:perplasintro} supplemented with the \emph{relaxed boundary conditions}
\begin{equation*}
P_{-\mathbf K\nu}(S\dot u)  + \sigma \nu = 0\quad \text{on } \dom \times (0,T)\,,
\end{equation*}
which is defined in Definition~\ref{def:variational}; 
\item In Theorems~\ref{thm:vardiss} and \ref{thm:dissvar}, we establish the equivalence between variational and entropic--dissipative solutions.
\end{enumerate}

Although the outcomes of our analysis are similar to those in the simplified setting of \cite{BM17},  we develop here completely  different tools due to the lack of counterparts for the regularity results employed 
in \cite{BM17}, specific to the anti--plane case, and of the setting therein. In particular, this allowed us both to simplify the argument and to avoid further regularity assumptions on $\Omega$ and $\mathbf{K}$, as we briefly discuss below.
 
\medskip 

One of the main features of plasticity models is the linear growth of the plastic dissipation (the support function $H$ of the elasticity set $\mathbf K$), which leads to strain concentration and to a possibly singular measure plastic strain $p$. From a functional analytic point of view, kinematically admissible displacements fields turn out to belong to the space $BD$ of functions of bounded deformation introduced in \cite{Suquet} for the study of perfect plasticity models (see also \cite{Tem85}). This pathological functional setting prevents to easily give a sense to the flow rule \eqref{1.3} since the stress $\sigma$ (which is usually square integrable as in classical elasticity) is not in duality with the plastic strain $p$ which is a singular measure. It is although possible to define a generalized stress--strain duality $[\sigma:\dot p]$ as in \cite{KohTem83,FraGia12,BabMor15} in $\Omega$, but the lack of control of the boundary terms prevents to extend this definition to $\overline \Omega$. A remedy to this obstacle in variational approaches to evolution problems is to express the flow rule as an energy--dissipation balance. It thus leads to study the limit of the energy--dissipation balance for approximated evolutions. In the present model, the relaxation of boundary conditions follows from the interplay between the stress constraint \eqref{1.2}, \eqref{1.3} and the unrelaxed boundary conditions \eqref{admBCVar} for the viscous approximations. 

Concerning the lower energy--dissipation inequality, the difficulty is related to the lack of lower semicontinuity of the term 
$$\int_0^t\int_\Omega H(\dot{p}_\varepsilon) \dx\ds + \int_0^t \int_\dom S \dot{u}_\varepsilon \cdot \dot{u}_\varepsilon \dh\ds$$
with respect to the weak convergence in the energy space. Using the approximate boundary condition, it can be rewritten as 
\begin{equation}\label{EBeps}
\int_0^t \int_\Omega H(\dot{p}_\varepsilon) \dx\ds + \frac12 \int_0^t \int_\dom S \dot{u}_\varepsilon \cdot \dot{u}_\varepsilon \dh \ds+ \frac12 \int_0^t \int_\dom  \hspace{-0.5em} S^{-1}(\sigma_\varepsilon\nu)  \cdot(\sigma_\varepsilon\nu) \dh\ds + o(1).
\end{equation}
Here, the approximate displacement $u_\varepsilon(t) \wstar u(t)$ weakly* in $BD(\Omega)$ with a trace which is bounded in $L^2(\partial\Omega;\Rn)$. Unfortunately, the trace operator in $BD(\Omega)$ is not continuous with respect to the weak* convergence in that space (see \cite{Bab}) so that  $u_\varepsilon(t) \weak w(t)$ weakly in $L^2(\partial\Omega;\Rn)$ for some $w(t)$ which differs from the trace of $u(t)$. Passing to the lower limit in \eqref{EBeps}, owing to standard lower semicontinuity results for convex functionals of measures, yields a lower bound of the form
\begin{multline*}
\int_0^t \int_\Omega H(\dot{p}) \dx\ds+ \int_0^t \int_\dom H((\dot w-\dot u)\odot \nu)\dh \ds\\
+ \frac12 \int_0^t \int_\dom S \dot{w} \cdot \dot{w} \dh \ds+ \frac12 \int_0^t \int_\dom  S^{-1}(\sigma\nu)  \cdot(\sigma\nu) \dh\ds,
\end{multline*}
where the first integral should be interpreted as a convex functional of a measure. This lower bound suggests to define the following infimal convolution boundary function
 \begin{equation}\label{defPsiintro}
 \psi(x,z)  =  \inf_{z' \in \Rn} \left\{ \frac{1}{2} S(x) z' \cdot z' + H((z'-z)\odot \nu(x)) \right\}\,,
\end{equation}
which leads to the lower bound
\begin{multline}\label{sciHintro}
\int_0^t \int_\Omega H(\dot{p}) \dx\ds+ \int_0^t\int_\dom  \psi(x,\dot{u}) \dh\ds+ \frac12 \int_0^t \int_\dom S^{-1}(\sigma\nu)  \cdot(\sigma\nu) \dh\ds\\
 \leq \liminf_{\varepsilon\to 0} \left\{ \int_0^t \int_\Omega H(\dot{p}_\varepsilon) \dx\ds +  \int_0^t \int_\dom S \dot{u}_\varepsilon \cdot \dot{u}_\varepsilon \dh \ds\right\}.
\end{multline}
In particular, this lower semicontinuity result improves the analogue result \cite[Proposition 5.1]{BM17} in the scalar case,  where the $\mathcal C^1$ regularity of the boundary was needed to use the optimal constant in the trace inequality in $BV$  (see \cite{AnzGia78}). The lower bound \eqref{sciHintro} is used to infer the lower inequality in the limit energy--dissipation balance \eqref{en_balance}. 

As for the upper energy--dissipation inequality, the formal argument leading to the missing energy inequality is very simple and essentially rests on both convexity inequalities:
$$H(\dot p(t))\geq \sigma(t):\dot p(t) \text{ in }\Omega,\quad\psi(x,\dot u(t))+\frac12 S^{-1}(\sigma(t)\nu)\cdot(\sigma(t)\nu) \geq -(\sigma(t)\nu)\cdot \dot u(t)\text{ on }\dom$$
together with an integration by part formula. In our case, the first inequality can be proven to be satisfied in the sense of measures, while the validity of the second one depends on  whether $\sigma(t)\nu \in K\nu$ on $\dom$. This property is known to be true in the case where $\mathbf K$ is cylindrical and bounded in the direction of deviatoric matrices (see \cite[Theorem 3.5]{FraGia12}). The generalization to any arbitrary convex set $\mathbf K$ depends on the possibility to approximate 
the space of admissible stresses
$$ \Big\{\sigma \in L^2(\Omega;\Mnn) \colon\;  \mathrm{div}\, \sigma \in L^2(\Omega;\Rn), \, \sigma\nu \in L^2(\partial\Omega; \Rn)\Big\}$$
by smooth functions. The main difficulty is related to the fact that the normal trace $\sigma\nu$ of $\sigma \in L^2(\Omega;\Mnn)$ with $\diver \sigma \in L^2(\Omega;\Rn)$ is canonically defined as an element of $H^{-\frac12}(\dom;\Rn)$, so that any approximation will have a normal trace strongly converging in that space instead of $L^2(\dom;\Rn)$ as required. The desired approximation ensuring all properties is known  from \cite{CD} when $\sigma$ is vector--valued (instead of matrix-valued), where it is proven the density of smooth functions in the space $\{z \in L^2(\Omega;\Rn) \colon \mathrm{div}\,z \in L^2(\Omega),\,z\cdot\nu \in L^2(\dom)\}$, which naturally arises in the study of Maxwell's equations. The proof relies on an abstract regularity result \cite{JerKen82} for nonhomogeneous Neumann problem which falls outside the scope of standard elliptic regularity. Unfortunately, up to our knowledge, there is no analogue of that result for the symmetric matrix-valued  case. Even assuming the validity of both convexity formulas, we then need to be able to apply a generalized integration by parts formula, which is the crucial tool in \cite{BM17} (see \cite{AnzAMPA}). In our case, again, the validity of such an integration by parts formula would again equire to prove a density result as above. We overcome the lack of integration by parts formula by proving an integral version of the formula \eqref{defPsiintro} to obtain $\int_\dom \psi(x,\dot u(t))\dh$ by global minimization over a set of smooth functions $w$. Then we extend the measure $\dot p(t)$, defined only on $\Omega$, up to the boundary by a boundary term of the form $(w-\dot u(t))\odot \nu \hn \mres \dom$ where $w$ is an arbitrary smooth function. We are now in a situation very similar to plasticity models with a Dirichlet boundary condition given by $w$. We can in particular use the results in \cite{BabMor15}, more precisely the generalized convexity inequality $H(\dot p(t)) \geq [\sigma(t):\dot p(t)]$ in the sense of measures in $\overline \Omega$, and the generalized integration by parts formula for fixed $w$ proven therein. The desired upper inequality is then obtained by minimizing with respect to $w$.

A last issue to overcome is the uniqueness of variational solutions. Again the formal argument rests on taking the difference of the equations of motion associated to two different solutions, multiply by the difference of the velocities and integrate by parts. It usually leads to a comparison principle between solutions, known as Kato inequality in the context of hyperbolic equations.
Once again, we are stucked by the nonvalidity of the generalized integration by parts formula. We use here a variational argument based on the strict convexity of the total energy, which, to our knowledge, was never used in such an evolutionary context.

Finally, we show that both notions of entropic--dissipative solutions and variational solutions are actually identical. The implication that variational solutions are entropic--dissipative ones is quite easy and relies on the visco--plastic approximation.  For fixed $\varepsilon$, the formal operations leading to \eqref{FSIntro} can be fully justified and then we pass to the limit as $\varepsilon \to 0$. The converse implication is more involved and is inspired from \cite{BM17} except for the derivation of the relaxed boundary condition. In \cite{BM17}, a further $C^2$ regularity property of $\Omega$ is used to infer that any vector field  $z\in L^\infty(\Omega;\Rn)$ with $\mathrm{div}\,z \in L^2(\Omega)$ has a normal trace $z\cdot\nu\in L^\infty(\partial\Omega)$ which can be recovered as the weak* limit in $L^\infty$ of $z\cdot\nu(x)$ as the point $x\in \Omega$ tends to the boundary (see \cite[Theorem~2.2]{CheFri99}). This results strongly uses the fact that $z$ belongs to $L^\infty(\Omega;\Rn)$, which is not the case anymore in our (matrix-valued) situation since $\sigma(t) \in \mathbf K$ and $\mathbf K$ is not bounded. In order to recover the relaxed boundary conditions, we use the entropic inequality \eqref{FSIntro} to ensure that the normal trace $\sigma(t)\nu$ in actually a bounded measure on $\dom$ (instead of just a distribution in $H^{-\frac12}(\dom;\Rn)$), which allows us to localize the inequality on $\dom$. Then, since the term $\sigma(t)\nu$ is multiplied by an arbitrary vector, and all other terms in this inequality are absolutely continuous with respect to $\hn\mres \dom$, we infer that the singular part of the measure $\sigma(t)\nu$ vanishes, and we obtain so that $\sigma\nu \in L^2(\dom;\Rn)$. Finally the relaxed boundary condition is obtained as a consequence of the energy--dissipation balance and the uniqueness result proved for variational solutions. This new approach allows us to weaken the $C^2$ regularity hypothesis done in \cite{BM17} into a Lipschitz regularity.

\medskip

The paper is organized as follows. In Section 2, we introduce the main notation and recall basic facts used throughout this work. Section 3 is devoted to describe the model of small strain dynamic perfect plasticity and its reformulations as a constrained Friedrichs' system. We introduce the notion of entropic solution and derive all admissible dissipative conditions. In Section 4, we prove the existence and uniqueness of variational solutions to the dynamic perfectly plastic model endowed with the dissipative boundary condition. Finally, in Section 5 we prove that both notions of variational and entropic--dissipative solutions are equivalent, and, as a byproduct, the well posedness of the entropic--dissipative formuation.

\section{Notation and preliminaries}

\subsection{Linear algebra}

If $a$ and $b \in \R^n$, we write $a \cdot b:=\sum_{i=1}^n a_i b_i$ for the Euclidean scalar product, and we denote the corresponding norm by $|a|:=\sqrt{a \cdot a}$.

For any $m,n\in \mathbb{N}$, we denote by $\mathbb M^{m \times n}$ the space of $m \times n$ matrices, and by $\mathbb{M}^{n}_{\mathrm{sym}}$, $\mathbb M^n_D$, $\mathbb{M}^{n, +}_{\mathrm{sym}}$, and $\mathbb{M}^{n}_{\mathrm{skew}}$ the spaces of $n {\times} n$ matrices which are symmetric, symmetric deviatoric, symmetric positive definite, and skew symmetric, respectively. Given a matrix $A \in \mathbb M^{m \times n}$, we denote by $\Ker A$ its kernel, by $\Ima A$ its range, and by $\rk A$ its rank. If $m=n$, then $\det A$ stands for the determinant of $A \in \mathbb M^n$ and $\tr A$ denotes its trace. The space $\mathbb M^n$ is endowed with the Frobenius scalar product $A:B:=\tr(A^T B)$ and with the corresponding Frobenius norm $|A|:=\sqrt{A:A}$. The identity and zero matrices in $\mathbb M^n$ are denoted by $\I_n$ and $\OO_n$. Finally, if $a \in \R^m$ and $b \in \R^n$, we denote their tensor product by $a \otimes b:=ab^T \in \mathbb M^{m \times n}$. If $m=n$, the symmetric tensor product between $a$ and $b \in \Rn$ is defined by $a \odot b=(a\otimes b+b \otimes a)/2 \in \Mnn$.

\subsection{Convex analysis}

We recall several definitions and basic facts from convex analysis (see \cite{Rock}). Let $(V,\langle\cdot,\cdot\rangle)$ be an Euclidean space $\psi:V \to [0,+\infty]$ be a proper function (i.e.\ not identically $+\infty$). The convex conjugate of $\psi$ is defined as
$$\psi^*(q):=\sup_{z \in V} \bigl\{ \langle q,z\rangle - \psi(z) \bigr\} \quad \text{ for all }q \in V,$$
which is a convex and lower semicontinuous function. In particular, if $C \subset V$ is a convex set, we define the indicator function $\I_C$ of $C$ as $\I_C:=0$ in $C$ and $+\infty$ otherwise. The convex conjugate $(\I_C)^*$ of $\I_C$ is called the support function of $C$.

If $\phi_1,\phi_2 : V \to [0,+\infty]$ are proper convex functions, then the infimal convolution of $\phi_1$ and $\phi_2$ is defined as
$$(\phi_1  \boxempty  \phi_2)(z):=\inf_{z' \in V} \bigl\{ \phi_1(z-z')+\phi_2(z') \bigr\}\,,$$
which turns out to be a convex function. It can be shown that
$$\phi_1 \boxempty \phi_2=(\phi_1^*+\phi_2^*)^*\,.$$

\subsection{Measures}

The Lebesgue measure in $\Rn$ is denoted by $\mathcal L^n$, and the $(n-1)$-dimensional Hausdorff measure by $\hn$. If $X \subset \Rn$ is a Borel set and $Y$ is an Euclidean space, we denote by $\mathcal M_b(X;Y)$ the space of $Y$-valued bounded Radon measures in $X$ endowed with the norm $\|\mu\|:=|\mu|(X)$, where $|\mu|$ is the variation of the measure $\mu$. If $Y=\R$ we simply write $\mathcal M_b(X)$ instead of $\mathcal M_b(X;\R)$, and we denote by $\mathcal M_b^+(X)$ the cone of all nonnegative bounded Radon measures.

If the relative topology of $X$ is locally compact, by Riesz representation theorem, $\mathcal M_b(X;Y)$ can be identified with the dual space of $C_0(X;Y)$, the space of continuous functions $\varphi:X \to Y$ such that $\{|\varphi|\geq \varepsilon\}$ is compact for every $\varepsilon>0$. The weak* topology of $\mathcal M_b(X;Y)$ is defined using this duality.

Let $\mu \in \mathcal M_b(X;Y)$ and $f:Y \to [0,+\infty]$ be a convex, positively one-homogeneous function.  Using the theory of convex functions of measures developed in \cite{GS,DT1,DT2}, we introduce the nonnegative Radon measure $f(\mu) \in \mathcal M_b^+(X)$, defined for every Borel set $A \subset X$ by
$$f(\mu)(A):=\sup\left\{\sum_{i=1}^{m} f(\mu(A_i)) \right\}$$
where the supremum is taken over all $m \in \mathbb N$ and all finite Borel partitions $\{A_1,\ldots,A_m\}$ of $A$. In addition, it can be established that
$$f(\mu)=f\left(\frac{\de \mu}{\de |\mu|}\right)|\mu|\,,$$
where $\frac{\de \mu}{\de |\mu|}$ stands for the Radon-Nikod\'ym derivative of $\mu$ with respect to $|\mu|$.

\subsection{Functional spaces}

We use standard notation for Lebesgue spaces ($L^p$) and Sobolev spaces ($W^{s,p}$ and $H^s=W^{s,2}$).

The space of functions of bounded deformation is defined by
$$BD(\Omega)=\{u \in L^1(\Omega;\R^n) : \; \E u \in \mathcal M_b(\Omega;\mathbb M^n_{\rm sym})\}\,,$$
where $\E u:=(Du+Du^T)/2$ stands for the distributional symmetric gradient of $u$. We recall (see \cite{Tem85,Bab}) that, if $\Omega$ has Lipschitz boundary, every function $u \in BD(\Omega)$ admits a trace, still denoted $u$, in $L^1(\dom;\Rn)$ such that the integration by parts formula holds: for all $\varphi \in C^1(\overline \Omega;\Mnn)$,
$$\int_{\partial\Omega} u\cdot (\varphi\nu)\dh=\int_\Omega \diver \varphi \cdot u\dx + \int_\Omega \varphi:\de \E u\,.$$
Note that the trace operator is continuous with respect to the strong convergence of $BD(\Omega)$ but not with respect to the weak* convergence in $BD(\Omega)$.

Let us define 
$$\Hdiv=\{\sigma \in L^2(\Omega;\mathbb M^n_{\rm sym}) \colon \diver \sigma \in L^2(\Omega;\R^n)\}\,.$$
If $\Omega$ has Lipschitz boundary, for any $\sigma \in H(\mathrm{div}, \Omega)$ we can define the normal trace $\sigma\nu$ as an element of $H^{-\frac12}(\partial \Omega;\R^n)$  (cf. e.g.\ \cite[Theorem~1.2, Chapter~1]{Tem85}) by setting
\begin{equation}\label{2911181910}
\langle \sigma \nu, \psi \rangle_{\dom}:= \int_\Omega \psi \cdot \diver \sigma\dx + \int_\Omega \sigma \colon \E \psi \dx\,.
\end{equation}
for every $\psi \in H^1(\Omega;\Rn)$. 

\medskip

Following \cite{KohTem83,FraGia12,BabMor15}, we define a generalized notion of stress/strain duality as follows.

\begin{definition}Let $\sigma \in \Hdiv$ and $p\in \mathcal M_b(\overline \Omega;\Mnn)$ be such that $\E u= e + p$ in $\Omega$ and $p=(w-u)\odot \nu \hn$ on $\partial \Omega$ for some $u\in BD(\Omega) \cap L^2(\Omega;\Rn)$, $e\in \Lnn$ and $w \in H^1(\Omega;\Rn)$. We define the first order distribution 
$[\sigma \colon p]\in \mathcal D'(\R^n)$ by
$$\langle [\sigma \colon p], \varphi \rangle := -\int_\Omega \varphi \sigma : (e -\E w)\dx- \int_\Omega (u-w) \cdot \diver \sigma \, \varphi \dx - \int_\Omega \sigma \colon \big((u-w) \odot \nabla \varphi\big) \dx$$
for all $\varphi \in C^\infty_c(\Rn)$.
\end{definition}

In the sequel we will be interested in stresses $\sigma \in \Hdiv$ taking values in a given closed and convex subset $\mathbf K$ of $\mathbb{M}^{n}_{\mathrm sym}$ containing $0$ in its interior. We define the set of all plastically admissible stresses by
$$\mathcal K(\Omega):=\{ \sigma \in \Hdiv \colon \, \sigma(x) \in \mathbf K \text{ for a.e.\ } x \in \Omega \}$$
which defines a closed and convex subset of the space $\Hdiv$.

If $\Omega$ has a Lipschitz boundary, $\sigma \in \mathcal K(\Omega)$ and $H(p)$ is a finite measure, where $H$ is the support function of $\mathbf K$, then, using an approximation result for $\sigma$ by smooth functions (see \cite[Lemma 2.3]{DMDSMor06}) as well as the integration by parts formula in $BD(\Omega)$ (see \cite[Theorem 3.2]{Bab}), we can show as in \cite[Section 2]{BabMor15} that $[\sigma \colon p]$ may be extended to a bounded Radon measure in $\overline \Omega$, i.e. $[\sigma \colon p]\in\mathcal M_b(\overline \Omega;\mathbb M^n_{\rm sym})$, with
\begin{equation}\label{eq:conv-ineq}
H(p) \geq [\sigma \colon p]\quad\text{in }\mathcal M^+_b(\overline \Omega)\,.
\end{equation}

Note also that if $\varphi \in C^1_c(\Omega)$, thanks to the integration by parts formula in $H^1(\Omega;\Rn)$ the expression of duality reduces to
\begin{equation}\label{1410191537}
\langle [\sigma \colon p], \varphi \rangle = -\int_\Omega \varphi \sigma \colon e \dx- \int_\Omega u\cdot \diver \sigma \, \varphi \dx - \int_\Omega \sigma \colon (u\odot \nabla \varphi) \dx\,.
\end{equation}

\section{Description of the model}

\subsection{Small strain dynamic perfect plasticity}

Let us consider a bounded open set $\Omega \subset \Rn$ (in dimension $n=2$ or $3$) 
which stands for the reference configuration of an elasto-plastic body. We will work in the framework of small strain elasto-plasticity where the natural kinematic variable is the displacement field $u :\Omega \times (0,T) \to \Rn$ (or the velocity field $v:=\dot u$). In small strain elasto-plasticity, the linearized strain tensor $\E u:=(\mathrm{D}u+\mathrm{D}u^T)/2$ decomposes additively in the following form:
$$
\E u=e+p,
$$
where $e:\Omega \times (0,T) \to \mathbb M^n_{\rm sym}$ is the elastic strain and $p:\Omega \times (0,T) \to \mathbb M^n_{\rm sym}$ the plastic strain. The elastic strain is related to the Cauchy stress tensor $\sigma:\Omega \times (0,T) \to\mathbb M^n_{\rm sym}$ by means of Hooke's law $\sigma:=\C e$, where $\C$ is the symmetric fourth order elasticity tensor. In the sequel we will assume that $\C$ is isotropic, which means that
$$\C e= \lambda(\tr e) \I_n + 2 \mu e\,,$$
where $ \mu>0$ and $2\mu + n \lambda >0$. In a dynamical framework and in the presence of an external body load $f:\Omega \times (0,T) \to \Rn$, the equation of motion writes
$$
\ddot{u} - \diver \sigma=f\quad\text{in }\Omega \times (0,T)\,.
$$
Perfect plasticity is characterized by the existence of a fixed closed and convex subset $\mathbf K$ of $\mathbb M^n_{\rm sym}$ with non empty interior, in which the Cauchy stress tensor is constrained to stay:
$$
\sigma \in \mathbf K\,.
$$
If $\sigma$ lies inside the interior of $\mathbf K$, the material behaves elastically, so that unloading will bring back the body into its initial configuration ($\dot p=0$). On the other hand, if $\sigma$ reaches the boundary of $\mathbf K$ (called the yield surface), a plastic flow may develop, so that, after unloading, a non-trivial permanent plastic strain $p$ will remain.  Its evolution is described by means of the flow rule and is expressed with the Prandtl-Reuss law
$$
\dot p \in N_{\mathbf K}(\sigma)\,,
$$ 
where $N_{\mathbf K}(\sigma)=\partial {\rm I}_{\mathbf K}(\sigma)$ is the normal cone to $\mathbf K$ at $\sigma$. Note that the flow rule can be equivalently written as 
$$\sigma\colon\dot p=H(\dot p)\,,$$
where $H(q)=\sup_{\tau \in \mathbf K}\tau\colon q$ is the support function of $\mathbf K$, which expresses Hill's principle of maximum plastic work.

\medskip

In summary, the problem of dynamical perfect plasticity consists in finding a displacement field
$u \colon \Omega \times (0,T) \to \Rn$, a Cauchy stress $\sigma \colon \Omega \times (0,T) \to \Mnn$, an elastic strain $e \colon \Omega \times (0,T) \to \Mnn$ and a plastic strain $p \colon \Omega \times (0,T) \to \Mnn$ such that in $\Omega \times (0,T)$ the following hold:
\begin{equation}\label{0110181923}
\begin{dcases}
\ddot{u} - \diver \sigma=f\,;\\
\E u= e+p\,;\\
\sigma= \C e \in \mathbf K\,;\\
\sigma\colon\dot p=H(\dot p)\,.
\end{dcases}
\end{equation}
To formulate the Cauchy problem, we add initial conditions of the form
\begin{equation}\label{eq:initial-cond}
(u(0),\dot u(0),e(0),p(0))=(u_0,v_0,e_0,p_0)\,.
\end{equation}
Eventually, in order to close the initial-boundary value problem, we have also to add suitable boundary conditions, whose derivation will be the object of the following subsections.

\subsection{Entropic formulation of the model}\label{sec:entrop}

The object of this subsection is to show that, formally, the system of dynamic elasto--plasticity can be reformulated in a framework similar to initial--boundary value conservations laws, provided the boundary conditions are compatible with the hyperbolic structure of the problem. We will first recast \eqref{0110181923} as a constrained Friedrichs' system as in \cite{DepLagSeg11}. Then, defining a suitable notion of so-called dissipative boundary conditions introduced in \cite{DMS} and taking its origin in the seminal paper \cite{F}, we will in turn reformulate the problem as infinitely many nonlinear inequalities involving entropies as for conservations laws. The advantage of this new formulation is that, as in \cite{O}, it involves a boundary term containing the information of the boundary condition without appealing to any notion of trace of the solution on the boundary. In particular, it allows one to define a weaker notion of solutions in a (Lebesgue) functional space which is larger than the energy space provided by variational methods.

\subsubsection{Dynamic perfect plasticity as a constrained Friedrichs' system}

We first perform formal operations on the system \eqref{0110181923} to write it in a fashion  similar to that of linear, symmetric hyperbolic systems called Friedrichs' systems. In this paragraph we actually show that the first three conditions in \eqref{0110181923} correspond to consider  a ``constrained Friedrichs' system'' of the form
$$\partial_t U + \sum_{i=1}^n A_i \partial_{i} U + P=F\,,$$
where each line of the unknown $U: \Omega \times (0,T) \to \R^d$ ($d=n+\frac{n(n+1)}{2}$) is a  linear combination of the $n$ components of the velocity $v=\dot u$ and of the $\frac{n(n+1)}{2}$ components of the (symmetric) Cauchy stress $\sigma$ (see also \cite{MH} in the case elastodynamics).  For $i\in\{1,\ldots,n\}$, the matrices $A_i \in \mathbb M^d_{\rm sym}$ are constants matrices whose coefficients only depend on the Lam\'e coefficients $\lambda$ and $\mu$, and the vectors $P$ and $F$ result from the plastic strain $\dot p$ and the source term $f$, respectively. For simplicity, we will perform our computations only in the case $n=3$.

\medskip

Let us define the following (column) vectors representing, respectively, the velocity, the Cauchy stress and the linearized strain:
\begin{equation}\label{0110182012}
\begin{cases}
V  :=  (\dot{u}_1 ,\dot{u}_2 , \dot{u}_3)^T\in \R^3\,, \\
\Sigma :=  (\sigma_{11} ,\sigma_{22} , \sigma_{33}, \sigma_{12} , \sigma_{13}  , \sigma_{23})^T \in \R^6\,, \\
Y :=(\partial_1 \dot{u}_1 ,\partial_2 \dot{u}_2 , \partial_3 \dot{u}_3 , \partial_1 \dot{u}_2 + \partial_2 \dot{u}_1 , \partial_1 \dot{u}_3 + \partial_3 \dot{u}_1 ,  \partial_2 \dot{u}_3 + \partial_3 \dot{u}_2 )^T \in \R^6\,.
\end{cases}
\end{equation}
We also introduce the vectors corresponding to the initial velocity and the initial stress:
$$V_0=((v_0)_1,(v_0)_2,(v_0)_3)^T \in \R^3, \quad \Sigma_0:= ((\sigma_0)_{11} ,(\sigma_0)_{22} , (\sigma_0)_{33}, (\sigma_0)_{12} , (\sigma_0)_{13}  , (\sigma_0)_{23})^T\in \R^6.$$

First of all, we have that
\begin{equation}\label{0210181655}
Y=-\sum_{i=1}^3 (B_i)^T \, \partial_i V\,,
\end{equation}
where the matrices $B_1$, $B_2$ and $B_3 \in \mathbb M^{3{\times}6}$ are given by
$$B_1=-
\begin{pmatrix}
1 & 0 & 0 & 0 & 0 & 0 \\
0 & 0 & 0 & 1 & 0 & 0 \\
0 & 0 & 0 & 0 & 1 & 0
\end{pmatrix},$$
$$B_2=-
\begin{pmatrix}
0 & 0 & 0 & 1 & 0 & 0 \\
0 & 1 & 0 & 0 & 0 & 0 \\
0 & 0 & 0 & 0 & 0 & 1
\end{pmatrix},$$
and $$B_3=-
\begin{pmatrix}
0 & 0 & 0 & 0 & 1 & 0 \\
0 & 0 & 0 & 0 & 0 & 1 \\
0 & 0 & 1 & 0 & 0 & 0
\end{pmatrix}.$$
Employing such matrices $B_i$ we express the equation of motion $\ddot u-{\rm div}\sigma=f$ as
\begin{equation}\label{0210181729}
\partial_t V+ \sum_{i=1}^3 B_i \partial_i \Sigma =f\,,
\end{equation}
and the constitutive law $ \dot \sigma=\mathbf A \dot e= \lambda \,(\diver \dot u) \, \I_3 + 2\mu \E \dot u-\lambda\, (\tr \dot p) \, \I_3 - 2\mu \dot p$ as
\begin{equation}\label{0110182011}
\partial_t \Sigma +CQ= C Y\,,
\end{equation}
where
$$Q:= (\dot p_{11} ,\dot p_{22} , \dot p_{33}, 2 \dot p_{12} ,2 \dot p_{13}  ,2 \dot p_{23})^T \in \R^6\,,$$
and
$$C:=
\begin{pmatrix}
\tilde C & \OO_{3} \\
\OO_{3} & \mu \,\I_{3}
\end{pmatrix}
\in \mathbb M^{6 ,+}_{\rm sym}
\,,
\qquad \tilde C := 
\begin{pmatrix}
\lambda+2\mu & \lambda & \lambda\\
\lambda & \lambda + 2\mu & \lambda\\
\lambda & \lambda & \lambda + 2\mu\\
\end{pmatrix} \in \mathbb M^{3,+}_{\rm sym}\,.
$$
Collecting \eqref{0210181655}, \eqref{0210181729} and \eqref{0110182011}, we obtain
\begin{equation}\label{0210181734}
\wt{A}_0 \partial_t \wt U + \sum_{i=1}^3 \wt{A}_i \partial_i \wt U + \wt Q=F, \quad \wt U \in \wt{\mathbf C}
\end{equation}
where
\begin{equation}\label{2210181057}
\wt U:=\begin{pmatrix}
V \\ \Sigma
\end{pmatrix} \in \R^9\,, \quad
\wt{A}_0:= 
\begin{pmatrix}
\I_3 &  \OO_{3}  & \OO_3 \\
\OO_{3} & \tilde C^{-1} & \OO_{3} \\ 
\OO_{3} & \OO_{3} & \frac{1}{\mu} \,\I_{3}
\end{pmatrix}\in \mathbb M^{9,+}_{\rm sym}
\,, \quad \wt{A}_i:= 
\begin{pmatrix}
\OO_{3} & B_i \\
B_i^T & \OO_{6}
\end{pmatrix}\in \mathbb M^{9,+}_{\rm sym}\, ,
\end{equation}
$$\wt Q=\begin{pmatrix}
 0_{\R^3}  \\  Q
\end{pmatrix}, \quad 
F= \begin{pmatrix}f \\  0_{\R^6}  \end{pmatrix},$$
and $\wt{\mathbf C} \subset \R^9$ is the set of all vectors $\wt{U}=(V,\Sigma)^T \in \R^9$ such that
$$V \in \R^3, \quad \begin{pmatrix}
\Sigma_1 & \Sigma_4 & \Sigma_5\\
\Sigma_4 & \Sigma_2 & \Sigma_6\\
\Sigma_5 & \Sigma_6 & \Sigma_3
\end{pmatrix} \in \mathbf K\,,$$
which is a closed and convex set containing $0_{\R^9}$ in its interior by the properties of $\mathbf K$.

\medskip

In order to get rid of the matrix $\wt A_0$ in front of the time derivative in \eqref{0210181734}, we set $A_i:=\wt{A}_0^{-\frac12} \wt{A}_i \wt{A}_0^{-\frac12}$ for all $i\in \{1, 2, 3\}$, and
$$ U:=\wt{A}_0^{\frac12} \wt U, \quad P:=\wt{A}_0^{-\frac12} \begin{pmatrix}
 0_{\R^3}  \\  Q \end{pmatrix}, \quad  U_0=\wt{A}_0^{\frac12} \begin{pmatrix} V_0 \\ \Sigma_0\end{pmatrix}.
$$
Using that $F=\wt{A}_0^{-\frac12} F$, we can recast \eqref{0210181734} in the form
\begin{equation}\label{0210181935}
\begin{cases}
\displaystyle \partial_t U + \sum_{i=1}^3 A_i \,\partial_i U +P= F\,,\\
U \in \mathbf C\,,\\
U(0)=U_0\,,
\end{cases}
\end{equation}
where $\mathbf C=\wt{A}_0^{\frac12}\wt{\mathbf C}$, which is a closed and convex subset of $\R^9$ containing $0_{\R^9}$ in its interior. Note that from the expression of $\wt{A}_0$, we have that
\begin{equation}\label{1611180941}
\wt{A}_0^{\pm \frac12}=\left(
\begin{array}{ccc}
\I_3 &  \OO_{3}  & \OO_{3} \\
\OO_{3} & \wt C^{\mp \frac{1}{2}}  & \OO_3 \\ 
\OO_{3} & \OO_{3} & \mu^{\mp \frac{1}{2}} \,\I_3
\end{array}
\right)
\end{equation}
and, for all $i \in \{1,2,3\}$, denoting by $B_i=(B'_i|B''_i)$ with $B'_i$ and $B''_i \in \mathbb M^3$, then the matrices $A_i$ may  be expressed as
$$A_i=
\begin{pmatrix}
\OO_3 & B'_i \wt C^{\frac12} & \sqrt\mu B''_i\\
 \wt C^{\frac12}(B'_i)^T & \OO_3 & \OO_3\\
 \sqrt\mu (B''_i)^T & \OO_3 & \OO_3
\end{pmatrix}.$$

\begin{remark}\label{rem:0210182015}
We observe that $\tilde C$ may be diagonalized as
$$\tilde C= R D R^{T}\,,$$
where $R \in SO(3)$ and $D$ are given by
$$R=
\begin{pmatrix}
\frac{1}{\sqrt 2} & \frac{1}{\sqrt 6} & \frac{1}{\sqrt 3}\\
-\frac{1}{\sqrt 2} &  \frac{1}{\sqrt 6} &\frac{1}{\sqrt 3}\\
0 & -\frac{2}{\sqrt 6} & \frac{1}{\sqrt 3}
\end{pmatrix}, \qquad D=\mathrm{diag}(2\mu, 2\mu, 2\mu + 3\lambda)\,.$$
Thus, for any power $\gamma \in \R$, a direct computation gives that 
\begin{equation*}
\wt C^\gamma=R D^\gamma R^T= \begin{pmatrix}
\alpha_\gamma & \beta_\gamma & \beta_\gamma  \\
\beta_\gamma &  \alpha_\gamma & \beta_\gamma \\
\beta_\gamma & \beta_\gamma &  \alpha_\gamma  \\
\end{pmatrix} \,, 
\end{equation*}
for
$$\alpha_\gamma= \frac23 (2 \mu)^\gamma + \frac{1}{3} (2\mu + 3 \lambda)^\gamma\,, \qquad \beta_\gamma= -\frac{1}{3} (2 \mu)^\gamma + \frac{1}{3} (2\mu + 3 \lambda)^\gamma\,.$$
\end{remark}

We now employ the fourth condition in \eqref{0110181923} to infer an inequality useful in the subsequent analysis. We show that
\begin{equation}\label{0211191324}
P \cdot (U-\xi)\geq 0 \quad \text{for every } \xi \in \mathbf{C}\,.
\end{equation}
In fact, writing $\xi=\wt{A}_0^{\frac12} \wt\xi$ for some $\wt\xi=(z,\mathcal T)^T \in \wt{\mathbf C}$ with
$$z \in \R^3, \quad
\tau:= \begin{pmatrix}
\mathcal T_1 & \mathcal  T_4 &  \mathcal T_5\\
\mathcal  T_4 & \mathcal  T_2 & \mathcal  T_6\\
\mathcal  T_5 & \mathcal T_6 & \mathcal  T_3
\end{pmatrix} \in \mathbf K,$$ we get, using the symmetry of $\wt{A}_0^{\frac12}$, that $P\cdot (U-\xi)=(\wt{A}_0^{-\frac12}\wt Q) \cdot (\wt{A}_0^{\frac12}\wt U-\wt{A}_0^{\frac12}\wt \xi)=\wt Q\cdot (\wt U-\wt \xi)=Q\cdot (\Sigma-\mathcal T)$.
Then, by  the fourth condition in \eqref{0110181923}, we have that $\sigma  \colon  p \geq \tau  \colon p$,  and since $\tau \in \mathbf K$, we get that $Q\cdot (\Sigma-\mathcal T)=\dot p \colon (\sigma-\tau) \geq 0$, so \eqref{0211191324} holds true.

\subsubsection{Entropic formulation}

In the spirit of entropic formulations for conservations laws, we are going to reformulate \eqref{0110181923}, or \eqref{0210181935} and \eqref{0211191324}, as infinitely many nonlinear inequalities involving entropies well adapted to the system. At this stage we need to impose suitable admissible boundary conditions following an approach introduced in \cite{DMS,BM17} and in  the pioneering work \cite{F}. We are formally interested in dissipative boundary conditions of the form
\begin{equation}\label{0210181954}
(A_\nu-M)U=0 \quad \text{on }\dom \times (0,T)\,,
\end{equation}
where (denoting by $\nu(x)$ the unit outer normal to $\dom$ at $x \in \dom$)
$$A_\nu(x):= \sum_{i=1}^3 \nu_i(x) A_i \in \mathbb{M}^{9}_{\rm sym}$$
and, for each $x \in \partial \Omega$,  $M(x) \in \mathbb{M}^{9}_{\rm sym}$ satisfies
\begin{equation}\label{0210182002}
\begin{dcases}
\hspace{-1em}& M(x) \text{ is nonnegative,}\\
\hspace{-1em}& \Ker A_\nu(x) \subset \Ker M(x)\,,\\
\hspace{-1em}& \R^{9} = \Ker\big(A_\nu(x) - M(x)\big) + \Ker\big(A_\nu(x) + M(x)\big)\,.
\end{dcases}
\end{equation}
 
Taking the scalar product of the first equation in \eqref{0210181935} with $U$ yields
\begin{equation}\label{eq:1}
\frac12 \partial_t |U|^2+\frac12 \sum_{i=1}^3\partial_i (A_i U \cdot U) + P \cdot U = F\cdot U
\end{equation}
while, for every constant vector $\xi \in \mathbf C$, taking the  scalar product of \eqref{0210181935} with $\xi$  leads to
\begin{equation}\label{eq:2}
 \partial_t (U \cdot \xi)+ \sum_{i=1}^3\partial_i (A_i U \cdot \xi) + P \cdot \xi = F\cdot \xi.
\end{equation}
Substracting \eqref{eq:1} and \eqref{eq:2}, and using \eqref{0211191324}, leads to
$$\partial_t|U-\xi|^2 +\sum_{i=1}^3\partial_i \big( A_i(U-\xi)\cdot (U-\xi)\big) \leq 2 F \cdot (U-\xi).$$
We now multiply the above inequality
by a test function $\varphi \in C^\infty_c(\R^3 \times \R)$ with $\varphi \geq 0$, and integrate by parts to get that
\begin{equation}\label{Fried2}
\begin{split}
 &\int_{\R^+} \int_\Omega |U-\xi|^2 \partial_t \varphi\dx\de t + \sum_{i=1}^3\int_{\R^+} \int_\Omega A_i (U-\xi) \cdot (U-\xi)\partial_i \varphi \dx\de t 
+ \int_\Omega |U_0-\xi|^2\varphi(0)\dx\\
& \hspace{1em}+2 \int_{\R^+}\int_\Omega  F \cdot (U-\xi)\varphi \dx\de t- \int_{\R^+} \int_{\partial \Omega} A_\nu (U-\xi) \cdot (U-\xi) \varphi  \, \de \mathcal H^2\de t \geq 0\,.
\end{split}
\end{equation}

According to  \cite[Lemma 1]{DMS}, we have that
$$\R^9 = \Ker A_\nu  \oplus   \big(\Ker(A_{\nu}- {M})\cap  \Ima A_{\nu} \big)   \oplus  \big( \Ker(A_{\nu}+ {M}) \cap \Ima A_{\nu} \big).$$
For each $\xi\in \R^9$, we denote by $\xi^\pm$ the projection of $\xi$ onto $\Ker(A_{\nu}\pm {M})\cap \Ima A_{\nu} $. Using the (strong) boundary condition \eqref{0210181954}, we have that $U \in \Ker(A_\nu-M)$, or still $U^+=0$. The algebraic conditions \eqref{0210182002} together with \cite[Lemma 1]{DMS} thus yield
\begin{multline*}
A_\nu (U-\xi)\cdot  (U-\xi) =-M (U-\xi)^+\cdot  (U-\xi)^+ +M (U-\xi)^-\cdot  (U-\xi)^-\\
=-M\xi^+\cdot \xi^+ +M (U-\xi)^-\cdot  (U-\xi)^- \geq -M\xi^+\cdot \xi^+.
\end{multline*}
Inserting in \eqref{Fried2}, we get that for all constant vectors $\xi \in \mathbf C$ and all test functions $\varphi \in C^\infty_c( \R^3  \times \R)$ with $\varphi \geq 0$, it holds that
\begin{equation}\label{FS}
\begin{split}
& \int_{\R^+} \int_{\Omega}   |U-\xi|^2\partial_t \varphi\dx\de t + \sum_{i=1}^3 \int_{\R^+} \int_{\Omega} A_i(U-\xi)\cdot (U-\xi)\partial_i \varphi\dx \de t \\
 & \hspace{1em} + \int_\Omega |U_0-\xi|^2 \varphi(0)\dx + 2\int_{\R^+} \int_{\Omega} F\cdot (U-\xi)\varphi\dx\de t 
 +\int_{\R^+} \int_{\partial \Omega} M \xi^{+}\cdot \xi^{+} \varphi\,  \de \mathcal H^2 \de t \ge 0\,.
 \end{split}
\end{equation}
The previous family of inequalities defines a notion of entropic solutions $U \in L^2(\Omega \times \R^+;\mathbf C)$ to the dynamical elasto--plastic problem in the spirit of constrained Friedrichs' systems in the whole space as in \cite{DepLagSeg11}. Note that in the presence of a boundary, our formulation is meaningful within a $L^2$ theory of Friedrichs' systems (as suggested by \eqref{FS}) since the trace of $U$ on the boundary $\partial \Omega \times \R^+$ is not involved in this definition (see also \cite{O,MNRR} for an $L^\infty$-theory of initial--boundary value conservation laws).

\subsubsection{Characterization of all admissible boundary matrices}\label{par:3.2.3}

In order to derive the ``dissipative'' boundary conditions involved in our model, we need to  characterize all admissible boundary matrices $M$. Since the problem is local, let $\nu \in \mathbb S^2$ be fixed, and consider the matrix $A_\nu$ given by
$$A_\nu=\sum_{i=1}^3 \nu_i A_i.$$
A direct computation shows that
\begin{equation}\label{0210181955}
A_\nu =\begin{pmatrix}
\OO_3 &  A_\nu'  & A_\nu'' \\
(A_\nu')^T & \OO_3 & \OO_{3} \\ 
(A_\nu'')^T & \OO_{3} & \OO_{3}
\end{pmatrix}
\end{equation}
where, according to Remark \ref{rem:0210182015},
\begin{equation}\label{Aprime}
A_\nu'=B'_\nu \wt C^{\frac12} =-
\begin{pmatrix}
\nu_1\, \alpha & \nu_1\, \beta & \nu_1\, \beta\\
\nu_2\, \beta &  \nu_2\, \alpha& \nu_2\, \beta \\
\nu_3\, \beta & \nu_3\, \beta &  \nu_3\, \alpha
\end{pmatrix}\in \mathbb M^3\,, 
\quad
A_\nu''=\sqrt\mu B''_\nu=-\sqrt{\mu}
\begin{pmatrix}
\nu_2 & \nu_3 & 0 \\
\nu_1 & 0 & \nu_3\\
 0 & \nu_1 & \nu_2
\end{pmatrix} \in \mathbb M^3.
\end{equation}
with 
$$B'_\nu=\sum_{i=1}^3 B'_i \nu_i
=-\begin{pmatrix}
\nu_1 & 0 & 0 \\
0 & \nu_2 & 0\\
0 & 0 & \nu_3
\end{pmatrix}
, \quad
B''_\nu=\sum_{i=1}^3 B''_i \nu_i
=-\begin{pmatrix}
\nu_2 & \nu_3 & 0\\
\nu_1 & 0 & \nu_3\\
0 & \nu_1 & \nu_2
\end{pmatrix}
$$
and
$$\alpha=\alpha_{\frac12}= \frac23 \sqrt{2 \mu} + \frac{1}{3} \sqrt{2\mu + 3 \lambda}\,, \qquad \beta=\beta_{\frac12}= -\frac{1}{3} \sqrt{2 \mu}+ \frac{1}{3} \sqrt{2\mu + 3 \lambda}\,.$$

\medskip

Let us start by proving some dimensional properties of $A_\nu$.

\begin{lemma}
For all $\nu \in \mathbb S^2$, we have
\begin{equation}\label{1311181051}
\rk(A_\nu' | A_\nu'')=3
\end{equation}
and
\begin{equation}\label{0210182358}
\rk A_\nu = 6\,,\qquad \dime \Ker A_\nu =3.
\end{equation}
\end{lemma}

\begin{proof}
We first note that
\begin{equation}\label{1310191931}
\det(A'_\nu)=-\nu_1\nu_2\nu_3(2\mu)\sqrt{3\lambda+2\mu}\,.
\end{equation}
Then, \eqref{1310191931} gives \eqref{1311181051} if $\nu_i \neq 0$ for all $i \in \{1,2,3\}$.  Otherwise, assume that $\nu_i=0$ for some $i \in \{1,2,3\}$. Note that since $\nu \in \mathbb S^2$, then $\nu$ has at least one nonzero component.  Since $\mu>0$ and $3\lambda+2\mu>0$, then we have $\alpha\neq\beta$, $\alpha^2 \neq \beta^2$. Thus the nonzero lines of $A'_\nu$ are linearly independent, while the lines of $A''_\nu$ corresponding to the zero lines of $A'_\nu$ are linearly independent as well, which implies \eqref{1311181051} also in that case. 

Finally, \eqref{0210182358} is a consequence of the structure of the matrix $A_\nu$ in \eqref{0210181955} and the Rank Theorem.
\end{proof}

According to \eqref{1311181051}, in correspondence to $\nu \in \mathbb{S}^2$ there exists a matrix $C_\nu \in \mathbb M^9_{\rm sym}$ such that the fourth, fifth and sixth columns of $C_\nu A_\nu$ are linearly independent. We denote by

\begin{equation}\label{1310191900}
\widehat{A}_\nu:=C_\nu A_\nu C_\nu^T= \begin{pmatrix}
\OO_3 &  \widehat{A}_\nu'  & \widehat{A}_\nu'' \\
(\widehat{A}_\nu')^T & \OO_3 & \OO_{3} \\ 
(\widehat{A}_\nu'')^T & \OO_{3} & \OO_{3}
\end{pmatrix},
\end{equation}
with $\widehat A'_\nu$, and $\widehat A''_\nu \in \mathbb M^3$, 
which operates a suitable change of both lines and columns of $A_\nu$  in such a way that
\begin{equation}\label{1310191933}
\rk \widehat{A}_\nu' =3\,.
\end{equation} 
Notice that we may assume that $C_\nu={\rm I}_9$ (i.e. $A_\nu= \widehat A_\nu$), if and only if $\nu_i \neq 0$ for all $i\in\{1,2,3\}$.

\begin{remark}\label{rem:Cnu}
Since $C_\nu^T$ and $C_\nu^{-T}$ are permutation matrices which leave unchanged the first three components of a vector, we deduce that for any $U \in \R^9$,
$$(C_\nu^T U)_i=U_i=(C_\nu^{-T} U)_i \quad \text{for all }i \in \{1,2,3\}.$$
\end{remark}

\medskip

Let us now pass to the characterization of the admissible boundary matrices, fixing $x \in \partial \Omega$ in \eqref{0210182002}, so that we are given $M \in \mathbb M^9_{\rm sym}$ satisfying
\begin{equation}\label{eq:M}
\begin{dcases}
\hspace{-1em}& M \text{ is nonnegative,}\\
\hspace{-1em}& \Ker A_\nu \subset \Ker M\,,\\
\hspace{-1em}& \R^{9} = \Ker(A_\nu - M) + \Ker(A_\nu + M)\,.
\end{dcases}
\end{equation}
We write
\begin{equation}\label{0310180617}
M=\left(
\begin{array}{ccc}
M_1 &  M_2  & M_2' \\
(M_2)^T & M_3 & M_3' \\ 
(M_2')^T & (M_3')^T & M_3''
\end{array}
\right)\,,
\end{equation}
with $M_1$, $M_3$, $M''_3 \in \mathbb M^3_{\rm sym}$, $M_2$, $M_2'$ and $M'_3 \in \mathbb{M}^{3}$.  Moreover, let 
\begin{equation}\label{1310191911}
\widehat{M}:=C_\nu M C_\nu^T= \left(
\begin{array}{ccc}
\widehat{M}_1 &  \widehat{M}_2  & \widehat{M}_2' \\
(\widehat{M}_2)^T & \widehat{M}_3 & \widehat{M}_3' \\ 
(\widehat{M}_2')^T & (\widehat{M}_3')^T & \widehat{M}_3''
\end{array}
\right)\,,
\end{equation}
with $\widehat{M}_1$, $\widehat{M}_3$, $\widehat{M}''_3 \in \mathbb M^3_{\rm sym}$, $\widehat{M}_2$, $\widehat{M}_2'$ and $\widehat{M}'_3 \in \mathbb{M}^{3}$. Notice that $\widehat{M_1}=M_1$, by the properties of $C_\nu$.

\medskip

We have the following characterization of all admissible boundary matrices $M$.
\begin{theorem}\label{teo:1511181854}
Let $\nu \in \mathbb S^2$ and $A_\nu$,  $\widehat{A}_\nu$ be the corresponding matrices defined by \eqref{0210181955}  and \eqref{1310191900}.  A matrix $M \in \mathbb M^9_{\rm sym}$ satisfies \eqref{eq:M}, if and only if (recall \eqref{1310191911})
\begin{subequations}\label{eqs:1310191917}
\begin{equation}\label{1511181908}
\widehat{M}_1,\, \widehat{M}_3 \text{ are positive definite, }
\end{equation}
\begin{equation}\label{eq:skew}
(\widehat{A}'_\nu)^T\, \widehat{M}_1^{-1} \widehat{M}_2\in \mathbb M^3_{\rm skew}, \quad \widehat{A}'_\nu\, \widehat{M}_3^{-1} (\widehat{M}_2)^T \in \mathbb M^3_{\rm skew},
\end{equation}
\begin{equation}\label{eq:M1M3}
\widehat{M}_3 =  (\widehat{A}'_\nu)^T \, \widehat{M}_1^{-1} \widehat{A}'_\nu + (\widehat{M}_2)^T \widehat{M}_1^{-1} \widehat{M}_2,\quad
\widehat{M}_1 =  \widehat{A}'_\nu \, \widehat{M}_3^{-1} (\widehat{A}'_\nu)^T +  \widehat{M}_2 \widehat{M}_3^{-1} (\widehat{M}_2)^T,
\end{equation}
\begin{equation}\label{eq:M2bis}
 \widehat{M}_2'=  \widehat{M}_2  (\widehat{A}'_\nu)^{-1}\widehat{A}''_\nu,\quad \widehat{M}_3'=  \widehat{M}_3  (\widehat{A}'_\nu)^{-1}\widehat{A}''_\nu,\quad \widehat{M}_3''=  (\widehat{A}''_\nu)^T(\widehat{A}'_\nu)^{-T}\widehat{M}_3  (\widehat{A}'_\nu)^{-1}\widehat{A}''_\nu\,,
\end{equation}
and 
\begin{equation}\label{1511181747}
\rk (A_\nu + M) = \rk (A_\nu - M) =3\,.
\end{equation}
\end{subequations}
\end{theorem}

We now subdivide the proof into several technical lemmas.

\begin{lemma}\label{le:1611180036}
We have
\begin{equation}\label{1611180038}
\Ker A_\nu = \Ker(A_\nu + M) \cap \Ker(A_\nu-M)\,.
\end{equation}
In addition,
\begin{equation}\label{1611180040}
\Ker(A'_\nu | A''_\nu)  \subset  \Ker(M_2 | M_2') \cap \Ker(M_3 | M_3') \cap \Ker\left((M'_3)^T | M''_3\right).
\end{equation}
\end{lemma}

\begin{proof}

\medskip

As for \eqref{1611180038}, $\Ker A_\nu \subset \Ker M$  readily gives that $\Ker A_\nu \subset \Ker(A_\nu + M) \cap \Ker(A_\nu-M)$. On the other hand, if $U \in  \Ker (A_\nu + M) \cap \Ker (A_\nu - M)$, then $(A_\nu + M)U=(A_\nu - M)U=0$, so summing up both equalities yields $A_\nu\, U=0$, hence $U \in \Ker A_\nu$. 

\medskip

Finally, we observe that if $U \in \Ker A_\nu$, then 
$$\begin{pmatrix}U_1\\U_2\\U_3\end{pmatrix} \in \Ker \begin{pmatrix}(A'_\nu)^T\\(A''_\nu)^T\end{pmatrix}$$
which implies that $U_1=U_2=U_3=0$ according to \eqref{1311181051} and the Rank Theorem. Thus $\Ker A_\nu =\{\OO_3\}\times \Ker (A'_\nu|A''_\nu)$ and thanks to the decomposition  \eqref{0310180617} of the matrix $M$, the inclusion \eqref{1611180040} follows from the second condition in \eqref{eq:M}.
\end{proof}

\begin{remark}
Notice that \eqref{eqs:1310191917} hold with $A'_\nu$, $A''_\nu$, $M_1$, $M_2$, $M_3$ in place of $\widehat{A}'_\nu$, $\widehat{A}''_\nu$, $\widehat{M}_1$, $\widehat{M}_2$, $\widehat{M}_3$ if and only if $\nu_i \neq 0$ for all $i \in \{1,2,3\}$. Moreover, the admissibility conditions \eqref{eq:M} are invariant if both $\widehat{A}_\nu$ and $\widehat{M}$ replace $A_\nu$ and $M$.
This allows us, in order to ease the reading, to argue for a fixed $\nu \in \mathbb{S}^2$ with $\nu_i \neq 0$ for all $i \in \{1,2,3\}$, as we do in the following. The case with general $\nu$ is readily obtained by replacing $A'_\nu$, $A''_\nu$, $M_1$, $M_2$, $M_2'$, $M_3$, $M_3'$, $M_3''$ by $\widehat{A}'_\nu$, $\widehat{A}''_\nu$, $\widehat{M}_1$, $\widehat{M}_2$, $\widehat{M}_2'$, $\widehat{M}_3$, $\widehat{M}_3'$, $\widehat{M}_3''$.
\end{remark}

From now on we assume that $\nu_i \neq 0$ for all $i \in \{1,2,3\}$. Since, by \eqref{1310191933}, $\rk A'_\nu=3$, there holds
\begin{equation}\label{eq:KerAnu}
\Ker A_\nu={\rm Vect} \left\{
\begin{pmatrix}
0_{\R^3}\\ (A'_\nu)^{-1} (A''_\nu)^{(j)} \\ -e_j
\end{pmatrix} \right\}_{1 \leq j \leq 3},
\end{equation}
where  $E^{(j)}$ denotes the $j$-th column of a matrix $E$  and $e_j$ 
is the $j$-th vector of the canonical basis of $\R^3$. 
Using that $\Ker A_\nu \subset \Ker M$, we get that 
\begin{equation}\label{eq:M2'}
 M_2'=  M_2  (A'_\nu)^{-1}A''_\nu,\quad M_3'=  M_3  (A'_\nu)^{-1}A''_\nu,\quad M_3''=  (M_3')^T  (A'_\nu)^{-1}A''_\nu\,.
 \end{equation}
As a consequence, the matrix $M$ will be entirely determined by the submatrices $M_1$, $M_2$ and $M_3$.

\begin{lemma}\label{lem:3}
Let $\nu \in \mathbb S^2$ be such that $\nu_i \neq 0$ for all $i \in \{1,2,3\}$. The matrices $M_1$ and $M_3$ are nonnegative,
\begin{equation}\label{0210182243}
 \rk M_1 = \rk M_3 =3,
 \end{equation}
and
 \begin{equation}\label{1511181747bis}
\rk (A_\nu + M) = \rk (A_\nu - M) =3\,.
\end{equation}
\end{lemma}

\begin{proof}
By the first condition in \eqref{eq:M}, since $M$ is non negative, we have that $M_1$ and $M_3$ are non negative. Next we only prove that $\rk M_3=3$ since the argument for $M_1$ is similar. Assume by contradiction that $\rk M_3 = m < 3$. 
\medskip
\paragraph{\it Step 1.} 
Let us show that
\begin{equation}\label{0210182246}
\rk  \begin{pmatrix}
M_2 \\ M_3
\end{pmatrix} = m \,.
\end{equation}
We first notice that 
$$\Ker \begin{pmatrix}
M_2\\ M_3
\end{pmatrix}
=\Ker M_2 \cap \Ker M_3 \subset \Ker M_3\,.$$
On the other hand, if $x\in \Ker M_3 \sm \Ker M_2$, then for any $c \in \R$
\begin{equation*}
M \begin{pmatrix}
M_2 x \\ c\, x \\ 0 
\end{pmatrix} \cdot \begin{pmatrix}
M_2 x \\ c\, x \\ 0 
\end{pmatrix} = M_1 (M_2 x) \cdot M_2 x + 2 c |M_2 x|^2  \,,
\end{equation*}
which tends to $-\infty$ as $c\to -\infty$, in contradiction with the nonnegativity of $M$. This implies that $\Ker M_3 \subset \Ker M_2$, hence
\begin{equation}\label{eq:KerM3}
\Ker M_3 = \Ker M_2 \cap \Ker M_3 = \Ker \begin{pmatrix}
M_2\\ M_3
\end{pmatrix},
\end{equation}
and thus, the Rank Theorem gives \eqref{0210182246}. 

\medskip

\paragraph{\it Step 2.} Let us prove that
\begin{equation}\label{0210182322}
\rk \begin{pmatrix}
M_2 \pm A_\nu' \\ M_3
\end{pmatrix} = 3 \,.
\end{equation}
Indeed, assume that there exists $x \in \R^3$ with $x \neq 0$ and 
$$x \in \Ker\begin{pmatrix}
M_2 \pm A_\nu' \\ M_3
\end{pmatrix}\,.$$
In particular, using \eqref{eq:KerM3},
$$x \in \Ker M_3 = 
\Ker \begin{pmatrix}
M_2\\ M_3
\end{pmatrix},$$
hence $x \in \Ker A'_\nu$, which is impossible since $A'_\nu$ is invertible. This yields
$$\Ker \begin{pmatrix}
M_2 \pm A_\nu' \\ M_3
\end{pmatrix}=\{0\},$$
and thus \eqref{0210182322} holds according to the Rank Theorem.

\medskip

\paragraph{\it Step 3} Let us show that
\begin{equation}\label{2210181220}
\rk (A_\nu \pm M) \geq 6-m\,.
\end{equation}

We denote by
$$A_\pm:=\begin{pmatrix}
M_1 &  M_2 \pm A_\nu'  \\
(M_2 \pm A_\nu')^T & M_3
\end{pmatrix} \in \mathbb M^6_{\rm sym},$$
and we observe that $\rk (A_\nu \pm M) \geq \rk A_\pm$. By \eqref{0210182322},  the last three lines of $A_\pm$ are linearly independent. Moreover, since $\rk M_3 =m$, using again \eqref{0210182322} we get that at least $3-m$ lines of $M_2 \pm A_\nu'$ are not contained in $\Ima M_3$. 
We denote their indices by $j_1, \dots, j_{3-m} \in \{1,2,3\}$. Since the lines with indices $j_1, \dots, j_{3-m}$ of $M_2 \pm A_\nu'$ are not a linear combination of those of $M_3$, \emph{a fortiori} the lines with indices $j_1, \dots, j_{3-m}$ of $A_\pm$ are not a linear combination of the last three lines of $A_\pm$. We deduce that the lines of $A_\pm$ with indices $j_1, \dots, j_{3-m}$, $4$, $5$ and $6$  are linearly independent, and thus that $\rk A_\pm \geq 3+(3-m)$.

\medskip

\paragraph{\it Step 4} We now conclude the proof of \eqref{0210182243}.  By Lemma~\ref{le:1611180036}, we have that
\[\dim \Big(\Ker(A_\nu+M) \cap \Ker(A_\nu-M) \Big)=\dim\Ker A_\nu = 3\,,
\]
while \eqref{2210181220} together with the Rank Theorem imply that
$$\dime (\Ker (A_\nu \pm M)) =9-\rk(A_\nu \pm M) \leq 3+m\,.$$
Then, we deduce that
\begin{eqnarray*}
\dime (\Ker (A_\nu + M) + \Ker (A_\nu - M)) & =& \dime (\Ker  (A_\nu + M)) + \dime (\Ker  (A_\nu - M)) \\
&&- \dime (\Ker  (A_\nu + M) \cap \Ker  (A_\nu - M) ) \\ 
& \leq & 2(3+m)-3 < 9 \,,
\end{eqnarray*}
since $m<3$, which is against the third condition in \eqref{eq:M}. This gives $m=3$, and then completes the proof of \eqref{0210182243} for $M_3$. The same holds for $M_1$.

\medskip

\paragraph{\it Step 5} From the previous step, we infer that 
$$\begin{cases}
\dim (\Ker (A_\nu \pm M))\leq 6,\\
\dim (\Ker  (A_\nu + M)) + \dime (\Ker  (A_\nu - M)) - \dime (\Ker  (A_\nu + M) \cap \Ker  (A_\nu - M) )=9,\\
\dim \big(\Ker(A_\nu+M) \cap \Ker(A_\nu-M) \big)=3.
\end{cases}$$
Thus, we deduce that $\dim (\Ker  (A_\nu + M)) + \dime (\Ker  (A_\nu - M)) =12$, 
hence $\dime (\Ker (A_\nu \pm M))= 6$, leading to \eqref{1511181747bis} owing to the Rank Theorem.
\end{proof}

\begin{lemma}\label{le:1611180107}
Let $\nu \in \mathbb S^2$ be such that $\nu_i \neq 0$ for all $i \in \{1,2,3\}$. Then, $(A'_\nu)^T\, M_1^{-1} M_2\in \mathbb M^3_{\rm skew}$, $A'_\nu \, M_3^{-1} (M_2)^T \in \mathbb M^3_{\rm skew}$ and we have
\begin{eqnarray*}
M_3  & = & (A'_\nu)^T \, M_1^{-1} A'_\nu + (M_2)^T M_1^{-1} M_2,\\
M_1 & = & A'_\nu \, M_3^{-1} (A'_\nu)^T +  M_2 M_3^{-1} (M_2)^T.
\end{eqnarray*}
\end{lemma}

\begin{proof}
Since $\rk(A_\nu \pm M)=3$ and $M_1$ is invertible, then each column of $A_\nu\pm M$ is generated by the first three columns of $A_\nu\pm M$. Thus 
$$\begin{pmatrix}
 M_2 \pm A_\nu' \\ M_3
\end{pmatrix}=
\begin{pmatrix}
M_1 \\ (M_2 \pm A_\nu')^T
\end{pmatrix}  M_1^{-1}(M_2 \pm A_\nu'),$$
and we obtain that
\begin{eqnarray*}
M_3 & = & (M_2\pm A'_\nu)^T M_1^{-1} (M_2\pm A'_\nu)\\
& = & (A_\nu')^T \, M_1^{-1} A_\nu' + (M_2)^T M_1^{-1} M_2 \pm (M_2)^T M_1^{-1} A_\nu' \pm (A_\nu')^T\, M_1^{-1} M_2.
\end{eqnarray*}
Note that in particular
\begin{equation}\label{eq:det}
 \rk  (M_2\pm A'_\nu)=3.
\end{equation}
Similarly, using that $M_3$ is invertible, we have
$$M_1 = A'_\nu \, M_3^{-1} (A'_\nu)^T + M_2 M_3^{-1} (M_2)^T \pm M_2 M_3^{-1} (A'_\nu)^T \pm A'_\nu\, M_3^{-1} (M_2)^T .$$
These conditions are equivalent to those in the statement of the lemma. 
\end{proof}

We are now in position to complete the proof of Theorem~\ref{teo:1511181854}. 

\begin{proof}[Proof of Theorem \ref{teo:1511181854}]
Let us start by proving the necessary condition. We consider a matrix $M \in \mathbb M^9_{\rm sym}$ satisfying \eqref{eq:M}. By Lemma \ref{lem:3}, the matrices $M_1$ and $M_3$ are nonnegative and invertible so that \eqref{1511181908} holds. Next \eqref{eq:skew} and \eqref{eq:M1M3} have been proved in Lemma \ref{le:1611180107}, while  \eqref{eq:M2bis} is a consequence of \eqref{eq:M2'}. Finally, according to Lemma \ref{lem:3}, conditions \eqref{1511181747} are satisfied. 

\medskip

We now prove the sufficient condition. By \eqref{eq:M2bis} and the expression \eqref{eq:KerAnu} of $\Ker A_\nu$, we have that the second condition $\Ker A_\nu \subset \Ker M$ of \eqref{eq:M} is satisfied, from which we deduce that
$$\Ker (A_\nu + M) \cap \Ker (A_\nu - M) = \Ker A_\nu.$$
Since by \eqref{1511181747}, we have $\rk(A_\nu \pm M)=3$, then $\dime \Ker(A_\nu \pm M) = 6$, and
\begin{eqnarray*}
\dime \big(\Ker (A_\nu + M) + \Ker (A_\nu - M)\big)  &= & \dime \big(\Ker  (A_\nu + M)\big) + \dime \big(\Ker (A_\nu - M)\big)\\
&&- \dime \big(\Ker  (A_\nu + M) \cap \Ker  (A_\nu - M) \big)  = 9\,,
\end{eqnarray*}
so that the last condition $\Ker (A_\nu + M) + \Ker (A_\nu - M)=\R^9$ in \eqref{eq:M} holds. 
It lasts to show that $M$ is non-negative. We write any $v\in \R^9$ as $v=(v_1, v_2, v_3)$, with $v_1$,  $v_2$,  and $v_3 \in  \R^3$.  With this notation,  and  using \eqref{eq:M2bis}, we have  that 
\begin{multline*}
M v \cdot v = M_1 v_1 \cdot v_1 + M_3 v_2 \cdot v_2+ \big(M_3 (A'_\nu)^{-1}A''_\nu v_3\big)\cdot \big((A'_\nu)^{-1}A''_\nu v_3\big)\\
+ 2 M_2 v_2 \cdot v_1  + 2 \big(M_2 (A'_\nu)^{-1}A''_\nu v_3\big) \cdot v_1 + 2 \big(M_3 (A'_\nu)^{-1}A''_\nu v_3\big) \cdot v_2\,.\end{multline*}
According to \eqref{eq:M1M3}, we thus deduce that
\begin{eqnarray*}
M v \cdot v & =  & M_1 v_1 \cdot v_1 + (M_1^{-1}A'_\nu v_2) \cdot (A'_\nu v_2) + (M_1^{-1}M_2 v_2) \cdot (M_2 v_2)\\
&&+ \big(M_1^{-1}A''_\nu v_3\big)\cdot \big(A''_\nu v_3\big) + \big(M_1^{-1}M_2 (A'_\nu)^{-1}A''_\nu v_3\big)\cdot \big(M_2(A'_\nu)^{-1}A''_\nu v_3\big)\\
&&+ 2 (M_1^{-1} A''_\nu v_3) \cdot (A'_\nu v_2) + 2(M_1^{-1} M_2 (A'_\nu)^{-1} A''_\nu v_3) \cdot (M_2 v_2)\\
&&+ 2 M_2 v_2 \cdot v_1  + 2 (M_2 (A'_\nu)^{-1}A''_\nu v_3) \cdot v_1\\
& = &  \big(M_1^{-1} (M_1 v_1 + M_2 v_2) \big)\cdot (M_1 v_1 + M_2 v_2)\\
&&+ \big(M_1^{-1} A'_\nu(v_2 + (A'_\nu)^{-1}A''_\nu v_3)\big)\cdot \big(A'_\nu(v_2 + (A'_\nu)^{-1}A''_\nu v_3)\big) \\
&& + 2 (M_1^{-1} M_2 (A'_\nu)^{-1}A''_\nu  v_3) \cdot (M_1 v_1 + M_2 v_2)\\
&& + (M_1^{-1}M_2 (A'_\nu)^{-1}A''_\nu  v_3) \cdot (M_2 (A'_\nu)^{-1}A''_\nu  v_3) \\
& = & \big(M_1^{-1} (M_1 v_1 + M_2 v_2 + M_2 (A'_\nu)^{-1}A''_\nu  v_3)\big) \cdot (M_1 v_1 + M_2 v_2 + M_2 (A'_\nu)^{-1}A''_\nu  v_3) \\
&&+ \big(M_1^{-1} A'_\nu(v_2 +(A'_\nu)^{-1}A''_\nu v_3)\big)\cdot (A'_\nu(v_2 + (A'_\nu)^{-1}A''_\nu v_3)) \geq 0\,,
\end{eqnarray*}
where we have used that $M_1^{-1}$ is nonnegative in the last inequality. Therefore $M \in \Mnn$ is non negative, and then it satisfies all conditions in \eqref{eq:M}.
\end{proof}

\subsubsection{Derivation of all dissipative boundary conditions}

We are now in the position to write explicitly all admissible dissipative boundary conditions. Let us introduce the following notation.
\begin{definition}\label{def:tilde}
Given a matrix $\sigma \in \mathbb M^3$ (non necessarily symmetric), we associate to $\sigma$ the vector
$$ \sigma_{\rm pr}=(\sigma_{11},\sigma_{22},\sigma_{33},\sigma_{12},\sigma_{13},\sigma_{23})^T \in \R^6.$$
\end{definition}

\begin{remark}\label{rem:0611191250}
Let $\sigma \in \mathbb M^3$ and $\sigma_{\rm pr} \in \R^6$ be the associated vector given by Definition \ref{def:tilde}. We further decompose $\sigma_{\rm pr}$ as $\sigma_{\rm pr}=(\sigma'_{\rm pr},\sigma''_{\rm pr})^T$ where
$\sigma'_{\rm pr}=  (\sigma_{11} ,\sigma_{22} , \sigma_{33})^T \in \R^3$ and $\sigma''_{\rm pr}=(\sigma_{12} , \sigma_{13}  , \sigma_{23})^T \in \R^3$. An immediate computation shows that
\begin{equation}\label{0311190740}
B'_\nu \sigma'_{\rm pr}+B''_\nu \sigma''_{\rm pr}=-\sigma_{\rm sym}\nu
\end{equation}
where $\sigma_{\rm sym}\in \mathbb M^3_{\rm sym}$ is defined by
$$\sigma_{\rm sym}
:=\begin{pmatrix}
\sigma_{11} & \sigma_{12} & \sigma_{13}\\
\sigma_{12} & \sigma_{22} & \sigma_{23}\\
\sigma_{13} & \sigma_{23} & \sigma_{33}
\end{pmatrix}.$$
Note that $\sigma_{\rm sym}$ differs from the symmetric part of $\sigma$ defined by $(\sigma+\sigma^T)/2$. In addition, if $\sigma$ is already symmetric, then $\sigma_{\rm sym}=\sigma$.
\end{remark}

\begin{proposition}\label{prop:bdcond}
Let $\nu \in \mathbb S^2$  and $A_\nu \in \mathbb M^9_{\rm sym}$ be the matrix given by \eqref{0210181955}. Then $M\in \mathbb M^9_{\rm sym}$  is 
a boundary matrix satisfying properties \eqref{eq:M} if and only if there exist matrices $S_1 \in \mathbb M^{3,+}_{\rm sym}$, $S_2 \in \mathbb M^3_{\rm skew}$ such that both $S_1 \pm S_2$ are invertible and the following property holds: for any $(v,\sigma) \in \R^3\times\mathbb M^3_{\rm sym}$, defining
$$U:=\wt A_0^{\frac12}\begin{pmatrix} v \\ \sigma_{\rm pr} \end{pmatrix} \in \R^9,$$
then
\begin{equation}\label{0311190747}
(A_\nu \pm M)U=0 \quad \Longleftrightarrow \quad (S_1\pm S_2) v \mp \sigma\nu=0\,.
\end{equation}
\end{proposition}

\begin{proof}
According to \eqref{1611180941}, it holds that
$$U=
\begin{pmatrix}
v\\
\wt C^{-\frac12} \sigma'_{\rm pr}\\
\mu^{-\frac12} \sigma''_{\rm pr}
\end{pmatrix},$$
where
$\sigma'=  (\sigma_{11} ,\sigma_{22} , \sigma_{33})^T \in \R^3$ and $\sigma''=(\sigma_{12} , \sigma_{13}  , \sigma_{23})^T \in \R^3$ so that $\sigma_{\rm pr}=(\sigma'_{\rm pr},\sigma''_{\rm pr})^T$. 

\medskip

{\bf Step 1.} We claim that (recall \eqref{1310191911})
\begin{equation}\label{1310192026}
(A_\nu\pm M)U=0 \quad \Longleftrightarrow \quad \widehat{A}'_\nu \widehat{M}_3^{-1}(\widehat{A}'_\nu\pm \widehat{M}_2)^Tv\mp \sigma \nu=0\,.
\end{equation} 
We prove \eqref{1310192026} first in the case where $\nu_i \neq 0$ for all $i \in \{1,2,3\}$ (where $A'_\nu$, $A''_\nu$, $M_1$, $M_2$, $M_3$ coincide with $\widehat{A}'_\nu$, $\widehat{A}''_\nu$, $\widehat{M}_1$, $\widehat{M}_2$, $\widehat{M}_3$).  Using \eqref{0210181955} and \eqref{0310180617}, we have
$$A_\nu\pm M=
\begin{pmatrix}
\pm M_1 & A'_\nu \pm M_2 & A''_\nu\pm M'_2\\
(A'_\nu\pm M_2)^T & \pm M_3 & \pm M'_3\\
(A''_\nu\pm M'_2)^T & \pm (M'_3)^T & \pm M''_3
\end{pmatrix}.$$
By  Theorem~\ref{teo:1511181854},  we have $\rk(A_\nu\pm M)=3$ and $\rk (M_3)=3$. Therefore $\Ima(A_\nu\pm M)$ is generated by the lines $4$, $5$ and $6$ of $A_\nu\pm M$. Thus, 
$(A_\nu\pm M)U=0$ if and only if
$$(A'_\nu\pm M_2)^T v\pm M_3 \wt C^{-\frac12} \sigma'_{\rm pr}\pm \mu^{-\frac12}M'_3 \sigma''_{\rm pr}=0,$$
or still, using condition \eqref{eq:M2bis} of  Theorem~\ref{teo:1511181854}, 
$$(A'_\nu\pm M_2)^T v\pm M_3  \wt C^{-\frac12}\sigma'_{\rm pr}\pm  \mu^{-\frac12}M_3(A'_\nu)^{-1}A''_\nu \sigma''_{\rm pr}=0.$$
Since $M_3$ is invertible, using \eqref{Aprime} this last equation is again equivalent to
$$A'_\nu M_3^{-1}(A'_\nu\pm M_2)^Tv\pm B'_\nu \sigma'_{\rm pr}\pm B''_\nu \sigma''_{\rm pr}=0\,,$$
and then, by \eqref{0311190740}, to
$$A'_\nu M_3^{-1}(A'_\nu\pm M_2)^T v  \mp \sigma\nu=0.$$

Let us now address the case where $\nu$ has at least a null component.  Recalling $C_\nu$ introduced in \eqref{1310191900}, we get that
\begin{eqnarray}\label{1310192113}
(A_\nu \pm M)U=0 & \Longleftrightarrow &  C_\nu (A_\nu \pm M)C_\nu^T C_\nu^{-T} U=0\nonumber\\
& \Longleftrightarrow &  (\widehat A_\nu \pm \widehat M)\widehat U=0
\end{eqnarray}
where $\widehat U=C_\nu^{-T} U$. Since $\rk(\widehat A_\nu \pm \widehat M)=3$ and 
$$\rk \begin{pmatrix}
\widehat M_2 \pm \widehat A_\nu' \\ \widehat M_3
\end{pmatrix} = 3 \,,$$
then the condition \eqref{1310192113} is equivalent to
$$ \big((\widehat{A}'_\nu\pm \widehat{M}_2\big)^T  |   \pm \widehat M_3|   \pm \widehat M'_3\big) \widehat U=0\,,$$
or still, multiplying on the left by $\widehat{A}'_\nu \widehat{M}_3^{-1}$ and using \eqref{eq:M2bis},  to 
\begin{equation}\label{eq:1130}
\big(\widehat{A}'_\nu \widehat{M}_3^{-1} \big(\widehat{A}'_\nu\pm \widehat{M}_2\big)^T  |   \pm \widehat A_\nu' |   \pm \widehat A_\nu''\big) \widehat U=0\,.
\end{equation}
According to Remark \ref{rem:Cnu}, we have that $\widehat U_i=U_i$ for all $i \in \{1,2,3\}$, while using that
$\widehat A_\nu \widehat U=C_\nu A_\nu U$ gives
$$\widehat A'_\nu 
\begin{pmatrix}
\widehat U_4\\
\widehat U_5\\
\widehat U_6
\end{pmatrix}
+ \widehat A''_\nu 
\begin{pmatrix}
\widehat U_7\\
\widehat U_8\\
\widehat U_9
\end{pmatrix}
= A'_\nu 
\begin{pmatrix}
 U_4\\
 U_5\\
 U_6
\end{pmatrix}+  A''_\nu 
\begin{pmatrix}
 U_7\\
 U_8\\
 U_9
\end{pmatrix}.$$
Hence, \eqref{eq:1130} is equivalent to
$$\Big(\widehat{A}'_\nu \widehat{M}_3^{-1} \big(\widehat{A}'_\nu\pm \widehat{M}_2\big)^T  |   \pm  A_\nu' |   \pm  A_\nu''\Big)  U=0$$
and arguing as above, we get that
$$ \Big(\widehat{A}'_\nu \widehat{M}_3^{-1} \big(\widehat{A}'_\nu\pm \widehat{M}_2\big)^T  |   \pm  A_\nu' |   \pm  A_\nu''\Big)  U =  \widehat{A}'_\nu \widehat{M}_3^{-1}(\widehat{A}'_\nu\pm \widehat{M}_2)^Tv\mp \sigma \nu\,,$$ so that \eqref{1310192026} is proven.

\medskip

{\bf Step 2.}  Given $M$ satisfying \eqref{eq:M},  let us define $S_1:=\widehat{A}'_\nu \widehat{M}_3^{-1} (\widehat{A}'_\nu)^T$ and $S_2:=\widehat{A}'_\nu \widehat{M}_3^{-1} (\widehat{M}_2)^T$.  According to  \eqref{eq:skew}, the matrix $S_2$ is skew symmetric. Moreover, since $\widehat{A}'_\nu$ 
 is  invertible and $\widetilde{M}_3$ is positive definite, we deduce that $S_1$ is positive definite. Finally, using \eqref{eq:det}, we have that $S_1\pm S_2=\widehat{A}'_\nu \widehat{M}_3^{-1}(\widehat{A}'_\nu\pm \widehat{M}_2)^T \in \mathbb M^3$ is invertible. 

Conversely, given $S_1 \in \mathbb M^{3,+}_{\rm sym}$ and $S_2 \in \mathbb M^3_{\rm skew}$, then we define 
$$\widehat{M}_3:=(\widehat A'_\nu)^T (S_1)^{-1} \widehat A'_\nu\,, \quad \widehat{M}_2:=\widehat{M}_3 (\widehat A'_\nu)^{-T} S_2\,,$$ 
$$\widehat{M}_1: = \widehat A'_\nu \, \widehat{M}_3^{-1} (\widehat A'_\nu)^T +  \widehat{M}_2 \widehat{M}_3^{-1} (\widehat{M}_2)^T,$$
$$\widehat{M}_2':=  \widehat{M}_2  (\widehat A'_\nu)^{-1}\widehat A''_\nu,\quad \widehat{M}_3':=  \widehat{M}_3  (\widehat A'_\nu)^{-1}\widehat A''_\nu,\quad \widehat{M}_3'' :=  (\widehat A''_\nu)^T(\widehat A'_\nu)^{-T}\widehat{M}_3  (\widehat A'_\nu)^{-1}\widehat A''_\nu\,,$$
and let  $M$ and $\widehat M \in \mathbb M^9_{\rm sym}$ be defined by  \eqref{0310180617} and \eqref{1310191911}.
From Theorem~\ref{teo:1511181854}, it follows that $M$ satisfies \eqref{eq:M}. With this correspondence between $M$ and $(S_1, S_2)$, \eqref{1310192026} guarantees that \eqref{0311190747} holds true, and this concludes the proof.
\end{proof}

\subsubsection{Final formulation of the models}

We complement the dynamical system of perfect plasticity \eqref{0110181923}--\eqref{eq:initial-cond} with dissipative boundary conditions of the form
\begin{equation}\label{1911181008}
S \dot u+\sigma\nu=0 \quad \text{ on }\partial \Omega \times (0,T)\,,
\end{equation}
for some $S \in \mathbb M^3$ an invertible matrix whose symmetric part $S_1:=(S+S^T)/2 \in \mathbb M^{3,+}_{\rm sym}$ is positive definite.
Indeed, setting $S_2:=-(S-S^T)/2 \in \mathbb M^3_{\rm skew}$, then $S_1-S_2=S$, $S_1+S_2=S^T$ and \eqref{0311190747} with the minus sign provides the correspondence between boundary conditions of different type.
Therefore, the full dynamical system of perfect plasticity reads as
\begin{equation}\label{0611191846}
\begin{dcases}
\ddot{u} - \diver \sigma=f & \text{ in } \Omega \times (0,T)\\
\E u= e+p& \text{ in } \Omega \times (0,T)\\
\sigma= \C e\in \mathbf K& \text{ in } \Omega \times (0,T)\\
\sigma:\dot{p}=H(\dot p) & \text{ in } \Omega \times (0,T)\\
S\dot u+\sigma\nu=0& \text{ on } \partial \Omega \times (0,T)\\
(u(0),\dot u(0),e(0),p(0))=(u_0,v_0,e_0,p_0)& \text{ in } \Omega.
\end{dcases}
\end{equation}

\medskip

The final part of this section is devoted to recast the entropic formulation from \eqref{FS}, according to the correspondence between $U$ and $(\dot{u}, \sigma)$ given in \eqref{0110182012}. We thus fix $\xi \in \R^9$ of the form
$$\xi=\wt A_0^{\frac12} \begin{pmatrix} z\\ \tau_{\rm pr}\end{pmatrix}$$
for some $(z,\tau) \in \R^3 \times  \mathbb M^3_{\rm sym}$.

\medskip

The following remark gathers some further properties of the matrices $A_i$ and $A_\nu$.

\begin{remark}
We observe three facts, following from direct computations. First, for all $X \in \R^3$, we have that
\begin{equation}\label{eq:usefull-formula}
\begin{cases}
|U-\xi|^2=|\dot u-z|^2 + \mathbf A^{-1}(\sigma-\tau):(\sigma-\tau)\,, \\
\displaystyle \sum_{i=1}^3 X_i A_i (U-\xi)\cdot (U-\xi)=-2\big( (\sigma-\tau) X\big) \cdot (\dot u-z)\,.
 \end{cases}
 \end{equation}
Second, since
$$\mathbf{A}(w \odot \nu) \nu= \Big(\mu \I_3 + (\mu + \lambda) (\nu \otimes \nu) \Big) w \quad \text{for every }w \in \R^3\,,$$
and $\mu \I_3 + (\mu + \lambda) (\nu \otimes \nu)$ is invertible when $\mu>0$ and $3\lambda+2\mu>0$, then there exists a linear map $\Phi\colon \R^3 \to \R^3$ such that
\begin{equation*}
\Phi(\eta):=\Big(\mu \I_3 + (\mu + \lambda) (\nu \otimes \nu) \Big)^{-1} \eta \quad\text{ satisfies }\quad \mathbf{A}(\Phi(\eta) \odot \nu) \nu=\eta \text{ for any }\eta \in \R^3\,.
\end{equation*}
Third, for any vector $\theta \in \R^9$ of the form $\theta=\widetilde{A}_0^{\frac12}(\theta_1 | \theta_2)^T$ with $\theta_1 \in \R^3$, $\theta_2 \in \R^6$, we have by \eqref{0311190740} and \eqref{2210181057},
\begin{equation*}
A_\nu \theta= \widetilde{A}_0^{-\frac12} \Big(\sum_{i=1}^3\widetilde{A}_i \nu_i \Big) (\theta_1 | \theta_2)^T= \widetilde{A}_0^{\frac12}\Big( -(\theta_2)_{\rm sym} \nu | \big(\mathbf{A} (\theta_1 \odot \nu)\big)_{\rm pr}  \Big)^T\,,
\end{equation*}
and we obtain that
\begin{equation*}
\Ima A_\nu=\Big\{ \widetilde{A}_0^{\frac12} \Big(w_1 | \big( \mathbf{A} (w_2 \odot \nu) \big)_{\rm pr} \Big)^T \colon w_1, \, w_2 \in \R^3 \Big\}\,.
\end{equation*}
\end{remark}

We now have at our disposal all the elements to determine the projections $\xi_0$ (resp. $\xi^\pm$) of $\xi$ onto $\Ker A_\nu$ (resp. $\Ker(A_{\nu}\pm {M})\cap \Ima A_{\nu} $).
Let us assume that $S$ is a symmetric matrix, which is the case where the variational approach can be developed, see Remark~\ref{rem:Ssym} below. 
Then, in view of Remark~\ref{rem:0611191250} and of Proposition~\ref{prop:bdcond}, we obtain that
\begin{equation*}
\xi_0=\frac12 \wt A_0^{\frac12}
\begin{pmatrix}
0_{\R^3}\\
\big( \tau-\mathbf{A}(\Phi(\tau \nu) \odot \nu) \big)_{\rm pr}
\end{pmatrix},
\qquad
\xi^\pm=\frac12 \wt A_0^{\frac12}
\begin{pmatrix}
z \pm S^{-1}\tau\nu\\
\big( \mathbf{A}(\Phi(\tau\nu \pm Sz)\odot \nu) \big)_{\rm pr}
\end{pmatrix}.
\end{equation*}
This finally gives, with \eqref{eq:usefull-formula}, that
\begin{equation}\label{eq:usefull-formula2}
M\xi^\pm\cdot \xi^\pm=\pm\frac12 S^{-1}(\tau\nu\pm S z)\cdot(\tau\nu\pm S z)\,.
\end{equation}

Therefore, setting $v:=\dot u$ and using \eqref{eq:usefull-formula}--\eqref{eq:usefull-formula2},  \eqref{FS} is equivalent to the following family of inequalities:  for all constant vectors $(z,\tau) \in \R^3 \times \mathbf K$ and all test functions $\varphi \in C^\infty_c(\R^3 \times \R)$ with $\varphi \geq 0$ it holds that
\begin{multline*}
\int_{\R^+}\int_\Omega  \left(  |v-z|^2 +\mathbf A^{-1}(\sigma-\tau): (\sigma-\tau)\right) \partial_t \varphi\dx \,\de t-2\int_{\R^+}\int_\Omega (\sigma -\tau):((v-z)\odot \nabla \varphi )\dx\, \de t \\
+\int_\Omega  \left(  |v_0-z|^2 +\mathbf A^{-1}(\sigma_0-\tau): (\sigma_0-\tau)\right) \varphi(0) \dx\\
+ 2 \int_{\R^+}\int_\Omega f\cdot(v-z) \varphi \dx\, \de t+\frac12\int_{\R^+}\int_\dom   S^{-1}(\tau\nu+Sz)\cdot(\tau\nu+Sz)\varphi\, \de \mathcal H^2 \de t\geq 0\,.
\end{multline*}

\section{Variational solutions}

\subsection{The elastoplastic model}

In this section, we study the model \eqref{0611191846}, corresponding to \eqref{0110181923}--\eqref{eq:initial-cond}  complemented  with the dissipative boundary conditions \eqref{1911181008}. Our analysis rests on a variational approach, starting from a visco--elasto--plastic model, for which well-posedness is easily established. Then we study the limit of these viscous solutions as the viscosity parameter becomes vanishingly small. We derive a unique evolution satisfying \eqref{0110181923} (in a suitable sense), and a relaxed version of the boundary conditions \eqref{1911181008}. Indeed the stress contraint $\sigma \in \mathbf K$ and the boundary condition $\sigma\nu+S\dot u=0$ are incompatible since formally $\sigma\nu$ belongs to $\mathbf K\nu$ while $\dot u$ is free. Thus the boundary condition will have to accommodate the constraint by projecting $S\dot u$ onto $-\mathbf K\nu$. From a mathematical point of view, this is related to the lack of continuity of the trace operator in $BD(\Omega)$ with respect to weak* convergence in that space (see \cite{Tem85,Bab}),  as well as the lack of lower semicontinuity of the energy which has to be relaxed, i.e., replaced by its lower semicontinuous envelope.  The relaxation procedure will make appear a new boundary energy related to the previous projection property.

\begin{remark}\label{rem:Ssym}
Note that our variational approach is based on the minimization of an energy functional at the discrete time level, see Subsection~\ref{subsec:visc}. This so-called incremental problem involves various energies (the kinetic and elastic energies, and the plastic dissipation) which are  convex, except a boundary energy of the form
$$u \mapsto \int_\dom S u\cdot u\dh$$
which is related to our choice of boundary conditions. Since the variational structure is lost while passing to the limit as the time-step tends to zero, it is necessary the minimality to be equivalent to the Euler-Lagrange equation (in order not to lose information). As usual, this property depends on the convexity of the energy functional and it rests on the symmetry of the matrix $S$. For that reason, we will assume the matrix $S$ to be symmetric (although the algebraic conditions derived in Proposition \ref{prop:bdcond} allow one to include matrices $S$ with a nontrivial skew symmetric part), in  hypothesis $(H_4)$ below. Our variational method does not permit to consider the more general situation of non symmetric matrices $S$.
\end{remark}

Let us now describe precisely the various hypotheses needed. 

\medskip

\noindent $(H_1)$ {\bf Reference configuration.} Let $\Omega \subset \Rn$ is a bounded Lipschitz open set which stands for the reference configuration of a linearly elastic and perfectly plastic body.

\medskip

\noindent $(H_2)$ {\bf Elasticity properties.} We suppose that the material is isotropic, which means that the Hooke's tensor involved in the constitutive law is expressed by
$$\C e =\lambda (\tr e) \, \I_n + 2\mu e \quad \text{ for all }e \in \Mnn,$$
where $\lambda$ and $\mu$ are the Lam\'e coefficients satisfying $\mu>0$ and $ 2\mu+ n\lambda >0$ (ensuring the ellipticity of the tensor $\C$). It generates a quadratic form given by
$$Q(e):=\frac12 \C e:e =\frac{\lambda}{2} (\tr e)^2 +\mu|e|^2\quad \text{ for all }e \in \Mnn.$$

If $e \in L^2(\Omega;\Rn)$, we define the elastic energy by
$$\mathcal Q(e):=\int_\Omega Q(e)\dx.$$

\medskip

\noindent $(H_3)$ {\bf Plastic properties.} We suppose that the stress is constrained to stay inside a fixed closed and convex set $\mathbf K \subset \Mnn$ containing $0$ in its interior. In particular, there exists $r>0$ such that 
\begin{equation}\label{eq:K}
\{\tau\in\Mnn : \; |\tau|\leq r\} \subset \mathbf K.
\end{equation}
The support function $H:\Mnn \to [0,+\infty]$ of $\mathbf K$ is defined by
$$H(q):=\sup_{\sigma \in \mathbf K}\sigma:q \quad \text{ for all }q \in \Mnn.$$
Note that according to \eqref{eq:K}, we get that
\begin{equation}\label{eq:H}
H(q) \geq r|q| \quad \text{ for all }q \in \Mnn.
\end{equation}

If $p \in \mathcal M_b(\Omega;\Mnn)$, we denote the convex function of measure $H(p)$ by
$$H(p):=H\left(\frac{\de p}{\de |p|}\right) |p|,$$
and we define the plastic dissipation by
$$\mathcal H(p):=\int_\Omega H\left(\frac{\de p}{\de |p|}\right)\de |p|.$$

\medskip

\noindent $(H_4)$ {\bf Dissipative boundary condition.} We suppose that the dissipative boundary condition is expressed by means of a boundary matrix 
\begin{equation}\label{1410191016}
S \in L^{\infty}(\dom; \Mnnp)
\end{equation} 
satisfying the following condition: there exists $c>0$ such that for $\hn$-a.e. $x \in \partial\Omega$ and all $z \in \Rn$
\begin{equation}\label{2711182136}
S(x)z\cdot z \geq c|z|^2\,.
\end{equation}

\medskip

 \noindent $(H_5)$ {\bf External forces.} We assume that the body is subjected to external body loads 
$$ f \in H^1(0,T; L^2(\Omega;\Rn))\,.$$

\medskip

\noindent $(H_6)$ {\bf Initial conditions.} Let $u_0 \in H^1(\Omega;\Rn)$, $v_0 \in H^2(\Omega;\Rn)$, $e_0\in  \Lnn$, $p_0 \in  \Lnn$ and $\sigma_0:= \C e_0\in \mathcal{K}(\Omega)$ be such that 
$$\begin{dcases}
\E u_0= e_0+p_0 \quad\text{a.e. in }\Omega\,,\\
S v_0 + \sigma_0 \nu = 0 \quad \hn\text{-a.e.\ on } \dom\,.
\end{dcases}$$

\begin{definition}\label{def:variational}
 Assume that $(H_1)$--$(H_4)$ are satisfied. A triple $(u, e, p)$ is a \emph{variational solution to the dynamic elasto-plastic model} associated to the initial data  $(u_0, v_0, e_0, p_0) \in L^2(\Omega;\Rn) \times L^2(\Omega;\Rn) \times L^2(\Omega;\Mnn) \times L^2(\Omega;\Mnn)$ and the source term $f \in L^\infty(0,T;L^2(\Omega;\Rn))$ if 
$$\begin{dcases}
u \in W^{2, \infty}(0,T; L^2(\Omega;\Rn)) \cap C^{0,1}([0,T]; BD(\Omega))\,,\\
e \in W^{1, \infty}(0,T; \Lnn)\,,\\
p \in C^{0,1}([0,T]; \mathcal M_b(\Omega;\Mnn))\,,
\end{dcases}
$$
$$\sigma:= \C e \in L^\infty(0,T;\Hdiv), \quad \sigma\nu \in L^\infty(0,T;L^2(\dom;\Rn)),$$
 and it satisfies the following properties:
\begin{itemize}
\item[1.] \emph{The initial conditions}:
$$u(0)=u_0, \quad \dot{u}(0)=v_0, \quad e(0)=e_0, \quad p(0)=p_0;$$
\item[2.] \emph{The additive decomposition}: for all $t \in [0,T]$,
$$\E u(t) = e(t) + p(t) \quad \text{in } \Mbn\,;$$
\item[3.] \emph{The equation of motion:}
$$\ddot{u}- \diver \sigma=f \quad \text{in } L^2(0,T; L^2(\Omega;\Rn))\,;$$
\item[4.] \emph{The relaxed dissipative boundary conditions}:
$${\rm P}_{-\mathbf K\nu}(S\dot u)  + \sigma \nu = 0\quad \text{in } L^2(0,T; L^2(\dom; \Rn)\,;$$
\item[5.] \emph{The stress constraint}: for every $t \in [0,T]$,
$$\sigma(t) \in \mathbf K \quad\text{a.e.\ in }\Omega\,;$$
\item[6.] \emph{The flow rule}:  for every $t \in [0,T]$,
$$H(\dot{p}(t)) = [\sigma(t) \colon \dot{p}(t)] \quad\text{in }\mathcal M_b(\Omega)\,.$$
\item[7.] \emph{The energy balance}: for every $t\in [0,T]$, 
\begin{equation}\label{en_balance}
\begin{split}
\frac12 &\int_\Omega |\dot{u}(t)|^2\dx+ \Q(e(t)) + \int_0^t \HH(\dot{p}(s)) \ds 
+   \int_0^t \int_\dom \psi(x, \dot{u}) \dh \ds \\ & + \frac{1}{2} \int_0^t \int_\dom S^{-1}(\sigma \nu) \cdot (\sigma \nu) \dh \ds 
=\frac12 \int_\Omega |v_0|^2\dx + \Q(e_0) + \int_0^t \int_\Omega f \cdot \dot{u} \dx \ds \,,
\end{split}
\end{equation}
where $\psi\colon \partial\Omega \times \R^n \to \Rn$ is defined, for $\hn$-a.e. $x \in \partial\Omega$ and all $z\in \R^n$, by
\begin{equation}\label{defPsi}
 \psi(x,z)  =  \inf_{z' \in \Rn} \left\{ \frac{1}{2} S(x) z' \cdot z' + H((z'-z)\odot \nu(x)) \right\}\,.
\end{equation}
\end{itemize}
\end{definition}

\begin{remark}
In the 
relaxed dissipative 
boundary conditions at point 4.\ above,  ${\rm P}_{-\mathbf K \nu(x)}$ stands for the orthogonal projection in $\R^n$ onto the closed and convex set $-\mathbf K\nu(x)$ in $\Rn$ with respect to the scalar product $\langle u,v\rangle_{S^{-1}(x)}:=S^{-1}(x)u \cdot v$. It precisely states that $\sigma\nu \in \mathbf K \nu$ a.e.\ on $\partial \Omega \times (0,T)$ and that $S\dot u$ has to be projected onto $-\mathbf K\nu$ to let
 the boundary conditions accommodate the stress constraint. 
\end{remark}

\begin{remark}
If the convex set $\mathbf K$ is of the form
\begin{equation}\label{eq:cylincric}
\mathbf K=\left\{\tau \in \Mnn : \; \tau_D:=\tau-\frac{\tr \sigma}{n}{\rm I_n} \in K\right\},
\end{equation}
where $K$ is a compact and convex set in $\mathbb M^n_D$ containing $0$ in its interior (as e.g. for Von Mises or Tresca models), then the plastic strain turns out to be a deviatoric measure, i.e.,
$$p \in C^{0,1}([0,T]; \mathcal M_b(\Omega;\mathbb M^n_D))$$
and the normal component of the dissipative boundary condition reads as
$$(S^{-1}\sigma\nu)\cdot \nu+\dot u\cdot \nu=0 \quad \text{ in }L^2(0,T;L^2(\dom;\Rn)).$$
In other words, since the stress constraint only acts on the deviatoric part of the Cauchy stress, then only shearing stresses are responsible of plastification. It follows that only the tangential component of the boundary condition (the one involving the boundary shearing stress) needs to be relaxed, and the normal component of the boundary condition (involving the boundary hydrostatic pressure) is not affected by the stress constraint.
\end{remark}

The present section is devoted to the proof of the following existence and uniqueness result. 

\begin{theorem}\label{teo:existence_variational}
Under the assumptions $(H_1)$--$(H_6)$, there exists a unique variational solution to the elasto--plastic problem associated to the initial data $(u_0, v_0, e_0, p_0)$ and the source term $f$ according to Definition~\ref{def:variational}.
\end{theorem}

The result is obtained starting from a elasto--visco--plastic approximation, introduced in Subsection~\ref{subsec:visc}.  Subsection~\ref{subsec:4.3} collects some preliminary tools, in view of the existence and uniqueness proof, performed in Subsection~\ref{subsec:4.4}. 
 
\subsection{The visco-elastoplastic model}\label{subsec:visc}

The visco-elastoplastic regularization consists in adding a visco--elastic diffusion term of Kelvin--Voigt type in the constitutive law, and a visco--plastic term of Perzyna type in the plastic flow rule penalizing the stress constraint. This type of models has been considered
\begin{itemize}
\item in \cite{Suq81,BabMor15,DMSca} for variational models as in \eqref{0110181923} but with different boundary conditions; 
\item in \cite{BabMifSeg16, DepLagSeg11} for constrained Friedrichs' systems in the whole space;
\item in \cite{BM17} for a simplified antiplane elastoplastic model with dissipative boundary conditions.
\end{itemize} 
One deduces the following existence and uniqueness result arguing exactly as done for the antiplane simplified case of \cite{BM17,BabMor15}.

\begin{proposition}\label{teo:apprvisc}
Assume that $(H_1)$--$(H_6)$ are satisfied. For each $\varepsilon>0$, we define $g_\varepsilon:= \varepsilon E v_0 \, \nu \in L^2(\dom; \Rn)$. Then, there exists a unique triple $(u_\varepsilon, e_\varepsilon, p_\varepsilon)$ with
\begin{equation*}
\begin{dcases}
u_\varepsilon \in W^{2, \infty}(0,T; L^2(\Omega;\Rn)) \cap H^2(0,T; H^1(\Omega;\Rn))\,,\\
e_\varepsilon \in W^{1, \infty}(0,T; L^2(\Omega; \Mnn))\,,\\
p_\varepsilon \in H^1(0,T; L^2(\Omega; \Mnn))\,,
\end{dcases}
\end{equation*}
which satisfies the following properties:
\begin{itemize}
\item[1.] \emph{The initial conditions}:
$$u_\varepsilon(0)=u_0, \quad \dot{u}_\varepsilon(0)=v_0, \quad e_\varepsilon(0)=e_0, \quad p_\varepsilon(0)=p_0;$$
\item[2.] \emph{The additive decomposition}:
$$\E u_\varepsilon = e_\varepsilon + p_\varepsilon\quad \text{a.e. in } \Omega \times (0,T)\,;$$
\item[3.] \emph{The constitutive law}:
$$\sigma_\varepsilon = \C e_\varepsilon \quad \text{a.e. in } \Omega \times (0,T)\,;$$
\item[4.] \emph{The equation of motion}:
$$\ddot{u}_\varepsilon- \diver(\sigma_\varepsilon + \varepsilon \E \dot{u}_\varepsilon)=f \quad \text{in }L^2(0,T; L^2(\Omega;\Rn))\,;$$
\item[5.] \emph{The dissipative boundary conditions}:
$$S \dot{u}_\varepsilon + (\sigma_\varepsilon + \varepsilon \E \dot{u}_\varepsilon) \nu = g_\varepsilon \quad \text{in }L^2(0,T; L^2(\dom; \Rn))\,;$$
\item[6.] \emph{The visco--plastic flow rule}:
$$\dot{p}_\varepsilon=\frac{\sigma_\varepsilon- \mathrm{P}_{\mathbf K}(\sigma_\varepsilon)}{\varepsilon}\quad  \text{a.e. in } \Omega \times (0,T)\,;$$
\item[7.] \emph{The energy balance}: for every $t\in [0,T]$,
\begin{multline}\label{en_balan_visc}
\frac12 \int_\Omega |\dot u_\varepsilon(t)|^2\dx + \Q(e_\varepsilon(t)) + \int_0^t \int_\Omega H(\dot{p}_\varepsilon) \dx\ds\\
 +\int_0^t \int_\dom S \dot{u}_\varepsilon \cdot \dot{u}_\varepsilon \dh \ds
+\varepsilon \int_0^t \int_\Omega  |\E \dot{u}_\varepsilon|^2\dx\ds + \varepsilon \int_0^t \int_\Omega |\dot{p}_\varepsilon|^2 \dx\ds  \\
 =\frac12\int_\Omega |v_0|^2\dx + \Q(e_0) + \int_0^t \int_\Omega f \cdot \dot{u}_\varepsilon \dx \ds + \int_0^t \int_{\dom} g_\varepsilon \cdot \dot{u}_\varepsilon \dh \ds\,.
\end{multline}
\end{itemize}
Moreover, the following uniform estimate holds
\begin{multline}\label{estimate_viscosa}
 \|\ddot{u}_\varepsilon\|_{L^\infty(0,T;L^2(\Omega;\Rn))}^2 {+}\|\dot{e}_\varepsilon\|_{L^\infty(0,T;L^2(\Omega;\Mnn))}^2  \\+\varepsilon \|\E \ddot{u}_\varepsilon\|_{L^2(0,T;L^2(\Omega;\Mnn))}^2 {+} \int_0^T \hspace{-0.5em} \int_{\dom} \hspace{-0.7em} S \ddot{u}_\varepsilon \cdot \ddot{u}_\varepsilon \dh \mathrm{d}t 
 \leq C ,
\end{multline}
for some constant $C>0$ independent of $\varepsilon$ and $S$.
\end{proposition}

We do not give a proof of this result which closely follows the arguments in \cite[Subsection~3.4.1]{MifThese}, to which we refer for details with just some formal modifications to pass from the antiplane case setting to the present general one (see also \cite{BabMor15} in the vectorial case but with different boundary conditions). We just point out the general method. The starting point consists in defining suitable time discrete 
approximating solutions. For any $N \in \mathbb N$, large enough, we consider a discretisation $\{t_i\}_{i=0}^N:=\{i\frac{T}{N}\}_{i=0}^{N}$ of the time interval $[0,T]$, and approximate the values of $(u, e, p)$ at the nodes of the discretisation as follows. Setting $\delta:=\frac{T}{N}$ the time step, we start by the initial values $(u_0, e_0, p_0)$, $(u_1, e_1, p_1):=  (u_0, e_0, p_0) + \delta (v_0, 0, \E u_0)$ and define by induction, for $i\geq 2$ ,
$$(u_i, e_i, p_i) \in X:= \{(v, \eta, q) \in H^1(\Omega;\Rn)\times \Lnn \times \Lnn \colon\; \E v = \eta + q\text{ in }\Omega\}$$
as the unique minimizer over $X$ of
\begin{multline*}
(v,\eta,q)\mapsto\mathcal{Q}(\eta)  {+} \int_\Omega H(q-p_{i-1}) \dx {+} \frac{1}{2 \delta^2} \int_\Omega (v {-} 2 u_{i-1} {+} u_{i-2})^2 \dx \\
{+} \frac{\varepsilon}{2 \delta} \int_\Omega \big( |\E v - \E u_{i-1}|^2 {+} |q-p_{i-1}|^2\big) \dx {+} \frac{1}{2\delta} \int_{\partial \Omega} S (v-u_{i-1}) \cdot (v-u_{i-1}) \dh\\
 {-} \int_\Omega f(t_{i-1}) \cdot v \dx {-} \int_{\partial \Omega} g_\varepsilon \cdot v \dh\,.
\end{multline*}
We use the discrete solutions $\{(u_i, e_i, p_i)\}_{i=0}^N$ to construct $N$-dependent quadratic interpolations $(u_N(t),e_N(t),p_N(t))$, which can be shown to satisfy suitable \emph{a priori} estimates and then converge in a weak sense as $N \to +\infty$ towards the weak notion of evolution. Then it is possible to prove uniqueness, and \emph{a posteriori} regularity estimates that guarantees the properties stated in Proposition~\ref{teo:apprvisc}.

We notice that the matrix $S$ has to be symmetric in order to guarantee the minimality of the incremental minimization problems above to be equivalent to the corresponding Euler--Lagrange equation, which is in turn a key tool, e.g.\ in the derivation of the \emph{a priori} estimates. 

\subsection{Some technical results}\label{subsec:4.3}

In the  sequel, $\nu(x)$ denotes the  outward  unit normal to $\dom$ at $x \in \dom$, which is defined $\hn$-almost everywhere in $\dom$ by the Lipschitz regularity of $\dom$.  Let us define the functions $f$ and $h:\partial \Omega \times \R^n \to \R$ by setting, for $\hn$-a.e. $x \in \dom$ and all $z \in \Rn$, 
$$f(x,z) :=  \frac12 S(x) z\cdot z$$
and
$$h(x,z) :=  H(-z\odot \nu(x))=\sup_{\sigma \in \mathbf K} -\sigma:(z \odot \nu(x))=\sup_{\sigma \in \mathbf K} -(\sigma\nu(x))\cdot z=(\I_{-\mathbf K \nu(x)})^*(z).$$

We next remark that $\psi:\partial \Omega \times \R^n \to \R$, introduced in \eqref{defPsi}, is  the inf-convolution (with respect to the second variable) of the functions $f(x,\bullet)$ and $h(x,\bullet)$ for $\hn$-a.e. $x \in \dom$.

\begin{remark}
By standard properties of the inf-convolution and of the convex conjugate (see e.g.\ \cite[Theorem 16.4]{Rock}), we have
$$\psi(x, \bullet) = f(x, \bullet)  \boxempty h(x, \bullet)  =\Big( f^*(x,\bullet) + h^*(x,\bullet)\Big)^*.$$
As a consequence, for $\hn$-a.e. $x \in \dom$ and all $z \in \Rn$,
\begin{eqnarray*}
\psi(x,z) & = & \sup_{q \in -\mathbf K \nu(x)} \left\{ z \cdot q - \frac12 S^{-1}(x) q \cdot q  \right\} \\
& = &  \frac{1}{2} S(x) z \cdot z - \frac12 \inf_{q \in -\mathbf K \nu(x)}S^{-1}(x) \big(q- S(x) z\big) \cdot \big(q- S(x) z\big) \\
& = &  \frac{1}{2} S(x) z \cdot z - \frac{1}{2} \| \mathrm{P}_{-\mathbf K \nu(x)} \big(S(x)z \big) - S(x) z \|_{S^{-1}(x)}^2\,,
\end{eqnarray*}
where ${\rm P}_{-\mathbf K \nu(x)}$ stands for the orthogonal projection in $\R^n$ onto the closed and convex set $-\mathbf K\nu(x)$ in $\Rn$ with respect to the Euclidean scalar product $\langle u,v\rangle_{S^{-1}(x)}:=S^{-1}(x)u \cdot v$. It follows that, for $\hn$-a.e. $x \in \partial \Omega$, the function $\psi(x,\cdot)$ is of class $C^{1,1}$ in $\Rn$ because
\begin{equation}\label{eq:dzpsi}
D_z \psi(x, z)=  S(x)z -\left(S(x)z-P_{-\mathbf K\nu(x)}(S(x)z)\right) =P_{-\mathbf K\nu(x)}(S(x)z),
\end{equation}
and the projection is Lipschitz continuous.
\end{remark}

In the following lemma, we prove an explicit expression for the element realizing the minimum value in the infimal convolution defining the function $\psi$.

\begin{lemma}\label{lem:infconv}
For $\hn$-a.e. $x \in \dom$ and all $z \in \Rn$, there exists a unique $\bar z=\bar z(x,z) \in \Rn$ such that
$$\psi(x,z)= \frac{1}{2} S(x) (z-\bar z) \cdot (z-\bar z) + h(x,\bar z).$$
In addition, we have $z-\bar z=S^{-1}(x)D_z \psi(x,z)$.
\end{lemma}

\begin{proof}
Since the function $z' \mapsto  \frac{1}{2} S(x) (z-z') \cdot (z-z')+ h(x,z')$ is continuous, coercive and strictly convex, it admits a unique minimum point $\bar z \in \Rn$ which thus satisfies
$$\psi(x,z)= \frac{1}{2} S(x) (z-\bar z) \cdot (z-\bar z) + h(x,\bar z).$$

Let $t>0$ and $\xi \in \Rn$. Since $\psi(x,z-t\xi) \leq  \frac{1}{2} S(x) (z-t\xi-\bar z) \cdot (z-t\xi-\bar z) + h(x,\bar z)$ then
\begin{eqnarray*}
\frac{\psi(x,z)-\psi(x,z-t\xi)}{t} & \geq & \frac{S(x) (z-\bar z) \cdot (z-\bar z) -S(x) (z-t\xi-\bar z) \cdot (z-t\xi-\bar z) }{2t}\\
& = & S(x)(z-\bar z)\cdot \xi - \frac{t}{2} S(x)\xi\cdot \xi.
\end{eqnarray*}
Passing to the limit as $t \to 0$ and using that $\psi(x,\bullet)$ is differentiable, we get that $D_z \psi(x,z)\cdot \xi \geq S(x)(z-\bar z)\cdot \xi$, and since $\xi$ is arbitrary,
$D_z \psi(x,z)=S(x)(z-\bar z),$
which completes the proof.
\end{proof}

From the previous lemma, we deduce the following measurable selection result when the argument of $\psi$ is not anymore a constant but a function defined on $\dom$.

\begin{lemma}\label{lem:meas-select}
For a given $u \in L^2(\dom;\Rn)$, there exists a unique function $w \in L^2(\partial \Omega;\Rn)$ such that, for $\hn$-a.e. $x \in \dom$,
$$\psi(x,u(x)) =\frac{1}{2} S(x) w(x) \cdot w(x) + H\big((w(x)-u(x)) \odot \nu(x)\big) .$$
\end{lemma}

\begin{proof}
For $\hn$-a.e. $x \in \partial \Omega$, let $w(x):=S^{-1}(x)D_z \psi(x,u(x))$. Since $(x,z) \mapsto D_z\psi(x,z)$ is a Carath\'eodory function, we deduce that the function $w:\partial\Omega \to \Rn$ is $\hn$-measurable. In addition, according to Lemma \ref{lem:infconv}, we have
$$\psi(x,u(x)) =\frac{1}{2} S(x) w(x) \cdot w(x) + H\big((w(x)-u(x)) \odot \nu(x)\big).$$
Finally, since $H$ is nonnegative and vanishes in $0$, we have that for $\hn$-a.e. $x \in \partial \Omega$,
$$\frac{c}{2} |w(x)|^2 \leq \frac12 S(x)w(x)\cdot w(x) \leq \psi(x,u(x)) \leq \frac12 S(x)u(x)\cdot u(x)$$
 using \eqref{2711182136},  which shows that $w \in L^2(\partial \Omega;\Rn)$ since $u \in L^2(\partial \Omega;\Rn)$ and by \eqref{1410191016}.
\end{proof}

We now prove that the integral functional associated to the density $\psi$ can be globally expressed as an infimum.

\begin{lemma}\label{lem:inf-ext}
For a given $u \in L^2(\dom;\Rn)$,  it holds that 
$$\int_{\partial \Omega} \psi(x,u)\dh  =  \inf_{z \in C^\infty_c(\Rn;\Rn)}\bigg\{ \int_{\partial \Omega}\left(\frac12 Sz\cdot z + H\big((z-u)\odot\nu\big)\right)\dh\bigg\} .$$
\end{lemma}

\begin{proof}
According to Lemma \ref{lem:meas-select}, we have
$$\int_{\partial \Omega} \psi(x,u)\dh  =  \inf_{z \in L^2(\partial \Omega;\Rn)}\bigg\{ \int_{\partial \Omega}\left(\frac12 Sz\cdot z + H\big((z-u)\odot\nu\big)\right)\dh\bigg\} .$$
By density of $C(\partial \Omega;\Rn)$ in $L^2(\partial \Omega;\Rn)$  and Tietze's extension Theorem, for any $z \in L^2(\partial \Omega;\Rn)$ and $\eta>0$ there exists a function $\tilde z  \in C_c(\Rn;\Rn)$ such that $\|z-\tilde z\|_{L^2(\partial \Omega;\Rn)}\leq\eta/2$. Then approximating by convolution, there exists $\hat z \in C^\infty_c(\Rn;\Rn)$ such that $\|\hat z-z\|_{L^2(\partial \Omega;\Rn)}\leq\eta$. The conclusion then follows from the continuity of the functional
$$z \in L^2(\partial \Omega;\Rn) \mapsto  \int_{\partial \Omega}\left(\frac12 Sz\cdot z + H\big((z-u)\odot\nu\big)\right)\dh$$
with respect to the strong $L^2(\partial \Omega;\Rn)$-convergence.
\end{proof}

In the  following result we establish a convexity inequality which would be easily obtained if we had at our disposal a generalized Green formula, and if we a priori knew that $\sigma\nu \in \mathbf K\nu$ when $\sigma \in \mathcal K(\Omega)$. This result will be essential to obtain the upper energy--dissipation inequality in Section \ref{sec:energy-balance}.

\begin{proposition}\label{prop:IPP}
Let $(u, e, p) \in (BD(\Omega) \cap L^2(\Omega;\Rn)) {\times} \Lnn{\times}\mathcal M_b(\Omega;\Mnn)$ be such that $\HH(p)<+\infty$, and let $\sigma \in \mathcal K(\Omega)$ with $\sigma\nu \in L^2(\dom;\Rn)$. Then for all $\varphi \in C^1_c(\Rn)$ with $\varphi\geq 0$, we have
\begin{multline*}
\int_\Omega \varphi \de H(p)  + \int_{\partial \Omega}\varphi \psi(x,u)\dh+\frac12\int_{\partial \Omega}  \varphi S^{-1}(\sigma\nu)\cdot (\sigma\nu)\dh\\
 \geq-\int_\Omega \varphi \sigma :  e\dx- \int_\Omega \varphi  u \cdot \diver \sigma \dx- \int_\Omega \sigma \colon ( u \odot \nabla \varphi) \dx.
\end{multline*}
\end{proposition}

\begin{proof} Let $z \in C^\infty_c(\Rn;\Rn)$,  and let us  define an extension $q \in \mathcal M_b(\overline \Omega;\Mnn)$ of $p$ by setting
$$q:=p \mres \Omega + (z-u)\odot \nu\hn \mres \dom.$$
Then from \eqref{eq:conv-ineq} and the definition of the stress--strain duality, for all $\varphi \in C^1(\Rn)$ with $\varphi \geq 0$, we have that
\begin{multline*}
\int_\Omega \varphi \de H(p)  + \int_{\partial \Omega}\varphi \left(\frac12  S z\cdot z+ H((z- u)\odot \nu)\right)\dh\\
=\int_{\overline \Omega} \varphi \de H(q)  +\frac12 \int_{\partial \Omega}\varphi   S z\cdot z \dh
\geq \langle [\sigma \colon q],\varphi\rangle +\frac12\int_{\partial \Omega} \varphi  S z\cdot z\dh\\
= -\int_\Omega \varphi \sigma \colon (e -Ez)\dx- \int_\Omega (u-z) \cdot \diver \sigma \, \varphi \dx
 - \int_\Omega \sigma \colon \big((u-z) \odot \nabla \varphi\big) \dx+ \frac12 \int_{\partial \Omega}  \varphi S z\cdot z\dh.
\end{multline*}
Integrating by parts the terms involving $z$ yields
\begin{multline*}
\int_\Omega \varphi \de H(p)  + \int_{\partial \Omega}\varphi \left(\frac12  S z\cdot z+ H((z-u)\odot \nu)\right)\dh\\
\geq-\int_\Omega \varphi \sigma :  e\dx- \int_\Omega \varphi u \cdot \diver \sigma \dx
 - \int_\Omega \sigma \colon (u \odot \nabla \varphi) \dx+\int_{\partial \Omega}  \varphi  \left((\sigma\nu)\cdot z + \frac12 S z\cdot z\right)\dh\\
 \geq-\int_\Omega \varphi \sigma : e\dx- \int_\Omega \varphi u \cdot \diver \sigma \dx - \int_\Omega \sigma \colon (u \odot \nabla \varphi) \dx -\frac12\int_{\partial \Omega}  \varphi S^{-1}(\sigma\nu)\cdot (\sigma\nu)\dh, 
\end{multline*}
where we used that $(\sigma\nu)\cdot z + \frac12 S z\cdot z \geq -\frac12 S^{-1}(\sigma\nu)\cdot (\sigma\nu)$ $\hn$-a.e. on $\partial \Omega$. Passing to the infimum in $z \in C^\infty_c(\Rn;\Rn)$ in the left hand side of the previous inequality, and using Lemma \ref{lem:inf-ext} leads to the desired inequality.
\end{proof}

Our last tool is the following lower bound for the dissipation energy along sequences converging in a natural topology associated to the energy space. This lower bound estimate will be the main tool employed in the lower energy--dissipation inequality in Section \ref{sec:energy-balance}.

\begin{proposition}
Let $(u_k, e_k, p_k)_{k \in \mathbb N}$ be a sequence in $BD(\Omega){\times} \Lnn{\times}\mathcal M_b(\Omega;\Mnn)$  such that $\E u_k = e_k + p_k$ in $\mathcal M_b(\Omega;\Mnn)$, and 
$$\begin{dcases}
u_k \wstar u &\quad\text{weakly* in }BD(\Omega)\,,\\
e_k \weak e &\quad\text{weakly in }\Lnn\,,\\
p_k \wstar p &\quad \text{weakly* in }\mathcal M_b(\Omega;\Mnn)\,,\\
u_k \weak w & \quad \text{weakly in } \Lbn\,,
\end{dcases}
$$
for some $(u, e, p) \in BD(\Omega){\times} \Lnn{\times}\mathcal M_b(\Omega;\Mnn)$ and $w \in L^2(\dom;\Rn)$. Then
\begin{equation}\label{2211182119}
p_k  \wstar p+(w-u) \odot \nu \hn \mres \dom \quad \text{ weakly* in }\mathcal M_b(\overline \Omega;\Mnn)
\end{equation}
and
\begin{equation}\label{2211182148}
\mathcal H(p) + \int_\dom  \psi(x,u) \dh
 \leq \liminf_{k\to \infty} \left\{ \mathcal H(p_k)+ \frac12 \int_{\partial \Omega}  Su_k \cdot u_k \dh \right\}.
\end{equation}
\end{proposition}

\begin{proof}
We first establish \eqref{2211182119}. 
By Green's formula in $BD(\Omega)$ (see \cite[Theorem 3.2]{Bab}), we have that for any $\varphi \in C^1(\R^n;\Mnn)$, 
$$\int_{\Omega} \varphi \colon \mathrm{d}(\E  u_k - \E  u) = - \int_\Omega\diver \varphi \cdot (u_k-u) \dx + \int_\dom \varphi \colon (u_k-u) \odot \nu \dh\,,$$
so that, using the convergences in the hypotheses, 
\begin{equation}\label{2111182226}
\lim_{k\to \infty} \int_{\Omega} \varphi \colon \mathrm{d}(\E  u_k - \E  u) = \int_{\dom}\varphi \colon (w-u) \odot \nu \dh\,.
\end{equation}
By a  density argument, the same convergence holds also for every $\varphi \in C(\overline \Omega;\Mnn)$. Since $e_k \weak e$ in $\Lnn$, then using that $p_k=\E u_k -e_k$ and $p=\E u -e$, and by \eqref{2111182226}, we get that 
$$\lim_{k\to \infty} \int_{\Omega} \varphi \colon \mathrm{d}(p_k - p) = \int_{\dom}\varphi \colon (w-u) \odot \nu \dh\quad\text{for every } \varphi\in C(\overline \Omega;\Mnn)\,,$$
which gives \eqref{2211182119}.

We now prove \eqref{2211182148}. By \eqref{2211182119} and Reshetnyak's lower semicontinuity Theorem we obtain that 
\begin{equation*}
\begin{split}
\liminf_{k\to \infty} &\left\{ \mathcal H(p_k) + \frac12 \int_{\partial \Omega}  Su_k \cdot u_k \dh \right\}\\
&\geq \mathcal H(p)+\int_{\partial \Omega}  H\big((w-u)\odot \nu\big)\dh + \frac12 \int_{\partial\Omega} Sw\cdot w\dh 
\geq\mathcal H(p) +\int_{\partial \Omega} \psi(x,u)\dh\,,
\end{split}
\end{equation*}
 by definition of $\psi$.
\end{proof}

\begin{remark}\label{rem:wnu=unu}
In the case of a cylindrical convex set $\mathbf K$ of the form \eqref{eq:cylincric}, since the plastic strain is deviatoric then
the condition
$$p_k \wstar p \quad \text{weakly* in }\mathcal M_b(\Omega;\mathbb M^n_D)$$
implies that $\tr \big((w-u)\odot \nu \big)=(w-u)\cdot\nu=0$ $\hn$-a.e. on $\partial \Omega$.
\end{remark}

\subsection{Proof of Theorem \ref{teo:existence_variational}}\label{subsec:4.4}

In this subsection we prove Theorem \ref{teo:existence_variational}. Using the energy balance as well as the estimate \eqref{estimate_viscosa}, we deduce weak convergence properties of the various fields.  This  enables one to pass to the limit in the linear relations (the initial condition, the additive decomposition, the constitutive law and the equation of motion) as well as in the convex relation (the stress constraint). Then we improve these weak convergences into strong ones (in particular for the trace of the displacement) thanks to which we can pass to the lower limit in the localized energy balance.  This  gives a first energy inequality which turns out to be an equality.  Specializing   this   energy balance to test functions compactly supported in $\Omega$ we obtain the measure theoretic flow rule. Finally we derive the relaxed dissipative boundary condition thanks to a convexity argument by proving that the limit of the  traces  of $(u_\varepsilon)_{\varepsilon>0}$ is the solution of the minimization problem in the infimal convolution defining $\psi(x,\dot u(t,x))$. Eventually we prove the uniqueness of the solution.
 
\subsubsection{Weak compactness}
For any $\varepsilon>0$, let $(u_\varepsilon, e_\varepsilon, p_\varepsilon)$ be the solution of the elasto--visco--plastic approximation given by Proposition~\ref{teo:apprvisc}.
Thanks to estimates \eqref{estimate_viscosa} in Proposition~\ref{teo:apprvisc}, there exist functions $u\in W^{2, \infty}(0,T; L^2(\Omega;\Rn))$, $e\in W^{1, \infty}(0,T; \Lnn)$, $\sigma \in W^{1, \infty}(0,T; \Lnn)$ and $w \in H^2(0,T; L^2(\dom;\Rn))$ such that, up to a subsequence,
\begin{equation}\label{2811182355}
\begin{dcases}
u_\varepsilon \wstar u &\quad \text{weakly* in }W^{2, \infty}(0,T; L^2(\Omega;\Rn))\,,\\
e_\varepsilon \wstar e &\quad \text{weakly* in } W^{1, \infty}(0,T; \Lnn)\,,\\
\sigma_\varepsilon \wstar \sigma &\quad \text{weakly* in } W^{1, \infty}(0,T; \Lnn)\,,\\
u_\varepsilon \weak w &\quad \text{weakly in }H^2(0,T; L^2(\dom;\Rn))\,,
\end{dcases}
\end{equation}
and, by Ascoli-Arzel\`a Theorem, for every $t\in [0,T]$,
\begin{equation}\label{2811182356}
\begin{dcases}
u_\varepsilon(t) \weak u(t), \quad  \dot{u}_\varepsilon(t) \weak \dot{u}(t) &\quad\text{weakly in }L^2(\Omega;\Rn)\,,\\ 
e_\varepsilon(t) \weak e(t), \quad \hspace{0.3em} \sigma_\varepsilon(t) \weak \sigma(t) &\quad\text{weakly in }\Lnn\,,\\
u_\varepsilon(t) \weak w(t), \quad \dot u_\varepsilon(t) \weak \dot w(t) &\quad \text{weakly in }L^2(\dom;\Rn)\,.
\end{dcases}
\end{equation}
By the energy balance \eqref{en_balan_visc} between two arbitrary times $0\leq t_1 <t_2 \leq T$ we get that
\begin{multline*}
\int_{t_1}^{t_2}  \HH(\dot{p}_\varepsilon(s)) \ds \leq  \frac{1}{2} \left(\|\dot{u}_\varepsilon(t_1)\|_{L^2(\Omega;\Rn)}^2 - \|\dot{u}_\varepsilon(t_2)\|_{L^2(\Omega;\Rn)}^2 \right) + \big( \Q(e_\varepsilon(t_1)) - \Q(e_\varepsilon(t_2)) \big)  \\
 +\int_{t_1}^{t_2} \int_\Omega f \cdot \dot{u}_\varepsilon \dx \ds + \int_{t_1}^{t_2} \int_{\dom} g_\varepsilon \cdot \dot{u}_\varepsilon \dh \ds\,.
\end{multline*}

According to the bounds \eqref{estimate_viscosa}, the right hand side above can be controled by $C(t_2-t_1)$, for some constant $C>0$ independent of $\varepsilon$ and $t_1$, $t_2$. On the other hand, by Jensen's inequality and \eqref{eq:H}, we get that
$$\|p_\varepsilon(t_2)- p_\varepsilon(t_1)\|_{L^1(\Omega;\Mnn)} \leq C(t_2-t_1)\,,$$
so, by Ascoli-Arzel\`a Theorem, there is $p \in C^{0,1}([0,T];\mathcal M_b(\Omega; \Mnn))$ such that, up to a further subsequence (independent of $t$), for every $t \in [0,T]$
\begin{equation}\label{2911180016}
p_\varepsilon(t) \wstar p(t) \quad\text{weakly* in }\mathcal M_b(\Omega;\Mnn)\,.
\end{equation}
By the additive decomposition $\E u_\varepsilon(t) = e_\varepsilon(t) + p_\varepsilon(t)$, and the convergences in \eqref{2811182356} and \eqref{2911180016}, we have that $u\in C^{0,1}([0,T]; BD(\Omega))$, and for every $t\in [0,T]$
$$u_\varepsilon(t) \wstar u(t) \quad\text{weakly* in }BD(\Omega)\,.$$

\medskip

According to the previously established weak convergences, we can pass to the limit in the initial conditions
$$u(0)=u_0, \quad \dot{u}(0)=v_0, \quad e(0)=e_0, \quad p(0)=p_0,$$
the additive decomposition, 
$$\E u(t) = e(t) + p(t) \quad \text{in } \Mbn\,,$$
the constitutive law
$$\sigma =\C e \quad \text{a.e. in } \Omega \times (0,T)\,,$$
and the equation of motion 
$$\ddot{u}- \diver\sigma =f \quad \text{in }L^2(0,T; L^2(\Omega;\Rn))\,,$$
where we used that $\varepsilon \E \dot{u}_\varepsilon \to 0$ strongly in $L^2(0,T;L^2(\Omega;\Mnn))$ in the last relation.
Note that $\sigma_\varepsilon+\varepsilon \E \dot{u}_\varepsilon \weak \sigma$ weakly in $L^2(0,T;L^2(\Omega;\Mnn))$ and $\diver(\sigma_\varepsilon+\varepsilon \E \dot{u}_\varepsilon)=\ddot u_\varepsilon +f\weak \ddot u+f=\diver \sigma$ weakly in $L^2(0,T;L^2(\Omega;\Rn))$ so that  $\sigma_\varepsilon+\varepsilon \E \dot{u}_\varepsilon \weak \sigma$ weakly in $L^2(0,T;\Hdiv)$. Thus $(\sigma_\varepsilon+\varepsilon \E \dot{u}_\varepsilon)\nu \weak \sigma\nu$ weakly in $L^2(0,T;H^{-\frac12}(\partial \Omega;\Rn))$, and since $\big((\sigma_\varepsilon+\varepsilon \E \dot{u}_\varepsilon)\nu\big)_{\varepsilon>0}$ is bounded in $L^2(0,T;L^2(\partial\Omega;\Rn))$ by the energy balance, we have that $(\sigma_\varepsilon+\varepsilon \E \dot{u}_\varepsilon)\nu \weak \sigma\nu$ weakly in $L^2(0,T;L^2(\partial \Omega;\Rn))$. In particular, passing to the limit in the boundary condition, we obtain that
\begin{equation}\label{eq:bc-partial}
\sigma\nu +S\dot w =0 \quad \text{ in }L^2(0,T;L^2(\partial\Omega;\Rn))
\end{equation}
which will be made more precise later.

\medskip

It is also possible to pass  to  the limit in the stress constraint by setting $\tilde \sigma_\varepsilon:={\rm P}_{\mathbf K}(\sigma_\varepsilon)$, where ${\rm P}_{\mathbf K}$ is the orthogonal projection in $\Mnn$ onto the closed and convex set $\mathbf K$ with respect to the canonical scalar product. Since ${\rm P}_{\mathbf K}$ is $1$-Lipschitz and $0 \in \mathbf K$, we infer that $(\tilde \sigma_\varepsilon)_{\varepsilon>0}$ is bounded in $L^2(0,T;L^2(\Omega;\Mnn))$,  like   $(\sigma_\varepsilon)_{\varepsilon>0}$. Therefore, up to a subsequence, $\tilde \sigma_\varepsilon \weak \tilde \sigma$ weakly in $L^2(0,T;L^2(\Omega;\Mnn))$, for some $\tilde \sigma \in L^2(0,T;L^2(\Omega;\Mnn))$ with $\tilde \sigma \in \mathbf K$ a.e. in $\Omega \times (0,T)$. By the energy balance \eqref{en_balan_visc} together with the visco--plastic flow rule, we have that $\sigma_\varepsilon -\tilde \sigma_\varepsilon=\varepsilon \dot{p}_\varepsilon \to 0$ strongly in $L^2(0,T;L^2(\Omega;\Mnn))$. As a consequence $\sigma=\tilde \sigma$ a.e. in $\Omega \times (0,T)$ which shows the validy of the stress constraint 
$$\sigma\in \mathbf K\quad\text{ a.e. in }\Omega\times (0,T).$$

\subsubsection{Strong compactness}

We next improve the weak convergences into strong convergences. Indeed, substracting the equations of motion, we have
$$\ddot u_\varepsilon - \ddot u - \diver(\sigma_\varepsilon-\sigma +\varepsilon \E\dot u_\varepsilon)=0 \quad \text{ in }L^2(0,T;L^2(\Omega;\Rn))\,.$$
Taking $\dot u_\varepsilon {\bf 1}_{[0,t]} \in L^2(0,T;H^1(\Omega;\Rn))$ as test function, we get that
\begin{multline*}
\int_0^t \int_\Omega (\ddot u_\varepsilon - \ddot u)\cdot \dot u_\varepsilon\dx\ds + \int_0^t \int_\Omega (\sigma_\varepsilon-\sigma):\E \dot u_\varepsilon\dx\ds+\varepsilon\int_0^t \int_\Omega |\E \dot u_\varepsilon|^2\dx\ds\\
=\int_0^t \int_{\partial \Omega} ((\sigma_\varepsilon+\varepsilon \E\dot u_\varepsilon)\nu-\sigma\nu)\cdot \dot u_\varepsilon\dh\ds\\
=-\int_0^t \int_{\partial \Omega} S(\dot u_\varepsilon-\dot w)\cdot \dot u_\varepsilon\dh\ds+\int_0^t \int_{\partial \Omega}g_\varepsilon \cdot \dot u_\varepsilon\dh\ds\,,
\end{multline*}
where we used that $S\dot u_\varepsilon+(\sigma_\varepsilon+\varepsilon \E\dot u_\varepsilon)\nu=0$ and $S\dot w+\sigma\nu=0$ in $L^2(0,T;L^2(\partial\Omega;\Rn))$. According to the additive decomposition, we have that
$$  \int_0^t \int_\Omega (\sigma_\varepsilon-\sigma):\E \dot u_\varepsilon\dx\ds= \int_0^t \int_\Omega (\sigma_\varepsilon-\sigma):\dot e_\varepsilon\dx\ds + \int_0^t \int_\Omega (\sigma_\varepsilon-\sigma):\dot p_\varepsilon\dx\ds.$$
Since $\sigma \in \mathbf K$ a.e. in $\Omega \times (0,T)$, the visco--plastic flow rule implies that
$$\int_0^t \int_\Omega (\sigma_\varepsilon-\sigma):\dot p_\varepsilon\dx\ds \geq 0\,.$$
Thus
\begin{multline*}
\int_0^t \int_\Omega (\ddot u_\varepsilon - \ddot u)\cdot (\dot u_\varepsilon-\dot u)\dx\ds + \int_0^t \int_\Omega (\sigma_\varepsilon-\sigma):(\dot e_\varepsilon-\dot e)\dx\ds
+\varepsilon\int_0^t \int_\Omega |\E \dot u_\varepsilon|^2\dx\ds\\+\int_0^t \int_{\partial \Omega} S(\dot u_\varepsilon-\dot w)\cdot (\dot u_\varepsilon-\dot w)\dh\ds
\leq-\int_0^t \int_\Omega (\ddot u_\varepsilon - \ddot u)\cdot \dot u\dx\ds \\- \int_0^t \int_\Omega (\sigma_\varepsilon-\sigma):\dot e\dx\ds 
- \int_0^t \int_{\partial \Omega} S(\dot u_\varepsilon-\dot w)\cdot \dot w\dh\ds+\int_0^t \int_{\partial \Omega}g_\varepsilon \cdot \dot u_\varepsilon\dh\ds.
\end{multline*}
The weak convergences \eqref{2811182355} imply that the right hand side of the previous inequality tends to zero as $\varepsilon \to 0$. Thus, since 
$$\int_0^t \int_\Omega (\ddot u_\varepsilon - \ddot u)\cdot (\dot u_\varepsilon-\dot u)\dx\ds =\frac12 \|\dot u_\varepsilon(t)-\dot u(t)\|^2_{L^2(\Omega;\Rn)},$$
$$\int_0^t \int_\Omega (\sigma_\varepsilon-\sigma):(\dot e_\varepsilon-\dot e)\dx\ds
=\Q(e_\varepsilon(t)-e(t)),$$
and, by \eqref{2711182136},
$$\int_0^t \int_{\partial \Omega} S(\dot u_\varepsilon-\dot w)\cdot (\dot u_\varepsilon-\dot w)\dh\ds\geq c\int_0^t \int_{\partial \Omega}|\dot u_\varepsilon-\dot w|^2\dh\ds$$
we get the following strong convergences, for all $t \in [0,T]$,
$$\begin{cases}
\dot u_\varepsilon(t) \to \dot u(t) & \text{ strongly in }L^2(\Omega;\R^n),\\
\sigma_\varepsilon(t) \to \sigma(t) & \text{ strongly in }L^2(\Omega;\mathbb M^n_{\rm sym}),\\
\dot u_\varepsilon \to \dot w & \text{ strongly in }L^2(0,T ;L^2(\partial \Omega;\R^n)).
\end{cases}$$
In addition, the estimates \eqref{estimate_viscosa} and Lebesgue's dominated convergence theorem yield
$$\begin{cases}
\dot u_\varepsilon \to \dot u & \text{ strongly in }L^2(0,T;L^2(\Omega;\R^n)),\\
\sigma_\varepsilon \to \sigma & \text{ strongly in }L^2(0,T;L^2(\Omega;\mathbb M^n_{\rm sym})),\\
\end{cases}$$

According to  the  boundary condition at $\varepsilon$  fixed,  we deduce that the sequence $\big((\sigma_\varepsilon+\varepsilon \E \dot{u}_\varepsilon)\nu\big)_{\varepsilon>0}$ is strongly converging in $L^2(0,T;L^2(\partial\Omega;\Rn))$, and thus,
$$(\sigma_\varepsilon+\varepsilon \E u_\varepsilon)\nu \to \sigma\nu  \quad \text{ strongly in }L^2(0,T;L^2(\partial \Omega;\R^n)).$$

\begin{remark}
In the case of a cylindrical convex set $\mathbf K$ of the form \eqref{eq:cylincric}, since the plastic strain is deviatoric then
the convergence \eqref{2911180016} has to be replaced by
$$p_\varepsilon(t) \wstar p(t) \quad\text{weakly* in }\mathcal M_b(\Omega;\mathbb M^n_D)$$
for all $t \in [0,T]$. Thus, Remark \ref{rem:wnu=unu} ensures that $u(t)\cdot\nu=w(t)\cdot \nu$ on $\partial \Omega$, hence $\dot w(t)\cdot\nu=\dot u(t)\cdot \nu$ on $\partial \Omega$ for a.e. $t \in (0,T)$. Taking the normal component of the boundary condition \eqref{eq:bc-partial} yields
$$\dot u \cdot \nu+(S^{-1}\sigma\nu)\cdot\nu=0\quad \text{in }L^2(0,T; L^2(\dom; \Rn)).$$
\end{remark}

\subsubsection{Energy balance}\label{sec:energy-balance}

We first localize the energy balance \eqref{en_balan_visc} for fixed $\varepsilon$. To this end, we consider a test function $\varphi \in C^\infty_c(\Rn)$ and we multiply the equation of motion by $\varphi \dot u_\varepsilon$. After integration by parts in space and time, we obtain for all $t \in [0,T]$,
 \begin{multline}\label{eq:energy-balance}
\frac12 \int_\Omega \varphi |\dot{u}_\varepsilon(t)|^2\dx + \int_\Omega \varphi\, Q(e_\varepsilon(t))\dx + \int_0^t \int_\Omega \varphi \,H(\dot{p}_\varepsilon) \dx\ds +\int_0^t \int_\dom \varphi \,S \dot{u}_\varepsilon \cdot \dot{u}_\varepsilon \dh \ds\\
+\varepsilon \int_0^t \int_\Omega \varphi |\E \dot{u}_\varepsilon|^2\dx\ds + \varepsilon\int_0^t \int_\Omega \varphi |\dot{p}_\varepsilon|^2 \dx\ds +\int_0^t\int_\Omega \sigma_\varepsilon:(\nabla \varphi \odot \dot u_\varepsilon)\dx\ds\\
 =\frac12 \int_\Omega \varphi |v_0|^2\dx +\int_\Omega \varphi \, Q(e_0)\dx + \int_0^t \int_\Omega \varphi f \cdot \dot{u}_\varepsilon \dx \ds + \int_0^t \int_{\dom} \varphi g_\varepsilon \cdot \dot{u}_\varepsilon \dh \ds\,.
\end{multline}
If further $\varphi \geq 0$, we write the energy balance \eqref{eq:energy-balance} as
\begin{multline*}
\frac12 \int_\Omega \varphi |\dot{u}_\varepsilon(t)|^2\dx + \int_\Omega \varphi \, Q(e_\varepsilon(t))\dx + \int_0^{t} \int_\Omega \varphi H(\dot{p}_\varepsilon) \dx\ds +\int_0^{t} \int_\dom \varphi\, S \dot{u}_\varepsilon \cdot \dot{u}_\varepsilon \dh \ds\\
+\int_0^{t}\int_\Omega \sigma_\varepsilon:(\nabla \varphi \odot \dot u_\varepsilon)\dx\ds\\
 \leq \frac12 \int_\Omega \varphi |v_0|^2\dx +\int_\Omega \varphi \, Q(e_0)\dx + \int_0^{t} \int_\Omega \varphi f \cdot \dot{u}_\varepsilon \dx \ds + \int_0^{t} \int_{\dom} \varphi g_\varepsilon \cdot \dot{u}_\varepsilon \dh \ds\,.
\end{multline*}

According to the strong convergences of the velocity and the stress, we have
$$\frac12 \int_\Omega \varphi |\dot{u}_\varepsilon(t)|^2\dx + \int_\Omega \varphi \, Q(e_\varepsilon(t))\dx\to \frac12 \int_\Omega \varphi |\dot{u}(t)|^2\dx + \int_\Omega \varphi \, Q(e(t))\dx,$$
$$\int_0^{t}\int_\Omega \sigma_\varepsilon:(\nabla \varphi \odot \dot u_\varepsilon)\dx\ds \to \int_0^{t}\int_\Omega \sigma:(\nabla \varphi \odot \dot u)\dx\ds$$
and
$$\int_0^{t} \int_\Omega \varphi f \cdot \dot{u}_\varepsilon \dx \ds + \int_0^{t} \int_{\dom} \varphi g_\varepsilon \cdot \dot{u}_\varepsilon \dh \ds\to
\int_0^{t} \int_\Omega \varphi f \cdot \dot{u} \dx \ds.$$

Using the boundary condition and the strong convergences of the boundary velocity and the normal stress, we get
\begin{multline*}
\int_0^{t} \int_\dom \varphi\,S \dot{u}_\varepsilon \cdot \dot{u}_\varepsilon \dh \ds=\frac12 \int_0^{t} \int_\dom \varphi \,S \dot{u}_\varepsilon \cdot \dot{u}_\varepsilon \dh \ds\\
+\frac12  \int_0^{t} \int_\dom \varphi \,S^{-1}\big((\sigma_\varepsilon+\varepsilon \E \dot u_\varepsilon)\nu-g_\varepsilon\big)  \cdot\big((\sigma_\varepsilon+\varepsilon \E \dot u_\varepsilon)\nu-g_\varepsilon\big) \dh \ds\\
 \to \frac12  \int_0^{t} \int_\dom \varphi \, S\dot w\cdot \dot w\dh\ds+\frac12  \int_0^{t} \int_\dom \varphi \, S^{-1}(\sigma\nu)\cdot(\sigma\nu)\dh\ds.
\end{multline*}

Let $0=t_0 \leq t_1 \leq \cdots \leq t_N=t$ be a partition of $[0,t]$. Since, for all $0 \leq i \leq N$, we have $p_\varepsilon(t_i) \wstar p(t_i)\mres \Omega + (w(t_i)-u(t_i))\odot \nu\hn \mres \partial \Omega$ weakly* in $\mathcal M_b(\overline \Omega;\Mnn)$, using Jensen's inequality and Reshetnyak's lower semicontinuity Theorem, we get that
\begin{multline*}
\liminf_{\varepsilon \to 0}\int_0^t \int_\Omega \varphi H(\dot p_\varepsilon)\dx\ds \geq \sum_{i=1}^N \liminf_{\varepsilon \to 0}\int_\Omega \varphi H(p_\varepsilon(t_i)-p_\varepsilon(t_{i-1}))\dx\\
 \geq \sum_{i=1}^N \left\{ \int_\Omega\varphi  \de H(p(t_i)-p(t_{i-1}))  \right.\\
 \left.+ \int_\dom \varphi H\big( ((w(t_i)-u(t_i))-(w(t_{i-1})-u(t_{i-1}))\odot \nu \big)\dh\right\}.
\end{multline*}
Passing to the supremum with respect to all partitions and using \cite[Theorem 7.1]{DMDSMor06}, we get that
\begin{multline*}
\liminf_{\varepsilon \to 0}\int_0^t \int_\Omega \varphi H(\dot p_\varepsilon)\dx\ds 
\geq  \int_0^t  \int_\Omega\varphi  \de H(\dot p(s)) \ds + \int_0^t \int_\dom \varphi H\big( (\dot w(s)-\dot u(s))\odot \nu \big)\dh\ds\,.
\end{multline*}
Gathering all previous convergences and using the definition of $\psi$, we deduce that 
\begin{multline}\label{eq:ineq1}
\frac12 \int_\Omega \varphi|\dot u(t)|^2\dx + \int_\Omega \varphi \, Q(e(t)) \dx +  \int_0^{t} \int_\Omega \varphi \de  H(\dot p(s))  \ds \\
+\int_0^{t} \int_\Omega \sigma:(\nabla \varphi \odot \dot u)\dx\ds+   \int_0^{t} \int_\dom \varphi \psi(x, \dot{u})  \dh \ds+ \frac{1}{2} \int_0^{t} \int_\dom \varphi S^{-1}(\sigma \nu) \cdot (\sigma \nu) \dh \ds   \\
\leq 
\frac12 \int_\Omega \varphi|\dot u(t)|^2\dx + \int_\Omega \varphi \, Q(e(t)) \dx +  \int_0^{t} \int_\Omega \varphi \de H(\dot p(s) ) \ds\\
 +\int_0^{t} \int_\Omega \sigma:(\nabla \varphi \odot \dot u)\dx\ds+ \frac{1}{2} \int_0^{t} \int_\dom \varphi S^{-1}(\sigma \nu) \cdot (\sigma \nu) \dh \ds   \\
+   \int_0^t \int_\dom \varphi H\big( (\dot w-\dot u)\odot \nu\big)\dh\ds+\frac12  \int_0^{t} \int_\dom \varphi \, S\dot w\cdot \dot w\dh\ds\\
\leq \frac12 \int_\Omega \varphi |v_0|^2\dx + \int_\Omega \varphi \, Q(e_0) \dx + \int_0^t \int_\Omega \varphi f \cdot \dot{u}  \dx \ds \,.
\end{multline}

\medskip

We now prove the opposite inequality. According to Proposition \ref{prop:IPP} and using the equation of motion, we have that
\begin{multline*}
\int_\Omega \varphi \de H(\dot p(t))  + \int_{\partial \Omega}\varphi \psi(x,\dot u(t))\dh+\frac12\int_{\partial \Omega}  \varphi S^{-1}(\sigma(t)\nu)\cdot (\sigma(t)\nu)\dh\\
 \geq-\int_\Omega \varphi \sigma(t) : \dot e(t)\dx- \int_\Omega \varphi \dot u(t) \cdot \ddot u(t) \dx
 +\int_\Omega \varphi \dot u(t) \cdot f(t) \dx - \int_\Omega \sigma(t) \colon (\dot u(t) \odot \nabla \varphi) \dx.
\end{multline*}
Integrating in time between $0$ and $t$ and integrating by parts, leads to the remaining energy inequality
\begin{multline}\label{eq:ineq2}
\frac12 \int_\Omega \varphi|\dot u(t)|^2\dx + \int_\Omega \varphi \,Q(e(t))\dx +  \int_0^{t} \int_\Omega \varphi \de H(\dot p(s)) \ds \\
+\int_0^{t} \int_\Omega \sigma:(\nabla \varphi \odot \dot u)\dx\ds+   \int_0^{t} \int_\dom \varphi \psi(x, \dot{u})  \dh \ds+ \frac{1}{2} \int_0^{t} \int_\dom \varphi S^{-1}(\sigma \nu) \cdot (\sigma \nu) \dh \ds   \\
\geq \frac12 \int_\Omega \varphi |v_0|^2\dx + \int_\Omega \varphi \, Q(e_0)\dx + \int_0^{t} \int_\Omega \varphi f \cdot \dot{u}  \dx \ds \,.
\end{multline}
Gathering both energy inequalities \eqref{eq:ineq1} and \eqref{eq:ineq2} leads to the energy balance: for all $t \in [0,T]$ and all $\varphi \in C^\infty_c(\Rn)$,
\begin{multline}\label{eq:eqnrj2}
\frac12 \int_\Omega \varphi|\dot u(t)|^2\dx + \int_\Omega \varphi \,Q(e(t))\dx +  \int_{0}^{t} \int_\Omega \varphi \de H( \dot p(s))\ds \\
+\int_{0}^{t} \int_\Omega \sigma:(\nabla \varphi \odot \dot u)\dx\ds+   \int_{0}^{t} \int_\dom \varphi \psi(x, \dot{u})  \dh \ds+ \frac{1}{2} \int_{0}^{t} \int_\dom \varphi S^{-1}(\sigma \nu) \cdot (\sigma \nu) \dh \ds   \\
= \frac12 \int_\Omega \varphi |v_0|^2\dx + \int_\Omega \varphi \, Q(e_0)\dx + \int_0^t \int_\Omega \varphi f \cdot \dot{u}  \dx \ds \,.
\end{multline}

\subsubsection{Flow rule and relaxed boundary condition}

Derivating with respect to time the energy balance \eqref{eq:eqnrj2}, we get for a.e. $t \in [0,T]$ and all $\varphi \in C^\infty_c(\Rn)$,
\begin{multline*}
\int_\Omega \varphi \de H(\dot p(t))  + \int_{\partial \Omega}\varphi \psi(x,\dot u(t))\dh
 =-\int_\Omega \varphi \sigma(t) : \dot e(t)\dx- \int_\Omega \varphi \dot u(t)\cdot \diver \sigma(t) \dx\\
 - \int_\Omega \sigma(t) \colon (\dot u(t) \odot \nabla \varphi) \dx-\frac12\int_{\partial \Omega}  \varphi S^{-1}(\sigma(t)\nu)\cdot (\sigma(t)\nu)\dh.
\end{multline*}
Taking in particular a test function $\varphi \in C^\infty_c(\Omega)$ with compact support in $\Omega$,
\begin{eqnarray*}
\int_\Omega \varphi \de H(\dot p(t))  & =& -\int_\Omega \varphi \sigma(t) : \dot e(t)\dx- \int_\Omega \varphi \dot u(t)\cdot \diver \sigma(t) \dx
 - \int_\Omega \sigma(t) \colon (\dot u(t) \odot \nabla \varphi) \dx\\
& = & \langle [\sigma(t):\dot p(t)],\varphi\rangle,
\end{eqnarray*}
hence we get the measure theoretic flow rule
$$H(\dot p(t))=[\sigma(t):\dot p(t)] \quad \text{ in }\mathcal M_b(\Omega).$$

\medskip

Taking $\varphi\in C^\infty_c(\Rn)$ with $\varphi=1$ in $\overline \Omega$, we get from the energy balance \eqref{eq:eqnrj2} that
\begin{multline*}
\frac12 \int_\Omega |\dot u(t)|^2\dx + \mathcal Q(e(t)) +  \int_{0}^{t} \mathcal H(\dot p(s)) \ds 
+  \int_{0}^{t} \int_\dom  \psi(x, \dot{u})  \dh \ds \\ + \frac{1}{2} \int_{0}^{t} \int_\dom  S^{-1}(\sigma \nu) \cdot (\sigma \nu) \dh \ds   
=\frac12  \int_\Omega  |v_0|^2\dx + \mathcal Q(e_0) + \int_0^t \int_\Omega  f \cdot \dot{u}  \dx \ds \,,
\end{multline*}
and also from \eqref{eq:ineq1} that
$$\int_0^t \int_{\partial \Omega} \psi(x,\dot u)\dh\ds 
= \int_0^t \int_\dom  H\big( (\dot w-\dot u)\odot \nu \big)\dh\ds+\frac12  \int_0^{t} \int_\dom  S\dot w\cdot \dot w\dh\ds.$$
As a consequence for a.e. $(x,t) \in \partial \Omega \times (0,T)$, we have
$$ \psi(x,\dot u(x,t))= H\big( (\dot w(x,t)-\dot u(x,t))\odot \nu(x) \big)+\frac12    S(x)\dot w(x,t)\cdot \dot w(x,t)$$
which implies, in view of Lemma \ref{lem:infconv} and \eqref{eq:dzpsi}, that
$$\dot w(x,t)=S^{-1}(x)D_z \psi(x,\dot u(x,t))=S^{-1}(x) {\rm P}_{-\mathbf K\nu(x)} (S(x) \dot u(x,t)).$$
Inserting this expression into \eqref{eq:bc-partial} leads to the desired boundary condition
$$\sigma\nu +{\rm P}_{-\mathbf K\nu} (S \dot u)=0 \quad \text{ in }L^2(0,T;L^2(\partial\Omega;\Rn)).$$

\subsubsection{Uniqueness}
Let $(u_1,e_1,\sigma_1,p_1)$ and $(u_2,e_2,\sigma_2,p_2)$ be two variational solutions of the elasto--plastic model, according to Definition \ref{def:variational}, associated to the same initial data $(u_0,v_0,e_0,p_0)$, the same source term $f$ and the same boundary matrix $S$. Let us define
$$u:=\frac{u_1+u_2}{2}, \quad e:=\frac{e_1+e_2}{2}, \quad p:=\frac{p_1+p_2}{2}, \quad \sigma:=\frac{\sigma_1+\sigma_2}{2}.$$
Then, by linearity and convexity, we have
\begin{equation}\label{eq:init}
u(0)=u_0, \quad \dot{u}(0)=v_0, \quad e(0)=e_0, \quad p(0)=p_0,
\end{equation}
for all $t\in [0,T]$,
$$\E u(t) = e(t) + p(t) \quad \text{in } \Mbn,$$
$$\sigma =\C e\in \mathbf K\quad \text{a.e. in } \Omega \times (0,T),$$
and
$$\ddot{u}- \diver\sigma =f \quad \text{in }L^2(0,T; L^2(\Omega;\Rn)).$$

For $i=1$, $2$ and all $t \in [0,T]$, we define
\begin{multline*}
\mathcal E_t(u_i,e_i,p_i):=\frac12 \int_\Omega |\dot u_i(t)|^2\dx + \mathcal Q(e_i(t)) +  \int_{0}^{t} \mathcal H(\dot p_i(s))  \ds - \int_0^t \int_\Omega  f \cdot \dot{u}_i  \dx \ds\\
+  \int_{0}^{t} \int_\dom  \psi(x, \dot{u}_i)  \dh \ds+ \frac{1}{2} \int_{0}^{t} \int_\dom  S^{-1}(\sigma_i \nu) \cdot (\sigma_i \nu) \dh \ds,
\end{multline*}
and similarly for $\mathcal E_t(u,e,p)$.
By convexity of $H$ and $\psi$, we have for all $t \in [0,T]$
\begin{multline}\label{eq:firstineq}
\int_\Omega\left|\frac{u_1(t)-u_2(t)}{2}\right|^2\dx + \mathcal Q\left(\frac{e_1(t)-e_2(t)}{2}\right)+\mathcal E_t(u,e,p)\\
\leq \frac12 \mathcal E_t(u_1,e_1,p_1) + \frac12 \mathcal E_t(u_2,e_2,p_2)
= \frac12  \int_\Omega  |v_0|^2\dx + \mathcal Q(e_0) \,.
\end{multline}

According to Proposition \ref{prop:IPP} together with the equation of motion, we have that for a.e. $t \in (0,T)$,
\begin{multline*}
\mathcal H(\dot p(t)) + \int_\dom \psi(x,\dot u(t))\dh+ \frac{1}{2} \int_\dom S^{-1}(\sigma(t)\nu) \cdot (\sigma(t) \nu) \dh\\
\geq - \int_\Omega \sigma(t):\dot e(t)\dx -  \int_\Omega \ddot u(t) \cdot \dot u(t)\dx+\int_\Omega f(t)\cdot \dot u(t)\dx.\end{multline*}
We next integrate by parts with respect to time and use the initial conditions \eqref{eq:init} to obtain that 
\begin{equation}\label{eq:secondineq}
\mathcal E_t(u,e,p) \geq \frac12 \int_\Omega |v_0|^2\dx + \mathcal Q(e_0).
\end{equation}
Gathering \eqref{eq:firstineq} and \eqref{eq:secondineq} implies that $u_1=u_2$ and $e_1=e_2$, hence $p_1=p_2$ as well.

\begin{remark}\label{rem:unique}
Let us observe that our proof of uniqueness does not explicitely use the dissipative boundary condition, but rather its weak formulation through the energy balance. As a consequence, we proved the uniqueness of solutions among all triples $(u,e,p)$
satisfying all requirements in Definition \ref{def:variational} except the relaxed boundary condition.
\end{remark}

\section{Entropic-dissipative solutions}

Following the discussion of  Subsection~\ref{sec:entrop},  we introduce the following generalized notion of solutions to the  dynamic  elasto--plasticity model.

\begin{definition}\label{def:entropic}
Assume that  hypotheses $(H_1)$--$(H_4)$ are satisfied. A pair 
$$(v,\sigma) \in W^{1,\infty}(0,T;L^2(\Omega;\Rn \times \mathbf K))$$
defines an {\it entropic--dissipative solution to the dynamic elasto--plasticity model} associated to the initial data $(v_0,\sigma_0) \in L^2(\Omega;\Rn\times \mathbf K)$ and the source term $f \in L^2(0,T;L^2(\Omega;\Rn))$ if
\begin{multline*}
\int_0^T\int_\Omega  \left(  |v-z|^2 +\mathbf A^{-1}(\sigma-\tau): (\sigma-\tau)\right) \partial_t \varphi\dx \de t-2\int_0^T\int_\Omega (\sigma -\tau):((v-z)\odot \nabla \varphi )\dx \de t \\
+\int_\Omega \varphi(0) \left(  |v_0-z|^2 +\mathbf A^{-1}(\sigma_0-\tau): (\sigma_0-\tau)\right)\dx\\
 +  2 \int_0^T\int_\Omega f\cdot(v-z) \varphi \dx \de t+\frac12\int_0^T\int_\dom   S^{-1}(\tau\nu+Sz)\cdot(\tau\nu+Sz)\varphi\dh \de t\geq 0,
\end{multline*}
for all constant vectors $(z,\tau) \in \R^n\times \mathbf K$ and all test functions $\varphi \in C^\infty_c(\R^n \times [0,T))$ with $\varphi \geq 0$.
\end{definition}

We intend to show that both notions of variational and entropic--dissipative solutions actually coincide. Using the elasto--visco--plastic approximation, we start by proving that variational solutions are entropic--dissipative ones.

\begin{theorem}\label{thm:vardiss}
Assuming $(H_1)$--$(H_6)$, let $(u,e,p)$ be the unique variational solution (obtained in Theorem \ref{teo:existence_variational}) to the dynamic elasto--plastic model associated to the initial data $(u_0,v_0,e_0,p_0)$ and the source term $f$ according to Definition \ref{def:variational}. Then $(v=\dot u,\sigma=\mathbf Ae)$ is an entropic--dissipative solution associated to the initial data $(v_0,\sigma_0=\mathbf A e_0)$ and the source term $f$ according to Definition~\ref{def:entropic}. 
\end{theorem}

\begin{proof}
We multiply the equation of motion $\ddot u_\varepsilon -\diver (\sigma_\varepsilon + \varepsilon \E \dot u_\varepsilon-\tau)=f$ by $\varphi (\dot u_\varepsilon-z)$ where $\varphi \in C^\infty_c(\Rn\times [0,T))$ with $\varphi\geq 0$, and integrate by parts in space and time. Then
\begin{equation*}
\begin{split}
&\int_0^T\int_\Omega\varphi f\cdot(\dot u_\varepsilon-z)\dx\de t+\int_0^T\int_{\partial\Omega} (\sigma_\varepsilon + \varepsilon \E \dot u_\varepsilon-\tau)\nu \cdot (\dot u_\varepsilon-z)\varphi\dh \de t\\
&=\int_0^T\int_\Omega \left(\varphi \ddot u_\varepsilon \cdot (\dot u_\varepsilon-z) + \varphi (\sigma_\varepsilon + \varepsilon \E \dot u_\varepsilon-\tau):\E\dot u_\varepsilon+(\sigma_\varepsilon + \varepsilon \E \dot u_\varepsilon-\tau):((\dot u_\varepsilon-z) \odot \nabla \varphi\right)\dx \de t\\
&=\frac12\int_0^T \int_\Omega \varphi \partial_t \left(  |\dot u_\varepsilon-z|^2 +\mathbf A(e_\varepsilon-\mathbf A^{-1}\tau): (e_\varepsilon-\mathbf A^{-1}\tau)\right)\dx \de t
+\int_0^T\int_\Omega \varphi (\sigma_\varepsilon -\tau):\dot p_\varepsilon\dx \de t 
\\& \hspace{1em} + \varepsilon \int_0^T \int_\Omega |\E \dot u_\varepsilon|^2\dx \de t+  \int_0^T\int_\Omega (\varepsilon \E \dot u_\varepsilon + (\sigma_\varepsilon -\tau)):((\dot u_\varepsilon-z) \odot \nabla \varphi)\dx \de t\,.
\end{split}
\end{equation*}
Since $\tau \in \mathbf K$, the flow rule implies that $(\sigma_\varepsilon -\tau):\dot p_\varepsilon \geq 0$ a.e. in $\Omega \times (0,T)$. Hence, integrating by parts the first integral in the right--hand-side yields
\begin{equation*}
\begin{split}
& \int_0^T \int_\Omega\varphi f\cdot(\dot u_\varepsilon-z)\dx \de t+\int_0^T\int_{\partial\Omega} (\sigma_\varepsilon + \varepsilon \E \dot u_\varepsilon-\tau)\nu \cdot (\dot u_\varepsilon-z)\varphi\dh \de t\\
&\geq -\frac12\int_0^T \int_\Omega \partial_t \varphi \left(  |\dot u_\varepsilon-z|^2 +\mathbf A(e_\varepsilon-\mathbf A^{-1}\tau): (e_\varepsilon-\mathbf A^{-1}\tau)\right)\dx \de t\\
& \hspace{1em}-\frac12\int_\Omega \varphi(0) \left(  |v_0-z|^2 +\mathbf A(e_0-\mathbf A^{-1}\tau): (e_0-\mathbf A^{-1}\tau)\right)\dx \de t \\ & \hspace{1em}
+ \int_0^T\int_\Omega (\varepsilon \E \dot u_\varepsilon + (\sigma_\varepsilon -\tau)):((\dot u_\varepsilon-z) \odot \nabla \varphi)\dx  \de t\,.
\end{split}
\end{equation*}
Passing to the limit as $\varepsilon \to 0$ and using the previously established convergence leads to
\begin{multline*}
2 \int_0^T \int_\Omega\varphi f\cdot(\dot u-z)\dx\de t+2\int_0^T \int_\dom \varphi (\sigma-\tau)\nu \cdot \left(S^{-1} D_z\psi(\, \bullet\, ,\dot u)-z\right)\dh \de t\\
+\int_0^T \int_\Omega \partial_t \varphi \left(  |\dot u-z|^2 +\mathbf A^{-1}(\sigma-\tau): (\sigma-\tau)\right)\dx \de t\\
+\int_\Omega \varphi(0) \left(  |v_0-z|^2 +\mathbf A^{-1}(\sigma_0-\tau): (\sigma_0-\tau)\right)\dx \de t\\
 -2\int_0^T \int_\Omega (\sigma -\tau):((\dot u-z)\odot \nabla \varphi )\dx\de t \geq 0\,.
\end{multline*}
Using the parallelogram identity, we have
\begin{equation*}
\begin{split}
2 & (\sigma-\tau)\nu \cdot \left(S^{-1} D_z\psi(\, \bullet\, ,\dot u)-z\right)  =  2(\sigma-\tau)\nu \cdot S^{-1} \left(D_z\psi(\, \bullet\, ,\dot u)-Sz\right)\\
&=2 \left\|\frac{(\sigma-\tau)\nu}{2}+\frac{D_z\psi(\, \bullet\, ,\dot u)-Sz}{2} \right\|^2_{S^{-1}} -2\left\|\frac{(\sigma-\tau)\nu}{2}-\frac{D_z\psi(\, \bullet\, ,\dot u)-Sz}{2} \right\|^2_{S^{-1}}\\
& \leq  2\left\|\frac{(\sigma-\tau)\nu}{2}+\frac{D_z\psi(\, \bullet\, ,\dot u)-Sz}{2} \right\|^2_{S^{-1}} =  \frac12 \left\|\tau\nu+Sz \right\|^2_{S^{-1}}=\frac12 S^{-1}(\tau\nu+Sz)\cdot(\tau\nu+Sz)\,,
\end{split}
\end{equation*}
where we used the boundary condition $\sigma\nu+D_z\psi(\,\bullet \,,\dot u)=0$ on $\dom \times (0,T)$.
We deduce that
\begin{multline*}
2 \int_0^T\int_\Omega\varphi f\cdot(\dot u-z)\dx\de t+\frac12\int_0^T \int_\dom \varphi  S^{-1}(\tau\nu+Sz)\cdot(\tau\nu+Sz)\dh \de t\\
 {+}\int_0^T \int_\Omega \partial_t \varphi \left(  |\dot u-z|^2 +\mathbf A^{-1}(\sigma-\tau): (\sigma-\tau)\right)\dx \de t\\
{+}\int_\Omega \varphi(0) \left(  |v_0-z|^2 +\mathbf A^{-1}(\sigma_0-\tau): (\sigma_0-\tau)\right)\dx \de t\\
 -2\int_0^T \int_\Omega (\sigma -\tau):((\dot u-z)\odot \nabla \varphi )\dx \de t \geq 0\,,
\end{multline*}
which completes the proof of the Theorem.
\end{proof}

We now prove the converse implication, namely that any entropic--dissipative solution  generates  a variational solution to the dynamic elasto--plastic model.

\begin{theorem}\label{thm:dissvar}
Assuming $(H_1)$--$(H_5)$, let $(v, \sigma) \in W^{1,\infty}(0,T;L^2(\Omega;\R^n \times \mathbf K))$ be an  entropic--dissipative solution to the dynamic elasto--plasticity model associated to the initial data $(v_0,\sigma_0) \in H^2(\Omega;\R^n) \times \mathcal K(\Omega)$ with $Sv_0+\sigma_0\nu=0$ on $\dom$, and the source term $f \in H^1(0,T;L^2(\Omega;\R^n))$. Then, for any initial displacement $u_0 \in H^1(\Omega;\R^n)$, 
there exists a triple $(u,e,p)$, with $v=\dot{u}$, $\sigma=\mathbf{A}e$, which is a variational solution to the dynamic elasto--plastic model associated to the initial data $(u_0, v_0, e_0, p_0=\mathrm{E} u_0-e_0)$ and the source term $f$.
\end{theorem}

\begin{proof}   We divide the proof in several steps, the four first ones closely following the proof of \cite[Proposition~7.2]{BM17}. 
\medskip
\paragraph{\it Step 1: Initial conditions for $v$ and $\sigma$.}  By a general property of dissipative formulations of initial-boundary value problems proven in \cite[Lemma~3]{DMS}, the initial condition is satisfied in the essential limit sense, i.e.,
$$\lim_{t\to 0} \lim_{\alpha \to 0} \frac{1}{\alpha} \int_{t-\alpha}^t \int_\Omega |\sigma(x,s)-\sigma_0(x)|^2\dx \ds =0, \quad \lim_{t\to 0} \lim_{\alpha \to 0} \frac{1}{\alpha} \int_{t-\alpha}^t \int_\Omega |v(x,s)-v_0(x)|^2  \dx \ds =0\,.$$
 Note that in \cite{DMS}, the authors consider unconstrained Friedrichs systems. However, a careful inspection of the proof of \cite[Lemma~3]{DMS} reveals that this result still holds in the constrained case.
Together with the property that $\sigma \in W^{1,\infty}(0,T;L^2(\Omega;\Mnn))$ and $v \in W^{1,\infty}(0,T;L^2(\Omega;\Rn))$, that gives $(\sigma(t),v(t)) \to (\sigma(0),v(0))$ in $L^2(\Omega;\Mnn) \times L^2(\Omega;\Rn)$ as $t \to 0$, the identities above imply that $(\sigma(0),v(0))=(\sigma_0,v_0)$.

\medskip

\paragraph{\it Step 2: Definition of $(u, e, p)$ and preliminary properties.} For all $t \in [0,T]$, we set
$$u(t):= u_0 + \int_0^t v(s) \ds$$
as a Bochner integral in $L^2(\Omega;\R^n)$. In view of the Lipschitz regularity of $v$, we get that $u \in W^{2,\infty}(0,T; L^2(\Omega; \R^n))$. By construction $u(0)=u_0$ and, by Step~1, $\dot{u}(0)=v_0$. Let us define
\begin{equation}\label{1410191146}
e:=\mathbf{A}^{-1}\sigma\ \in W^{1,\infty}(0,T;L^2(\Omega; \mathbb M^n_{\rm sym}))\,,\qquad   p:=\mathrm{E}u -e \in W^{1,\infty}(0,T; H^{-1}(\Omega; \Mnn)) \,.
\end{equation} 
Integrating by parts with respect to time the entropic--dissipative inequality in Definition~\ref{def:entropic}, we get that for any $\varphi \in C^1_c(\R^n \times [0,T))$ with $\varphi \geq 0$ and any $(z,\tau) \in \R^n \times \mathbf{K}$
\begin{multline*}
-2 \int_0^T\int_\Omega  \left( \ddot{u} \cdot (\dot{u}-z) + \dot{e} \colon (\sigma-\tau)\right) \varphi\dx \de t-2\int_0^T\int_\Omega (\sigma -\tau):((\dot{u}-z)\odot \nabla \varphi )\dx\de t \\
 +  2 \int_0^T\int_\Omega f\cdot(\dot{u}-z) \varphi \dx\de t+\frac12\int_0^T\int_\dom   S^{-1}(\tau\nu+Sz)\cdot(\tau\nu+Sz)\varphi\dh \de t\geq 0\,.
\end{multline*}
Factorizing  in $z$ and $\tau$, and performing a further integration by parts, we get that 
\begin{multline}\label{1410191339}
- \int_0^T\int_\Omega  \Big( \left( \ddot{u} \cdot \dot{u} +\dot{e} \colon \sigma - f \cdot \dot{u }\right) \varphi +  \sigma :( \nabla \varphi \odot \dot{u}) \Big)\dx \de t +z \,\cdot \int_0^T\int_\Omega  \left( \ddot{u} \varphi + \sigma\, \nabla \varphi - f\, \varphi \right) \dx \de t \\
 + \tau \, \colon \int_0^T \int_\Omega \left( \dot{e} \varphi+ \dot{u}\odot \nabla \varphi  \right)\dx  \de t
 +\frac14\int_0^T\int_\dom   S^{-1}(\tau\nu-Sz)\cdot(\tau\nu-Sz)\varphi\dh \de t\geq 0\,.
\end{multline}

With the choice $(z, \tau)=0$ we obtain that
$$- \int_0^T\int_\Omega  \Big( \left( \ddot{u} \cdot \dot{u} +\dot{e} \colon \sigma - f \cdot \dot{u }\right) \varphi +  \sigma :(\nabla \varphi \odot \dot{u}) \Big)\dx \,\de t \geq 0$$
for any $\varphi \in C^1_c(\R^n \times [0,T))$ with $\varphi \geq 0$, and localizing in time this gives  for a.e.\ $t \in (0,T)$,
\begin{equation}\label{1410191330}
\langle \mu(t), \phi \rangle:=- \int_\Omega  \Big( \left( \ddot{u}(t) \cdot \dot{u}(t) +\dot{e}(t)\colon \sigma(t) - f(t) \cdot \dot{u }(t)\right) \phi +  \sigma(t) :( \nabla \phi \odot \dot{u}(t)) \Big)\dx \geq 0
\end{equation}
for any $\phi \in C^1_c(\R^n)$ with $\phi \geq 0$.
We deduce that for a.e.\ $t \in (0,T)$ the (nonnegative) distribution $\mu(t)$ uniquely extends to a nonnegative measure, still denoted by $\mu(t)$, compactly supported in $\overline{\Omega}$. Moreover, a density argument and Fubini's Theorem imply that the function $t \mapsto \langle \mu(t), \phi \rangle$ is measurable for any $\phi \in C_c(\R^n)$, so that $\mu \colon t \mapsto \mu(t)$ is weak$^*$-measurable in $\Mb(\R^n)$.
Eventually, since $\mu(t)$ has compact support in $\overline{\Omega}$, we can evaluate its mass by taking $\phi \equiv 1$, which gives
$$\mu(t)(\R^n)= \langle \mu(t), 1 \rangle = - \int_\Omega  \big(\ddot{u}(t) \cdot \dot{u}(t) +\dot{e}(t) \colon \sigma(t) - f(t) \cdot \dot{u}(t)\big) \dx \,.$$
By the regularity properties on $u$, $\sigma$, and $f$, we obtain ${\rm ess\; sup}_{t \in (0,T)} \mu(t)(\R^n) < +\infty$, so that 
\begin{equation}\label{1410192022}
\mu \in L^\infty_{w^*}(0,T; \Mb^+(\R^n))\,.
\end{equation}

\medskip

\paragraph{\it Step 3: Equation of motion.} Choosing $\tau =0$, $z \in \R^n$ and $\varphi \in C^1_c(\Omega \times (0,T))$ with $\varphi\geq 0$  arbitrarily in \eqref{1410191339},  we get that
$$\int_0^T\int_\Omega  \left( \ddot{u} \varphi + \sigma\, \nabla \varphi - f\, \varphi \right) \dx \de t=0,$$
which implies the validity of the equation of motion
$$\ddot{u}-\diver \sigma= f \quad \text{in }L^2(0,T; L^2(\Omega;\R^n))\,,$$
and, in particular, $\sigma \in L^\infty(0,T; H(\mathrm{div}; \Omega))$.
Recalling \eqref{2911181910}, we can define the normal trace $\sigma \nu$ as a element of $L^\infty(0,T; H^{-\frac12}(\partial \Omega;\Rn))$ and, for every $z \in \Rn$ and $\varphi \in C^1_c(\R^n \times [0,T))$
$$z\cdot \int_0^T\int_\Omega  \left( \ddot{u} \, \varphi + \sigma\, \nabla \varphi - f\, \varphi \right) \dx \de t=\int_0^T \langle \sigma \nu, z\,\varphi \rangle_{\dom}  \de t \,. $$
Inserting the above identity in \eqref{1410191339} yields
\begin{multline}\label{1410191524}
- \int_0^T\int_\Omega  \Big( \left( \ddot{u} \cdot \dot{u} +\dot{e} \colon \sigma - f \cdot \dot{u }\right) \varphi +  \sigma :(\nabla \varphi \odot \dot{u}) \Big)\dx \de t + \int_0^T \langle \sigma\nu, z\varphi \rangle_\dom  \de t \\
 + \tau \, \colon \int_0^T \int_\Omega \left( \dot{e} \varphi+ \dot{u}\odot \nabla \varphi  \right)\dx  \de t
 +\frac14\int_0^T\int_\dom   S^{-1}(\tau\nu-Sz)\cdot(\tau\nu-Sz)\varphi\dh \de t\geq 0\,,
\end{multline}
for any $\varphi \in C^1_c(\R^n \times [0,T))$ with $\varphi \geq 0$ and any $(z,\tau) \in \R^n \times \mathbf{K}$. 

\medskip

\paragraph{\it Step 4: Flow rule.} With the choice $z=0$ in \eqref{1410191524}, and recalling the expression \eqref{1410191330} of $\mu(t)$, we have that 
$$-\tau \, \colon \int_0^T \int_\Omega \left( \dot{e} \varphi+ \dot{u}\odot \nabla \varphi  \right)\dx \de t \leq \int_0^T \int_\Omega \varphi \, \de \mu(t)  \de t\,,$$
for any $\tau \in \mathbf{K}$ and any $\varphi \in C^1_c(\Omega \times (0,T))$ with $\varphi\geq 0$.
By \eqref{1410191146} we can recast the above estimate into 
$$ \int_0^T \langle \tau:\dot{p}(t), \varphi(t) \rangle_{H^{-1}(\Omega),H^1_0(\Omega)} \de t \leq \int_0^T \int_\Omega \varphi \, \de \mu(t)  \de t \quad \text{for any } \varphi \in C^1_c(\Omega \times [0,T))\,,\ \varphi\geq 0\,,$$
that may be localized in time, since $\dot{p} \in L^\infty(0,T; H^{-1}(\Omega; \Mnn))$, to get, for a.e.\ $t \in (0,T)$,
\begin{equation}\label{eq:phi1A}
\langle \tau:\dot{p}(t), \phi \rangle_{H^{-1}(\Omega),H^1_0(\Omega)}  \leq \int_\Omega \phi  \de \mu(t) \quad \text{for any } \phi \in C^1_c(\Omega)\,,\ \phi\geq 0.
\end{equation}
This last property implies that, for a.e. $t \in (0,T)$ and all $\tau \in \mathbf K$, the distribution $\mu(t)-\tau:\dot p(t)$ is nonnegative. Thus, since $\mu(t)$ is a measure, we infer that the distribution $\dot p(t)$ uniquely extends to a measure in $\Omega$. 
Moreover, since $p\in W^{1,\infty}(0,T; H^{-1}(\Omega;\Mnn))$, the function $t \mapsto \langle \dot{p}(t), \phi \rangle$ is measurable for all  $\phi \in C^1_c(\Omega)$, and by density this holds true also for all $\phi \in C_c(\Omega)$. Therefore $t\mapsto \dot{p}(t)$ is weak$^*$ measurable in $\mathcal M_b(\Omega;\Mnn)$.
In addition, for any Borel set $A\subset \Omega$, approximating the characteristic function ${\bf 1}_A$ by $C^1_c(\Omega)$ functions, we deduce from \eqref{eq:phi1A} that for a.e.\ $t\in (0,T)$
$$\tau \colon \dot{p}(t)(A) \leq \mu(t)(A) \quad\text{for every Borel set } A\subset \Omega \text{ and every } \ \tau \in \mathbf{K}\,$$
Maximizing with respect to all $\tau \in \mathbf K$ yields by definition of $H$,
$$H\big(\dot{p}(t)(A)\big) \leq \mu(t)(A) \quad\text{for every Borel set } A\subset \Omega.$$
By the definition of a convex function of measures (cf.\ e.g.\ \cite{GS}), this leads to 
 \begin{equation}\label{1410192004}
 H(\dot{p}(t)) \leq \mu(t) \quad \text{in } \Mb^+(\Omega)\,.
 \end{equation}
Moreover, using \eqref{eq:H}, for a.e. $t \in (0,T)$,  we get that $0 \leq r |\dot p(t)| \leq \mu(t)$ in $\mathcal M^+_b(\Omega)$, 
which, together with \eqref{1410192022}, implies that 
\begin{equation}\label{eq:dotp}
\dot p \in L^\infty_{w*}(0,T;\mathcal M_b(\Omega;\Mnn))\,.
\end{equation} 
From this last property, we deduce that
$$ p\in C^{0,1}([0,T]; \Mb(\Omega;\Mnn))\,, \quad u \in C^{0,1}([0,T]; BD(\Omega))\,.$$
Indeed, for all $\phi \in C_c(\Omega)$ the function $t \mapsto \langle p(t), \phi \rangle$ belongs to $W^{1,\infty}(0,T; \Mnn)$ and $\frac{\de}{\de t}\langle p(t), \phi \rangle= \langle \dot{p}(t), \phi \rangle$ for a.e.\ $t \in (0,T)$. By \eqref{eq:dotp} we infer that 
$$|\langle p(t_2)-p(t_1), \phi \rangle |\leq  \sup_{t \in (0,T)} |\dot p(t)|(\Omega) \,\, \|\phi\|_{L^\infty(\Omega;\Mnn)} (t_2-t_1) \quad\text{for all } 0\leq t_1 \leq t_2 \leq T\,,$$
and the previous inequality gives $p\in C^{0,1}([0,T]; \Mb(\Omega;\Mnn))$, while $u \in C^{0,1}([0,T]; BD(\Omega))$ follows from \eqref{1410191146}. 

Thus, for a.e.\ $t\in (0,T)$, we have the additive decomposition $\mathrm{E}\dot{u}(t)=\dot{e}(t) + \dot{p}(t)$ with $\dot{u}(t)\in BD(\Omega) \cap L^2(\Omega;\Rn)$, $\dot{e}(t) \in \Lnn$, $\dot{p}(t) \in \Mb(\Omega;\Mnn)$,  and $\sigma(t) \in \mathcal{K}(\Omega)$. Moreover, using \eqref{1410192004}, we get that $H(\dot{p}(t))$ is a finite measure. Thus we satisfy the conditions that guarantee 
\eqref{eq:conv-ineq} and \eqref{1410191537}, which in turn give $[\sigma(t) \colon \dot{p}(t)] \in \Mb(\Omega)$ for a.e.\ $t \in (0,T)$ and 
$$\mu(t) \mres \Omega = [\sigma(t) \colon \dot{p}(t)] \mres \Omega \leq H(\dot{p}(t)) \quad\text{in }\Mb^+(\Omega) \ \text{ for a.e.\ }t \in (0,T)\,.$$
Together with \eqref{1410192004}, the estimate above gives that
\begin{equation}\label{eq:mut=H}
\mu(t)\mres\Omega= [\sigma(t) \colon \dot{p}(t)] \mres \Omega= H(\dot{p}(t))\quad\text{in }\Mb(\Omega) \ \text{ for a.e.\ }t \in (0,T)\,.
\end{equation}

\medskip

\paragraph{\it Step 5: Integrability property of the normal stress.} Putting the previous information into \eqref{1410191524} and using the integration by parts formula in $BD$ yields
 \begin{equation*}
 \begin{split}
& \int_0^T \int_\Omega \varphi \de  H(\dot p(t))\de t+\int_0^T\int_\dom \varphi \de \mu(t) \,\de t + \int_0^T \langle \sigma\nu, z\varphi \rangle_\dom  \de t 
+\int_0^T \int_\dom (\tau\nu)\cdot \dot u \,  \varphi \dh \de t \\
&-  \int_0^T \langle \tau \, \colon\dot p(t),\varphi(t)\rangle  \de t
 +\frac14\int_0^T\int_\dom   S^{-1}(\tau\nu-Sz)\cdot(\tau\nu-Sz)\varphi\dh \de t\geq 0\,,
 \end{split}
\end{equation*}
for any $\varphi \in C^1_c(\R^n \times (0,T))$ with $\varphi \geq 0$ and any $(z,\tau) \in \R^n \times \mathbf{K}$.
Localizing in time, we get that for a.e. $t \in (0,T)$
 \begin{multline*}
\int_\Omega\phi\de H(\dot p(t))+\int_\dom \phi\de \mu(t)  + \langle \sigma(t)\nu, z\phi \rangle_\dom -  \langle \tau \, \colon\dot p(t),\phi\rangle\\
+\int_\dom (\tau\nu)\cdot \dot u(t)\, \phi \dh +\frac14\int_\dom   S^{-1}(\tau\nu-Sz)\cdot(\tau\nu-Sz)\, \phi\dh\geq 0\,,
\end{multline*}
for any $\phi \in C^1_c(\R^n)$ with $\phi \geq 0$ and any $(z,\tau) \in \R^n \times \mathbf{K}$. As a consequence of the previous inequality, we deduce that, for a.e. $t \in (0,T)$, the normal trace $\sigma(t)\nu$ is a measure in $\mathcal M_b(\partial\Omega;\Rn)$ so that the previous inequality can be localized on $\partial \Omega$. We thus get that for a.e. $t \in (0,T)$,
\begin{multline}\label{eq:mut}
\mu(t) \mres\dom + z\cdot( \sigma(t)\nu) \\
\geq  \left( -(\tau\nu)\cdot \dot u(t) -\frac14   S^{-1}(\tau\nu-Sz)\cdot(\tau\nu-Sz)\right)\hn\mres\dom \quad \text{ in }\mathcal M_b^+(\dom)
\end{multline}
for any $(z,\tau) \in \R^n \times \mathbf{K}$. Using the Besicovitch decomposition Theorem we can split $\sigma(t)\nu$ as $\sigma(t)\nu=\sigma_t \hn \mres \dom + \sigma_t^s$ for some $\sigma_t \in L^1(\dom;\Rn)$ and some singular measure $\sigma_t^s \in \mathcal M_b(\dom;\Rn)$ with respect to $\hn\mres\dom$. Similarly, we can decompose $\mu(t)\mres\dom$ as $\mu(t)\mres \dom=\mu_t^a + \mu_t^s$ for some nonnegative measures $\mu_t^a$ and $\mu_t^s$ which are, respectively, absolutely continuous and singular with respect to $\hn\mres\dom$. According to \eqref{eq:mut}, we have that
$$\mu(t)\mres\dom\geq \mu_t^s \geq - z\cdot \sigma_t^s\quad \text{ in }\mathcal M_b^+(\dom)$$
so that if $A \subset \dom$ is a Borel set such that $\sigma_t^s(A) \neq 0$, then the previous inequality leads to a contradiction by sending $z$ to infinity since $\mu(t)$ is a finite measure. Therefore, the singular part $\sigma_t^s$ of the measure $\sigma(t)\nu$ vanishes, and $\sigma(t)\nu$ is absolutely continuous with respect to $\hn\mres\dom$. Identifying the measure $\sigma(t)\nu$ with its density $\sigma_t$ and going back to \eqref{eq:mut} yields
\begin{multline*}
\mu(t) \mres\dom + z\cdot(\sigma(t)\nu)\hn\mres\dom\\
 \geq  \left( -(\tau\nu)\cdot \dot u(t) -\frac14   S^{-1}(\tau\nu-Sz)\cdot(\tau\nu-Sz)\right)\hn\mres\dom\quad \text{ in }\mathcal M_b^+(\dom).
 \end{multline*}
A classical measure theoreric argument (see e.g.  \cite[Proposition 1.6]{Braides}) implies that one can pass to the supremum in $\tau \in \mathbf K$ under the integral sign, hence
\begin{multline}\label{eq:suptau}
\mu(t) \mres\dom + z\cdot(\sigma(t)\nu)\hn\mres\dom \\
\geq  \left[\sup_{\tau \in \mathbf K}\left( -(\tau\nu)\cdot \dot u(t) -\frac14   S^{-1}(\tau\nu-Sz)\cdot(\tau\nu-Sz)\right)\right]\hn\mres\dom\quad \text{ in }\mathcal M_b^+(\dom).
\end{multline}
Using standard convex analysis arguments, the previous supremum may be expressed, for a.e. $(x,t) \in \partial\Omega \times (0,T)$, as
\begin{multline*}
\sup_{\tau \in \mathbf K} \left( -(\tau\nu(x))\cdot \dot u(x,t) -\frac14   S^{-1}(x)(\tau\nu(x)-S(x)z)\cdot(\tau\nu(x)-S(x)z)\right)\\
= \sup_{q \in -\mathbf K\nu(x)}\left( q\cdot \dot u(x,t) -\frac14   S^{-1}(x)(S(x)z+q)\cdot(S(x)z+q)\right)\\
= \left(\frac14 S^{-1}(x)(S(x)z+\bullet)\cdot(S(x)z+\bullet)+{\rm I}_{-\mathbf K\nu(x)} \right)^*\hspace{-0.6em}(\dot u(x,t))\\
= \left(\frac14 S^{-1}(x)(S(x)z+\bullet)\cdot(S(x)z+\bullet)\right)^* \boxempty ({\rm I}_{-\mathbf K\nu(x)})^*(\dot u(x,t))\,.
\end{multline*}
Recalling that $({\rm I}_{-\mathbf K\nu(x)})^*=h(x,\bullet)$, and computing the convex conjugate
\begin{eqnarray*}
\left(\frac14 S^{-1}(x)(S(x)z+\bullet)\cdot(S(x)z+\bullet)\right)^*(q)&=&\frac12 S(x)q\cdot q + \frac12 S(x)(q-z)\cdot(q-z)-\frac12 S(x)z\cdot z\\
& \geq & \frac12 S(x)q\cdot q -\frac12 S(x)z\cdot z
\end{eqnarray*}
we obtain that
\begin{multline*}
\sup_{\tau \in \mathbf K}\left( -(\tau\nu(x))\cdot \dot u(x,t) -\frac14   S^{-1}(x)(\tau\nu(x)-S(x)z)\cdot(\tau\nu(x)-S(x)z)\right)\\
 \geq  f(x,\bullet) \boxempty h(x,\bullet)(\dot u(x,t)) - \frac12 S(x)z\cdot z= \psi(x,\dot u(x,t))- \frac12 S(x)z\cdot z\,.
 \end{multline*}
Inserting this last inequality into \eqref{eq:suptau}, we infer that for a.e. $t \in (0,T)$
$$\int_\dom \phi \de \mu(t) \geq \int_\dom \phi \left(-z\cdot(\sigma(t)\nu)- \frac12 Sz\cdot z\right)\dh+\int_\dom \phi \, \psi(x,\dot u(t))\dh$$
for all $\phi \in C(\dom)$ with $\phi\geq 0$. Using again \cite[Proposition 1.6]{Braides}, we can pass to the supremum in $z \in \Rn$ under the integral sign which leads to
\begin{equation}\label{eq:ineqmut2}
\int_\dom \phi \de \mu(t) \geq \frac12 \int_\dom \phi S^{-1}(\sigma(t)\nu)\cdot (\sigma(t)\nu)\dh +\int_\dom \phi \, \psi(x,\dot u(t))\dh.
\end{equation}
Employing  the coercivity property \eqref{2711182136} of $S$ and \eqref{1410192022}, we deduce that
$$\sigma\nu \in L^\infty(0,T;L^2(\dom;\Rn)).$$

\medskip

\paragraph{\it Step 6: The boundary condition.} Using \eqref{eq:ineqmut2}, \eqref{1410191330} and the flow rule \eqref{eq:mut=H} yields, for a.e. $t \in (0,T)$,
\begin{multline*}
\frac12 \int_\dom   S^{-1}(\sigma(t)\nu)\cdot (\sigma(t)\nu)\dh +\int_\dom  \psi(x,\dot u(t))\dh\leq  \mu(t)(\dom) 
=\mu(t)(\R^n)-\mu(t)(\Omega)\\=- \int_\Omega   ( \ddot{u}(t) \cdot \dot{u}(t) +\dot{e}(t)\colon \sigma(t) - f(t) \cdot \dot{u }(t)) \dx- \mathcal H(\dot p(t)).
\end{multline*}
Integrating by parts with respect to time we deduce that  for all $0 \leq t \leq T$
\begin{equation*}
\begin{split}
&\frac12 \int_\Omega |\dot u(t)|^2\dx + \mathcal  Q(e(t)) +  \int_{0}^{t} \mathcal H( \dot p(s)) \ds 
+  \int_{0}^{t} \int_\dom  \psi(x, \dot{u})  \dh \ds\\& + \frac{1}{2} \int_{0}^{t} \int_\dom  S^{-1}(\sigma \nu) \cdot (\sigma \nu) \dh \ds   
\leq \frac12  \int_\Omega  |v_0|^2\dx + \mathcal   Q(e_0) + \int_0^t \int_\Omega f \cdot \dot{u}  \dx \ds \,.
\end{split}
\end{equation*}
Using Proposition \ref{prop:IPP} and arguing word for word as in the second part of Subsection \ref{sec:energy-balance}, we obtain the other energy inequality
\begin{equation*}
\begin{split}
&\frac12 \int_\Omega |\dot u(t)|^2\dx + \mathcal   Q(e(t)) +  \int_{0}^{t} \mathcal H( \dot p(s)) \ds +  \int_{0}^{t} \int_\dom  \psi(x, \dot{u})  \dh \ds \\& + \frac{1}{2} \int_{0}^{t} \int_\dom  S^{-1}(\sigma \nu) \cdot (\sigma \nu) \dh \ds   
\geq \frac12  \int_\Omega  |v_0|^2\dx + \mathcal Q(e_0)\dx + \int_0^t \int_\Omega f \cdot \dot{u}  \dx \ds \,.
\end{split}
\end{equation*}

To summarize, we have proved that the triple $(u,e,p)$ satisfies
$$\begin{dcases}
u \in W^{2, \infty}(0,T; L^2(\Omega;\Rn)) \cap C^{0,1}([0,T]; BD(\Omega))\,,\\
e \in W^{1, \infty}(0,T; \Lnn)\,,\\
p \in C^{0,1}([0,T]; \mathcal M_b(\Omega;\Mnn))\,,
\end{dcases}$$
$$\sigma:= \C e \in L^\infty(0,T;\Hdiv),\quad \sigma\nu \in L^\infty(0,T;L^2(\dom;\Rn))$$
and there hold, at the moment, all the conditions in Definition~\ref{def:variational} except condition 4.

\medskip

\paragraph{\it Step 7: Dissipative boundary conditions}
From Remark \ref{rem:unique}, it follows that $(u,e,p)$ is the unique variational solution  the dynamic elasto--plastic model associated to the initial data $(u_0,v_0,e_0,p_0)$ and the source term $f$. In particular the dissipative boundary condition 
$$\sigma\nu+{\rm P}_{-\mathbf K\nu}(S\dot u)=0 \quad \text{ in }L^2(0,T;L^2(\dom;\Rn))$$
is satisfied, and this completes the proof of the Theorem.
\end{proof}

By  Theorems \ref{teo:existence_variational}, \ref{thm:vardiss} and \ref{thm:dissvar}, it turns out that both notions of variational and entropic--dissipative solutions coincide. We thus deduce the following well--posedness result for entropic--dissipative solutions.
\begin{corollary}
Assuming $(H_1)$--$(H_6)$, there exists a unique entropic--dissipative solution $(v,\sigma)$ associated to the initial data $(v_0,\sigma_0)$ and the source term $f$ according to Definition~\ref{def:entropic}. 
\end{corollary}

\section*{Acknowledgements}

V.\ Crismale has been supported by the Marie Sk\l odowska-Curie Standard European Fellowship No 793018, project \emph{BriCoFra}, and acknowledges the support of the LMH through the \emph{Investissement d'avenir} project, reference ANR-11-LABX-0056-LMH, LabEx LMH.


\end{document}